\documentclass[11pt,a4paper]{article}
\usepackage{enumerate}
\usepackage{amsmath}
\usepackage{amsthm}
\usepackage{fix-cm}
\usepackage{pb-diagram}
\usepackage[english]{babel}
\usepackage{amssymb,latexsym}
\usepackage{url}
\usepackage{atbegshi}
\usepackage{enumitem}  
\usepackage{tikz-cd}
\usepackage{authblk}
\usepackage{appendix}
\usepackage{amsfonts}
\usepackage{amsthm}
\usepackage[normalem]{ulem}
\usepackage[dvipsnames]{xcolor}
\usepackage[pdfencoding=unicode,pdftex, bookmarks, colorlinks=true, pdfpagemode=UseOutlines,
        linkcolor=blue, citecolor=green,
        urlcolor=primaryblue4,
        ]{hyperref}
\usepackage{float}
\usepackage{caption}

\usepackage{amsthm}

\definecolor{primaryblue4}{RGB}{70,130,180}
\definecolor{primaryblue2}{RGB}{84, 132, 169}
\definecolor{primaryblue}{RGB}{115, 155, 185}

\newcommand{\thref}[1]{\hyperref[#1]{Theorem~\ref*{#1}}}
\newcommand{\corref}[1]{\hyperref[#1]{Corollary~\ref*{#1}}}
\newcommand{\propref}[1]{\hyperref[#1]{Proposition~\ref*{#1}}}
\newcommand{\secref}[1]{\hyperref[#1]{Section~\ref*{#1}}}
\newcommand{\lemref}[1]{\hyperref[#1]{Lemma~\ref*{#1}}}
\newcommand{\rkref}[1]{\hyperref[#1]{Remark~\ref*{#1}}}
\newcommand{\defref}[1]{\hyperref[#1]{Definition~\ref*{#1}}}

\makeatletter
\renewcommand*{\eqref}[1]{%
  \hyperref[{#1}]{\textup{\tagform@{\ref*{#1}}}}%
}
\makeatletter
\renewcommand{\@cite}[2]{%
  \textcolor{green}{[}
  {#1\if@tempswa \textcolor{green}{, #2}\fi}
  \textcolor{green}{]}
}
\makeatother

\usepackage{comment}

\DeclareMathOperator{\Abs}{{Abs,p}}

\DeclareMathOperator{\Endo}{End}
\DeclareMathOperator{\Hess}{Hess}

\DeclareMathOperator{\tr}{tr}
\DeclareMathOperator{\vol}{Vol}

\DeclareMathOperator{\im}{Im}

\DeclareMathOperator{\w}{\omega}
\DeclareMathOperator{\projection}{proj}
\DeclareMathOperator{\spann}{Span}

\def \bui#1#2{\mathrel{\mathop{\kern 0pt#1}\limits^{#2}}}
\def \buil#1#2{\mathrel{\mathop{\kern 0pt#1}\limits_{#2}}}

\let\<\langle
\let\>\rangle

\newtheorem{example}{Examples}[section]
\newtheorem{thm}{Theorem}[section]
\newtheorem{lemma}[thm]{Lemma}
\newtheorem{prop}[thm]{Proposition}

\newtheorem{cor}[thm]{Corollary}
\newtheorem{remark}[thm]{Remark}
\newtheorem{remarks}[thm]{Remarks}
\newtheorem{definition}[thm]{Definition}
\newtheorem{notation}[thm]{Notation}
\newtheorem{exabout:ample}[thm]{Example}

\begin{document}
\date{\normalsize \today}

\title{Biharmonic Steklov problems with Neumann boundary conditions and spectral inequalities on differential forms}
\author[1,2]{Rodolphe Abou Assali\thanks{\texttt{rodolphe.abou-assali@univ-lorraine.fr, rodolphe.abouassali.1@ul.edu.lb}}}

\affil[1]{{\footnotesize Lebanese University, Faculty of Sciences II, Department of Mathematics, P.O. Box 90656 Fanar-Matn, Lebanon}}
\affil[2]{\footnotesize Universit\'e de Lorraine, CNRS, IECL, F-54000 Nancy, France}

\newcommand{\myAbst}[1]{
    \begin{center}
        \begin{minipage}{0.7\textwidth}
        \textbf{Abstract}. {\small #1}
        \end{minipage}
    \end{center}
}

\maketitle 
\myAbst{%
We introduce a biharmonic Steklov problem with Neumann-type boundary conditions on differential forms and show that it is well-posed. We prove the existence of a discrete spectrum for this problem and provide associated variational characterizations of its eigenvalues. We establish eigenvalue estimates known as Kuttler-Sigillito inequalities, relating the eigenvalues of this problem to those of the Steklov, Dirichlet and Neumann problems, as well as the biharmonic Steklov problem with Dirichlet boundary conditions on differential forms.
}

\section{Introduction}
\subsection{Historical background}
Spectral theory \cite{GirouardLivre, TopicsInSpectralGeometry} is a vast field that intersects with several aspects of mathematics, such as analysis, geometry, and topology. Its applications are closely related to physics. The Dirichlet and Neumann problems are classical examples of analysis on compact Riemannian manifolds with boundary \cite{Chavel, taylor}, which allow one to study the properties of eigenfunctions and eigenvalues. Their applications are directly relevant to the study of vibration modes and the diffusion of sound, heat, and electromagnetic fields in domains with boundaries. In electromagnetism, time-harmonic Maxwell’s equations reduce to Helmholtz-type eigenvalue problems \cite{Maxwell2, Maxwell1}, where the spectrum determines resonant frequencies and modes of propagation. Spectral geometry explores the links between a space's geometry and the spectrum of its natural operators, seeking to understand how the shape of a manifold determines its fundamental frequencies. For instance, given a bounded domain, whether a vibrating drum, a thermal conductor, or an electromagnetic cavity, a natural question arises: how does its geometry influence its vibrations, the evolution of temperature, or the distribution of energy within it \cite{EntendreLaFormeDuTambour1, Kac}? Mathematically, this corresponds to studying the spectrum of a differential operator such as the Laplacian under certain boundary conditions. The Steklov problem is another example of spectral problem, introduced by the Russian mathematician Vladimir Andreevich Steklov in a talk given in 1895, see \cite{SteklovLegacy}.

 Problems where the eigenvalues appear in the boundary conditions are generally called Steklov problems \cite{Kuttler&SigillitoLivre} from their initial appearance in \cite{Steklov}. There are distinct biharmonic Steklov problems depending on the boundary conditions. In this paper we work on biharmonic Steklov operators with Dirichlet and Neumann boundary conditions. The study of these differential operators originates from the theory of elastic oscillating membranes and was mathematically examined on planar domains by J. Kuttler and V. Sigillito \cite{Kuttler1, Kuttler&Sigillito} and L.E. Payne \cite{Payne}. The latter gave some isoperimetric inequalities for the first eigenvalue and eigenvalue estimates were given in \cite{Liu1, Liu2}. The problem also appeared with D. Bucur, A. Ferrero, and F. Gazzola \cite{BucurFerreroGazzola}, where some results about the first Steklov eigenvalue of the biharmonic operator on bounded domains are given. The Robin boundary eigenvalue problem for the bi-Laplacian was introduced by D. Buoso and J. B. Kennedy \cite{BuosoKennedy}. Related eigenvalue problems have also been studied in the literature, see, e.g., \cite{BerchioGazzolaMitidieri, BerchioGazzolaWeth, BucurGazzola, Buoso2, Buoso1, DerbissalyLamberti,GazzolaGrunauSweers, GazzolaSweers, Provenzano}.\\ Throughout this paper, we consider a compact Riemannian manifold $(M^n,g)$ with a smooth boundary $\partial M$ and let $\nu$ be the inward unit normal along the boundary. {We restrict our study to dimension $n \geq 2$, as the spectrum is finite for $n = 1$, see \rkref{rk:spectreNBSFct} below and \cite{FerreroGazzolaWeth}.} We recall the scalar Laplacian for all functions $u$ on a Riemannian manifold which is defined by $\Delta u=-\tr (\Hess(u))$. The bi-Laplacian is defined as the square of the Laplacian. We define the DBS operator denoted by $DBS^{func}$, as follows:  
\begin{align*}
DBS^{func}:L^2(\partial M) &\longrightarrow L^2(\partial M)\\
u &\longmapsto \Delta f_{|\partial M}
\end{align*}
where $f$ is related to $u$ by: $\Delta^2 f = 0$ on $M$, $f = 0$ on $\partial M$, and $\partial_\nu f = u$ on $\partial M$. For the existence of such $f$, we refer for instance to \thref{theo:thm5.4} below. The associated scalar biharmonic Steklov eigenvalue problem with Dirichlet boundary conditions, is given by: 
\begin{equation}\label{eq:IntroDBS}
    \left\{
    \hspace{-0.3em}
    \begin{array}{r@{\quad}c@{\quad}l@{\quad}l}
        \Delta^2 f &=& 0 & \text{on } M \\[0.15em]
        f &=& 0 & \text{on } \partial M \\[0.15em]
        \Delta f - q \partial_\nu f &=& 0 & \text{on } \partial M
    \end{array}
    \right.
\end{equation}where $f$ is a smooth function on $M$ and the real number $q$ is the eigenvalue of $DBS^{func}$. The whole spectrum of this problem was studied in \cite{FerreroGazzolaWeth} where one can also find a physical interpretation of the first eigenvalue and of the Steklov boundary conditions. Its extremal eigenvalues on rectangles with prescribed area have been studied by D. Buoso and P. Freitas \cite{BuosoFreitas}. The isoperimetric properties of the
first eigenvalue of above problem were studied in \cite{Payne}. Similarly, we define the biharmonic Steklov operator with Neumann boundary conditions denoted by $NBS^{func}$, as follows: 
    \begin{align*}
NBS^{func}:L^2(\partial M) &\longrightarrow L^2(\partial M)\nonumber\\
u&\longmapsto -\partial_\nu \Delta f
\end{align*}where $f$ is related to $u$ by: $\Delta^2 f = 0$ on $M$, $\partial_\nu f = 0$ on $\partial M$, and $f = u$ on $\partial M$. For the existence of such $f$ see for instance \thref{theo:thm5.4} below. Then, the associated scalar biharmonic Steklov eigenvalue problem with Neumann boundary conditions, is given by:
\begin{equation}\label{eq:IntroNBS}
    \left\{
    \hspace{-0.3em}
    \begin{array}{r@{\quad}c@{\quad}l@{\quad}l}
        \Delta^2 f &=& 0 & \text{on } M \\[0.15em]
        \partial_\nu f &=& 0 & \text{on } \partial M \\[0.15em]
        \partial_\nu (\Delta f) + \ell f &=& 0 & \text{on } \partial M
    \end{array}
    \right.
\end{equation}
where $f$ is a smooth function on $M$ and the real number $\ell$ is the eigenvalue of $NBS^{func}$. Problem \eqref{eq:IntroNBS} is an extremal case of the bulk-boundary eigenvalue problem for the bi-Laplacian, studied by D. Buoso, C. Falcó, M. del Mar González, and M. Miranda in \cite{BuosoFalcoGonzalezMiranda}. Notably, D. Buoso and L. Provenzano \cite{BuosoProvenzano1, BuosoProvenzano2} defined a related problem that reduces to \eqref{eq:IntroDBS} and \eqref{eq:IntroNBS} in certain cases, whereas P. D. Lamberti and L. Provenzano \cite{LambertiProvenzano2} introduced two biharmonic Steklov problems on Lipschitz domains, depending on a parameter, which generalize these definitions. It is well-known \cite{Kuttler1, Kuttler&Sigillito} that on planar domains, problems \eqref{eq:IntroDBS} and \eqref{eq:IntroNBS} admit a discrete spectrum consisting of real eigenvalues of finite multiplicities. These problems were more recently studied on higher-dimensional Euclidean domains in \cite{FerreroGazzolaWeth}. More recently, A. Hassannezhad and A. Siffert \cite{hassannezhadETsiffert} generalized some of the estimates of \cite{Kuttler1, Kuttler&Sigillito} for both scalar biharmonic Steklov operators on arbitrary Riemannian manifolds and showed that other quantities appear in their estimates that depend on the curvature of the manifold. In the same context, S. Kouzayha established in her PhD thesis \cite{KouzayhaThese} an inequality that links the first eigenvalue of the scalar biharmonic Steklov-Dirichlet problem with those of the Dirichlet and Robin Laplacians.

In a recent work \cite{BSFidaGeorgeOlaNicolas}, the biharmonic Steklov problem with Dirichlet boundary conditions \eqref{eq:IntroDBS} is generalized to differential forms of arbitrary degree, see \eqref{eq:DBSF1}. Appropriate boundary conditions are considered to make the corresponding problem well-posed. The smallest eigenvalue is characterized, and various estimates for this eigenvalue are obtained, some of which involve eigenvalues for other problems such as Dirichlet, Neumann, Robin, and Steklov. Independently, other inequalities concerning the eigenvalues of these latter problems are established.

In this work we extend the biharmonic Steklov problem with Neumann boundary conditions \eqref{eq:IntroNBS} to differential forms of arbitrary degree. This extension is not unique, there are different possible problems, all coinciding with the usual one for functions, which are well-posed and with discrete spectrum. We provide the variational characterization of the eigenvalues for three possible choices of such extensions. This allows us to generalize Kuttler-Sigillito spectral inequalities, which relate eigenvalues of different problems on a bounded planar domain \cite{Kuttler&Sigillito}. Some of these inequalities on functions were also established by A. Hassannezhad and A. Siffert on compact Riemannian manifolds with $C^2$ boundary. 

From now on, we refer to the Biharmonic Steklov problem with Dirichlet boundary conditions as DBS and the Biharmonic Steklov problem with Neumann boundary conditions as NBS. {We note that we always index the eigenvalues of non-negative operators so that $k=1$ corresponds to the first positive eigenvalue.} We fix the degree $p\in \{1,\cdots,n-1\}$, see \lemref{lem:p=nNBS}. 

\subsection{Biharmonic Steklov problems on differential forms}
We have seen above the DBS problem \eqref{eq:IntroDBS}, defined for functions on a compact Riemannian manifold with smooth boundary. This problem has been generalized for differential forms of arbitrary degree by F. El Chami, N. Ginoux, G. Habib, and O. Makhoul in \cite{BSFidaGeorgeOlaNicolas}. Motivated by the extension of Serrin problem \cite{Serrin} to differential forms, they defined the boundary-value problem given below by \eqref{eq:DBSF1}, showed the existence of a discrete spectrum and provided classical and alternative characterizations for the first eigenvalue of this problem. Actually the Serrin problem for functions is the following boundary problem: given $M$ a subset of ${\mathbb{R}}^n$, let $f$ be a solution to the following problem\begin{equation}\label{eq:SerrinFct}
    \left\{
    \hspace{-0.3em}
    \begin{array}{r@{\quad}c@{\quad}l@{\quad}l}
        \Delta f &=& 1 & \text{on } M \\[0.5em]
        f &=& 0 & \text{on } \partial M.
    \end{array}
    \right.
\end{equation}If $\partial_\nu f=c$ for some non-zero real constant $c$, then the domain $M$ must be a ball of radius $nc$ and the function $f$ has the form $\frac{n^2c^2-r^2}{2n}$. It is easy to show that any solution of \eqref{eq:SerrinFct} is an eigenfunction of \eqref{eq:IntroDBS} where $q=\frac{1}{c}$.\\
We recall that a compact Riemannian manifold $M$ with a smooth non-empty boundary $\partial M$ is called a harmonic domain if and only if it admits a solution $f$ to the Serrin boundary value problem \eqref{eq:SerrinFct} for some $c \neq 0$ (see, \cite{RaulotSavo}). In this context, we consider a compact Riemannian manifold $(M^n,g)$ of dimension $n\geq 2$ with smooth boundary $\partial M$. We denote by: $\nu$ the inward unit normal along the boundary $\partial M$, $d$ the exterior differential, $\delta$ the co-differential, $\iota: \partial M \to M$ the inclusion map, and $\lrcorner$ the interior product. We write $\Delta=d\delta+\delta d$ the Hodge-de Rham Laplacian on differential forms (see  \secref{sec:rappelGeometrie}).  \begin{prop}\cite[Proposition 2.1]{BSFidaGeorgeOlaNicolas}
    Let $M$ be a compact manifold with smooth non-empty boundary $\partial M$ and carrying a nontrivial parallel $p$-form $\w_0$.
\begin{enumerate}
    \item[$\bullet$] If $M$ is a harmonic domain, then for the solution $f$ to the Serrin problem on $M$ the $p$-form $\w:= f \cdot \w_0$ is the unique solution to the boundary-value problem
    \begin{equation}\label{eq:SerrinFormes}
    \left\{
    \hspace{-0.3em}
    \begin{array}{r@{\quad}c@{\quad}l@{\quad}l}
        \Delta \w &=& \w_0 & \text{\textit{on }} M \\[0.15em]
        \w_{|\partial M} &=& 0 & \text{\textit{on }} \partial M \\[0.15em]
        \nu \lrcorner d\w &=& c \iota^* \w_0 & \text{\textit{on }} \partial M \\[0.15em]
        \iota^*(\delta \w) &=& -c \nu \lrcorner \w_0 & \text{\textit{on }} \partial M
    \end{array}
    \right.
\end{equation}
    \textit{for some constant $c \in \mathbb{R}$.}
    
    \item[$\bullet$] Conversely, if \eqref{eq:SerrinFormes} has a solution $\w \in \Omega^p(M)$, then assuming without loss of generality that $|\w_0| = 1$ on $M$, we have that $\w = f \cdot \w_0$ where $f$ solves \eqref{eq:SerrinFct}. As a consequence, $M$ must be a harmonic domain.
\end{enumerate}
\end{prop}As noted in \cite[p.195]{BSFidaGeorgeOlaNicolas}, any solution to the Serrin problem \eqref{eq:SerrinFormes} is also a solution to the following DBS problem for $q=\frac{1}{c}$, which is a generalization of \eqref{eq:IntroDBS} to the setting of differential forms and is given by:

\begin{thm}\cite[Theorem 2.2]{BSFidaGeorgeOlaNicolas}\label{thm:DBSTHM}
For all $p\in\{0,\cdots,n\}$, the boundary-value problem \begin{equation}\label{eq:DBSF1}
    (DBS)
    \left\{
    \hspace{-0.3em}
    \begin{array}{r@{\quad}c@{\quad}l@{\quad}l}
        \Delta^2 \w &=& 0 & \text{on } M \\[0.15em]
        \w &=& 0 & \text{on } \partial M \\[0.15em]
        \nu \lrcorner \Delta \w + q \iota^* \delta\w &=& 0 & \text{on } \partial M \\[0.15em]
        \iota^* \Delta \w - q\nu \lrcorner d\w &=& 0 & \text{on } \partial M
    \end{array}
    \right.
\end{equation} on $p$-forms, has a discrete spectrum consisting of an unbounded monotonously non-decreasing sequence of positive eigenvalues of finite multiplicities $(q_{j,p})_{j\geq1}$.\end{thm}
In \cite[Theorem 2.6]{BSFidaGeorgeOlaNicolas}, the authors provided characterizations of the first eigenvalue $q_{1,p}$, which we will use to generalize Kuttler-Sigillito inequalities:
\begin{align}
q_{1,p} &= \inf\left\{\frac{\|\Delta\w\|_{L^2(M)}^2}{\|\nu\lrcorner d\w\|_{L^2(\partial M)}^2+\|\iota^*\delta\w\|_{L^2(\partial M)}^2}\,\Big|\,\w\in\Omega^p(M),\,\w_{|_{\partial M}}=0\ \text{and}\ \nabla_\nu\w\neq 0\right\} \label{eq:q_1Standard} \\[1ex]
&= \inf\left\{\frac{\|\w\|_{L^2(\partial M)}^2}{\|\w\|_{L^2(M)}^2}\,\Big|\,\w\in\Omega^p(M)\setminus\{0\},\,\;\Delta\w=0\textrm{ on }M\right\}. \label{eq:q_1Alternative}
\end{align}
Both infima are indeed minima, where the infimum in \eqref{eq:q_1Standard} is attained by a $q_{1,p}$-eigenform for DBS \eqref{eq:DBSF1} and \eqref{eq:q_1Alternative}
is attained by $\Delta\w$, where $\w$ is a $q_{1,p}$-eigenform for DBS \eqref{eq:DBSF1}.

To establish inequalities between the eigenvalues for differential forms, as it was done in the scalar case, we need to find a suitable NBS problem for differential forms, inspired by the work done in \cite{BSFidaGeorgeOlaNicolas, BucklingFida}. Let us first define the space that is the kernel of several {boundary-value problems} we are interested in.
\begin{definition}
    The absolute de Rham cohomology of degree $p$ is the finite dimensional space \begin{equation}\label{eq:Hap}
        H_{\Abs}(M) = \left\{ \w \in \Omega^p(M) \mid d\w = 0 \text{ on } M, \ \delta\w = 0 \text{ on } M, \ \nu \lrcorner \w = 0 \text{ on } \partial M \right\}.
 \end{equation}We refer to \cite[Section 2]{RaulotSavo} for more details.

\end{definition}
Our first main result is the following:

\begin{thm}\label{thm:TheoremeGeneralNBSF3} Let $(M^n,g)$ be a compact Riemannian manifold with a smooth boundary $\partial M$ and let $\nu$ be the unit inward normal vector to the boundary. Then for all $p\in\{1,\cdots,n-1\}$ the following boundary-value problem: \begin{equation}\label{eq:NBSF3}
    (NBS)
    \left\{
    \hspace{-0.3em}
    \begin{array}{r@{\quad}c@{\quad}l@{\quad}l}
        \Delta^2 \w &=& 0 & \text{on } M \\[0.15em]
        \nu \lrcorner \w &=& 0 & \text{on } \partial M \\[0.15em]
        \nu \lrcorner d\w &=& 0 & \text{on } \partial M \\[0.15em]
        \nu \lrcorner \Delta \w &=& 0 & \text{on } \partial M \\[0.15em]
        \nu \lrcorner d\Delta \w + \ell \iota^* \w &=& 0 & \text{on } \partial M
    \end{array}
    \right.
\end{equation} on $p$-forms, has a discrete spectrum consisting of an unbounded non-decreasing sequence of positive real eigenvalues with finite multiplicities $(\ell_{i,p})_{i\geq0}$. Its kernel is given by $H_{\Abs}(M)$.

\end{thm}
\begin{remark}
For $p=0$, the problem \eqref{eq:NBSF3} reduces to the scalar NBS, whereas for $p=n$, it has only trivial solutions (see \lemref{lem:p=nNBS}).
\end{remark}
We provide below the spectrum of the DBS and NBS operators on an interval. The following remark explains the reason for restricting our study to $n\geq 2$.

\begin{remark}\label{rk:spectreNBSFct}
 The spectrum of NBS for functions $f$ on the interval $M=[0,1]$ reduces to $\{0, 24\}$, with each of these eigenvalues having multiplicity $1$. 
\end{remark}
\begin{proof}
We get the spectrum by simply solving the system \begin{equation*}
    \left\{
    \hspace{-0.3em}
    \begin{array}{r@{\quad}c@{\quad}l@{\quad}l}
        f^{(4)} &=& 0 & \text{on } M \\[0.15em]
        f'(0) &=& 0 & \text{on } \partial M \\[0.15em]
        f'(1) &=& 0 & \text{on } \partial M \\[0.15em]
        -f^{(3)}(0) + \ell f(0) &=& 0 & \text{on } \partial M \\[0.15em]
        f^{(3)}(1) + \ell f(1) &=& 0 & \text{on } \partial M.
    \end{array}
    \right.
\end{equation*}
\end{proof}
Actually, the spectrum of DBS on functions \eqref{eq:IntroDBS} on an interval $[-1,1]$ is also reduced to $\{1,3\}$, see \cite[Remark 1.4]{FerreroGazzolaWeth}.

\subsection{Kuttler-Sigillito inequalities for differential forms}\label{sec:IntroKuttlerSigillito}

The study of spectral geometry often relies on understanding how the shape and connectedness of a manifold dictate its fundamental frequencies. Having established the variational frameworks for the DBS and NBS problems, a natural question arises: how do these specific “mixed"  spectra relate to the classical Dirichlet, Neumann, and Steklov eigenvalues? These relationships are formalized through the Kuttler-Sigillito inequalities. Originally introduced for bounded domains in $\mathbb{R}^2$ \cite{Kuttler&Sigillito}, these inequalities provide a quantitative bridge between different physical models, for instance, relating how a membrane vibrates when its boundary is fixed versus when it is free or concentrates its mass along the boundary. Recently, Hassannezhad and Siffert \cite{hassannezhadETsiffert} extended parts of this theory to the broader context of Riemannian manifolds.
In the following, we recall these foundational results, which we generalize to the more complex setting of differential forms. Similar inequalities have appeared in the literature, see, e.g., \cite{Berge, BuosoChasmanProvenzano, BuosoLuzziniProvenzanoStubbe, Kuttler3, Kuttler&Sigillito2, Kuttler&SigillitoLivre, Verma, Weinstock}.  Let $\lambda_k, \mu_k, \sigma_k, q_k,$ and $\ell_k$ denote the $k^{th}$ eigenvalues of the Dirichlet, Neumann, Steklov, DBS, and NBS problems, respectively (see the problems below).\\
The following inequalities, linking the NBS spectrum to the product of the Neumann and Steklov eigenvalues, were first established for bounded domains in $\mathbb{R}^2$ (see \cite[Table 1]{Kuttler&Sigillito}) in the specific case $k=1$. They were more recently generalized to compact Riemannian manifolds of dimension greater than $2$ with $C^2$ boundary in \cite[Theorem 1.1]{hassannezhadETsiffert} :
\begin{align}
\label{eq:firstineq}
\mu_k \sigma_1 \leq \ell_k \quad \text{and} \quad \mu_1 \sigma_k \leq \ell_k, \quad \text{for all } k \geq 1.
\end{align}%
For bounded domains of the plane with piecewise $C^1$ boundary, these inequalities from \cite[Table 1]{Kuttler&Sigillito} connect the classical Dirichlet, Neumann and Steklov fundamental modes to the DBS and NBS problems:
\begin{equation} \label{eq:ineqKuttlerSigillito}
\begin{aligned}
    \bullet \quad & q_1 \sigma_1^2 \leq \ell_1, \\
    \bullet \quad & \mu_1^{-1} \leq \lambda_1^{-1} + (q_1 \ell_1)^{-\frac{1}{2}}, \\
    \bullet \quad & \mu_1^{-1} \leq \lambda_1^{-1} + (q_1 \sigma_1)^{-1}.
\end{aligned}
\end{equation}
Now that we have the NBS problem, we can establish the Kuttler-Sigillito inequalities in the setting of differential forms. We first state the inequalities and prove them later on (see \secref{sec:kuttler-sigillito}). Let $\lambda_{k,p},\ \mu_{k,p}$ and $\sigma_{k,p}$ denote respectively, the $k^{th}$ positive eigenvalue of the Dirichlet Laplacian \eqref{eq:dirichletFormes},  Neumann Laplacian \eqref{eq:neumannFormesAbsolue} and Steklov operator \eqref{eq:steklovFormes}, on differential forms (see \secref{sec:DNSformes}).  For NBS, we will show the following results on differential forms:

\begin{thm}\label{thm:3èmePBkuttlersigilitto1}
For $p\in\{1,\cdots,n-1\}$ and $k\geq 1$, we have the following inequalities between the eigenvalues of the Steklov, Dirichlet, Neumann, DBS and NBS problems:
    \begin{equation}\label{eq:3èmePBmuksigma1<lk}
    \mu_{k,p} \sigma_{1,p} \leq \ell_{k,p},
    \end{equation}
    \begin{equation}\label{eq:3èmePBmu1sigmak<lk}
    \mu_{1,p} \sigma_{k,p} \leq \ell_{k,p}.
    \end{equation}
    The inequality is strict for $k=1,$ i.e., $ \mu_{1,p} \sigma_{1,p} < \ell_{1,p}.$
    \begin{equation}\label{eq:q1sigma1<l1}
    q_{1,p} \sigma_{1,p}^2 <\ell_{1,p}.
    \end{equation}
    \begin{equation}\label{eq:mu1<lamba1+(q1l1)}
    \mu_{1,p}^{-1} < \lambda_{1,p}^{-1} + (q_{1,p} \ell_{1,p})^{-\frac{1}{2}}.
    \end{equation} 
    \begin{equation}\label{eq:mu1<lamba1+(q1sigma1)}
    \mu_{1,p}^{-1} < \lambda_{1,p}^{-1} + (q_{1,p} \sigma_{1,p})^{-1}.
    \end{equation}
\end{thm}
Recall that inequalities \eqref{eq:3èmePBmuksigma1<lk} and \eqref{eq:3èmePBmu1sigmak<lk} were already established for $k=1$ in a bounded domain of $\mathbb{R}^2$ \cite{Kuttler&Sigillito}, and subsequently generalized to the compact Riemannian manifold setting for functions \cite{hassannezhadETsiffert}. Meanwhile, to our knowledge, inequalities \eqref{eq:q1sigma1<l1}, \eqref{eq:mu1<lamba1+(q1l1)}, and \eqref{eq:mu1<lamba1+(q1sigma1)} have only been proven for a bounded domain in $\mathbb{R}^2$ (see, \cite{Kuttler&Sigillito}) in the large sense and have not been generalized for the scalar NBS on smooth manifolds. Our proof for $p\in\{1, \cdots, n-1\}$ adapts easily for $p=0$, providing this generalization.

\begin{remark}
    Kuttler and Sigillito \cite{Kuttler&Sigillito} established other spectral inequalities that involve geometric controls of the boundary. These were partially extended to scalar functions on manifolds in \cite{hassannezhadETsiffert}. We have extended them to differential forms in \cite{Article2}.
\end{remark}

\textbf{Outline of the Article:}
In \secref{sec:preliminaires} of this article, we provide a review of concepts related to differential forms. We also recall spectral problems involving differential forms, particularly the Dirichlet, Neumann, Steklov, and biharmonic Steklov problems, while presenting their variational characterizations of eigenvalues.
In \secref{sec:NBS} we study the spectra of NBS for differential forms. It relies on the Shapiro-Lopatinskij formalism for elliptic boundary problems, which we present at the beginning of the section. In \secref{sec:VariationalcharacterizationofeigenvaluesNBS}, we show the variational characterization of the eigenvalues of NBS and related problems. Finally, in \secref{sec:kuttler-sigillito}, we establish generalizations of the Kuttler–Sigillito inequalities.

\textbf{Acknowledgments:} I would like to express my sincere gratitude to my PhD supervisors, Nicolas Ginoux, Georges Habib and Samuel Tapie, for their invaluable help, guidance, support and ideas since the beginning of my PhD. I am deeply grateful for the time they have dedicated to me throughout this journey. I would also thank the referees for their valuable comments and suggestions, which helped to improve this article. I acknowledge the IECL and the SAFAR scholarship for funding my PhD and for making this collaboration possible, along with everyone else who contributed to it. Finally, I also thank the support of the Grant ANR-24-CE40-0702 ORBISCAR.

\section{Preliminaries}\label{sec:preliminaires}
In this section, we provide a review of some geometric and analytic tools that we use in the paper.

\subsection{Differential forms}\label{sec:rappelGeometrie}
We give some definitions and results without proofs, some of which are from \cite[Chapter 5]{Lafontaine}, \cite[Chapter 1]{GunterSchwartz} and \cite[Chapter 1, Section 13]{taylor}.

 Let $M$ be a smooth compact manifold of dimension $n\geq2$ with smooth boundary. A differential form of degree $p$ on $M$ is a smooth section of the bundle $\Lambda^pT^*M$. The set of $p$-differential forms on $M$ is denoted $\Omega^p(M)$. For a differential form $\alpha \in \Omega^p(M)$, we denote by $\phi^* \alpha$ its pullback by a diffeomorphism $\phi$, and by $X \lrcorner \alpha$ its interior product by a vector field $X$. We denote by $\iota : \partial M \longrightarrow M$ the inclusion map. Thus, for any $\w \in \Omega^p(M)$, $\iota^*\w$ is the tangential component of $\w_{|\partial M}$, and $\nu \lrcorner \w$ is its normal component. \\ For all $x\in M$ we denote by $\langle.,.\rangle$ the natural inner product on $\Lambda^pT_x^*M$ induced by the Riemannian metric $g_x$, and by $d\mu_g$ the Riemannian volume form. We define then $(\w,\w'):=\int_M\langle\w,\w'\rangle d\mu_g.$
    The co-differential operator, denoted $\delta$, is a map $\delta : \Omega^p(M) \longrightarrow \Omega^{p-1}(M)$. It is the formal $L^2$-adjoint of the exterior derivative $d$, i.e., for all $\alpha\in \Omega^p(M)$ and $\beta\in \Omega^{p+1}(M)$ with support in $M\setminus \partial M$, we have  $(d\alpha, \beta) = (\alpha, \delta\beta).$\\
 For a vector field $X$, we write $X^\flat$ for the differential $1$-form such that for all vector field $Y$ we have $X^\flat(Y)=g(X,Y).$ We denote by $\nabla$ the Levi-Civita connection. Given a local orthonormal basis  $(e_i)_{i=1}^n$, we have the local expressions for $d\w=\sum_i e_i^\flat\wedge\nabla_{e_i}\w$ and $\delta\w=-\sum_ie_i\lrcorner\nabla_{e_i}\w$, see for instance \cite[Proposition 2.61]{GallotHulinLafontaine}.\\ The Hodge-de Rham Laplacian, the operator denoted $\Delta$ from  $\Omega^p(M)$ to itself defined for all $\w\in \Omega^p(M)$ by $\Delta \w = d\delta\w + \delta d\w,$ is formally self-adjoint, i.e., for all $\alpha, \ \beta\in \Omega^p(M)$ with support in $M\setminus \partial M$, we have  $(\Delta\alpha, \beta) = (\alpha, \Delta\beta).$ A $p$-differential form $\w$ is said to be harmonic if $\Delta \w = 0.$\vspace*{\baselineskip}
  
We now present the following integration by parts formulas that we frequently use:
\begin{prop}\label{prop:IPP}
    For $\w \in \Omega^p(M)$ and $\w' \in \Omega^{p+1}(M)$, we have 
    \begin{equation}\label{eq:ippPourLaPreuve}
        \int_M \langle d\w, \w' \rangle d\mu_g = \int_M \langle \w, \delta\w' \rangle d\mu_g -\int_{\partial M} \langle \iota^* \w, \nu \lrcorner \w' \rangle d\mu_g.
    \end{equation}

    Moreover, for any $\w, \w' \in \Omega^p(M)$, we have 
    \begin{equation}\label{eq:IPP1}
    \begin{split}
        & \int_M \Big( \langle \Delta \w , \w' \rangle -  \langle \w , \Delta \w' \rangle \Big) d\mu_g \\
        = & \int_{\partial M} \Big( \langle \nu \lrcorner d \w, \iota^* \w' \rangle - \langle \iota^* \w, \nu \lrcorner d \w' \rangle + \langle \nu \lrcorner \w, \iota^* \delta \w' \rangle - \langle \iota^* \delta \w, \nu \lrcorner \w' \rangle \Big) d\mu_g.
    \end{split}
    \end{equation}
   
Eventually, for all $\w,\w'\in \Omega^p(M)$ we have the following identity
    \begin{equation}\label{eq:IPP3'}
    \begin{aligned}
    \int_M \langle \Delta \w, \w' \rangle \, d\mu_g &= \int_M \langle d\w, d\w' \rangle \, d\mu_g + \int_M \langle \delta \w, \delta \w' \rangle \, d\mu_g \\
    &\quad + \int_{\partial M} \langle \nu \lrcorner d\w, \iota^* \w' \rangle \, d\mu_g - \int_{\partial M} \langle \iota^* \delta \w, \nu \lrcorner \w' \rangle \, d\mu_g.
    \end{aligned}
    \end{equation}

\end{prop}
\begin{proof}
    Equations \eqref{eq:ippPourLaPreuve} and \eqref{eq:IPP1} are shown in \cite[p.182]{taylor}. Equation \eqref{eq:IPP3'} is shown in \cite[Equation $30$]{BesselFidaGeorgeNicolas} and \cite[p.421]{taylor}. \end{proof}
  
    {\begin{cor}
For any $\omega, \omega' \in \Omega^p(M)$, we have:\\
First,\begin{align}
\int_M \langle \Delta^2 \omega , \omega' \rangle \, d\mu_g
&= \int_M \langle \Delta \omega , \Delta \omega' \rangle \, d\mu_g \notag \\
&\quad + \int_{\partial M} \Big[
    \langle \nu \lrcorner d \Delta \omega, \iota^* \omega' \rangle
    - \langle \iota^* \Delta \omega, \nu \lrcorner d \omega' \rangle \notag \\
&\qquad\qquad
    + \langle \nu \lrcorner \Delta \omega, \iota^* \delta \omega' \rangle
    - \langle \iota^* \delta \Delta \omega, \nu \lrcorner \omega' \rangle
\Big] \, d\mu_g.
\label{eq:IPP2}
\end{align}
Second, \begin{align}
\int_M \langle \Delta \omega, \omega \rangle \, d\mu_g
&= \int_M \left( |d\omega|^2 + |\delta \omega|^2 \right) d\mu_g \notag \\
&\quad + \int_{\partial M} \Big[
    \langle \nu \lrcorner d\omega, \iota^* \omega \rangle
    - \langle \iota^* \delta \omega , \nu \lrcorner \omega \rangle
\Big] \, d\mu_g.
\label{eq:IPP3}
\end{align}
Finally, \begin{align}
\int_M \langle \Delta^2 \omega, \omega' \rangle \, d\mu_g
&- \int_M \langle \omega, \Delta^2 \omega' \rangle \, d\mu_g \notag \\
&= \int_{\partial M} \Big[
    \langle \nu \lrcorner d \Delta \omega, \iota^* \omega' \rangle
    - \langle \iota^* \omega, \nu \lrcorner d \Delta \omega' \rangle \notag \\
&\qquad
    - \langle \iota^* \delta \Delta \omega, \nu \lrcorner \omega' \rangle
    + \langle \nu \lrcorner \omega, \iota^* \delta \Delta \omega' \rangle \notag \\
&\qquad
    + \langle \nu \lrcorner \Delta \omega, \iota^* \delta \omega' \rangle
    - \langle \iota^* \delta \omega, \nu \lrcorner \Delta \omega' \rangle \notag \\
&\qquad
    - \langle \iota^* \Delta \omega, \nu \lrcorner d \omega' \rangle
    + \langle \nu \lrcorner d \omega, \iota^* \Delta \omega' \rangle
\Big] \, d\mu_g.
\label{eq:IPP4}
\end{align}
\end{cor}
   \begin{proof}
        Replacing $\w$ with $\Delta \w$ in \eqref{eq:IPP1}, we obtain \eqref{eq:IPP2}.
    In addition, for $\w = \w'$ in \eqref{eq:IPP3'} we obtain \eqref{eq:IPP3}. Moreover, applying \eqref{eq:IPP2} to $\int_M \langle \Delta^2 \omega, \omega' \rangle \, d\mu_g - \int_M \langle \omega, \Delta^2 \omega' \rangle \, d\mu_g $ gives us \eqref{eq:IPP4}.
    \end{proof}}
Note that in \cite{taylor}, the identities involve $ \w $ on the boundary rather than  $\iota^*\w$, as in the above Proposition. This is equivalent since for all $\alpha\in\Omega^p(M),\ \beta\in\Omega^{p+1}(M)$ we have $\langle\alpha,\nu\lrcorner\beta\rangle=\langle\iota^*\alpha,\nu\lrcorner\beta\rangle$ on $\partial M$.

\subsection{Variational characterizations of classical eigenvalue problems on differential forms}\label{sec:DNSformes}
In this section, we review the classical problems on differential forms and the variational characterizations of their eigenvalues, which are derived from the integration by parts formulas in the previous section and the min-max principle (see \cite[Chapter 1]{Chavel}), as we use them to establish the Kuttler-Sigillito inequalities. The proofs primarily follow the same lines as those used for the scalar case and we omit them here. In this context, we let $(M^n,g)$ be a compact connected Riemannian manifold with boundary $\partial M$. We still denote by $\nabla$ the Levi-Civita connection on the tangent bundle. Let $d\mu_g$ be the Riemannian measure and let $E\longrightarrow M$ be a smooth vector bundle equipped with a vector bundle metric $\langle \cdot,\cdot\rangle_E$. We write $\Gamma(E)$ for smooth sections of $E.$ For $ 1\leq k < \infty$ and $ m \in \mathbb{N} $. We use the following Sobolev spaces: 
\begin{itemize}
\item $ L^k(M;E) = \mathcal{L}^k(M;E) / \sim $ where: 
\begin{itemize}
    \item $ \mathcal{L}^k(M;E) = \left\{ \omega : M \to E \mid \omega \text{ is measurable and } \int_M |\omega|_E^k \, d\mu_g < \infty \right\} $
    \item $\sim$ is the equivalence relation such that $ \omega_1 \sim \omega_2 \iff \omega_1 = \omega_2 \text{ a.e. on } M $
\end{itemize} 
    \item $ W^{m,k}(M;E) = \left\{ \w \in L^k(M) \ \middle| \ \forall \ 1\leq j\leq m, \nabla^j\w \in L^k(M;E^{(j)}) \right\},$\\ where $E^{(j)}=(TM)^{\otimes j}\otimes E$
    \item $ H^m(M;E) = W^{m,2}(M;E)$ 
  \item $H^1_0(M;E) = \overline{\left\{ \w \in \Gamma(E) \ \middle| \ \ \w_{|\partial M} = 0 \right\}}^{H^1(M)} $
    \item $ H^2_0(M;E) = \overline{\left\{ \w \in \Gamma(E) \ \middle| \ \w_{|\partial M} = 0 \ \text{and} \ (\nabla \w)_{|\partial M} = 0 \right\}}^{H^2(M)}.$
\end{itemize}  For the definition of $H^s(M;E)$ when $s\in\mathbb{R}^+\setminus\mathbb{Z}$, see \cite[Chapter 4 - Section 4]{taylor}. For the definitions in $\mathbb{R}^n$, see \cite[Definition 27]{AliMichael}, \cite[Chapter 1 - Section 7]{LionsMagenes}, \cite[Section 3]{Hitchhiker} and \cite[Chapter 4]{taylor}. For more general definitions of Sobolev spaces on compact manifolds, see \cite[Section 9]{AliMichael}. For the definition on Lipschitz domains, see \cite[Section 2]{
TraceLipschitz2}. For $s>0,$ $H^{-s}(M;E)$ is the dual of $H^s(M;E)$ in $L^2(M;E).$\\ In what follows, we will mostly omit the vector bundle $E$ from the notation, writing $L^k(M)$ instead of $L^k(M;E)$, and similarly for other spaces.\vspace*{\baselineskip}

We now recall the trace Theorem, which ensures that the last two Sobolev spaces $H^1_0(M;E)$ and $H_0^2(M;E)$ mentioned above coincide respectively with $ H^1_0(M;E) = \left\{ \w \in H^1(M;E) \ \middle| \ \ \w_{|\partial M} = 0 \right\} $ and $ H^2_0(M;E) = \left\{ \w \in H^2(M;E) \ \middle| \ \w_{|\partial M} = 0 \ \text{and} \ (\nabla \w)_{|\partial M} = 0 \right\}$. It is established for open sets in $\mathbb{R}^n$ 
in \cite[Section 5.5]{Evans}, with generalizations on Euclidean domains in \cite[Sections 2 - 3]{LambertiProvenzano}, on Lipschitz domains in \cite[Section 3]{
TraceLipschitz2} and \cite[Section 2]{TraceLipschitz1}, and on compact Riemannian manifolds with smooth boundaries 
in \cite[Chapter 1 Section 8, p.189 Theorem 6.5]{LionsMagenes}, \cite[Chapter 4,  Section 4, Proposition 4.5]{taylor}, \cite{TaylorSeeleyNote}.  
\begin{thm}\label{th:trace}
Let $(M, g)$ be a compact Riemannian manifold with smooth boundary $\partial M$. For any $m>\frac{1}{2}$ and $1 < k < \infty$, there exists a unique bounded linear operator 
\[
T \colon W^{m,k}(M) \to W^{m-\frac{1}{k},k}(\partial M)
\]such that $Tu = u_{|\partial M}$ for all $u \in C^\infty(M)$.
\end{thm} This operator allows for the following subspace to be well-defined:
\[
W_0^{m,k}(M) = \{ u \in W^{m,k}(M) \mid Tu = 0 \text{ on } \partial M \}
\]\label{espace:wk0}which coincides with the closure of $C_c^\infty(M)$ in the $W^{m,k}(M)$ norm. Note that for $m=1,2$, the space $W_0^{m,2}(M)$ is identical to $H_0^m(M)$ defined above. \vspace*{\baselineskip} 

The Dirichlet problem for differential forms is defined as follows:
\begin{equation}\label{eq:dirichletFormes}
    \left\{
    \hspace{-0.3em}
    \begin{array}{r@{\quad}c@{\quad}l@{\quad}l}
        \Delta \w &=& \lambda \w & \text{on } M \\[0.5em]
        \w &=& 0 & \text{on } \partial M
    \end{array}
    \right.
\end{equation}
where $\lambda \in \mathbb{R}$. For $p\in\{0,\cdots,n\}$, the kernel of the Dirichlet problem is trivial (see, \cite[Chapter 5]{taylor}). The spectrum of the Laplacian of Dirichlet on forms is discrete, consisting of eigenvalues 
\[0<\lambda_{1,p} \leq \lambda_{2,p} \leq \ldots\]
The $k\textsuperscript{th}$ eigenvalue $\lambda_{k,p}$ \cite{BesselFidaGeorgeNicolas,GueriniThese} is given by \begin{align}
\lambda_{k,p}&=\underset{\substack{V\subset{H}^1_0(M)\\ \dim(V)=k}}{\min}\ \underset{\substack{0\neq \w \in V }}{\max} \frac{\lVert d\w \rVert^2_{L^2(M)} + \lVert \delta \w \rVert^2_{L^2(M)}}{\lVert \w \rVert^2_{L^2( M)}}\label{eq:vpKDirichletForme} \\ 
&=\underset{\substack{V\subset H_0^1(M)\cap H^2(M)\\ \dim(V)=k}}{\min}\ \underset{\substack{0\neq \nabla\w \in V }}{\max} \frac{\lVert \Delta\w \rVert^2_{L^2( M)}}{\lVert d\w \rVert^2_{L^2(M)} + \lVert \delta \w \rVert^2_{L^2(M)}}.\label{eq:vpKAlternativeDirichletForme}
\end{align}
The Neumann problem for differential forms (with absolute boundary conditions) is defined as follows:
\begin{equation}\label{eq:neumannFormesAbsolue}
    \left\{
    \hspace{-0.3em}
    \begin{array}{r@{\quad}c@{\quad}l@{\quad}l}
        \Delta \w &=& \mu \w & \text{on } M \\[0.5em]
        \nu \lrcorner \w &=& 0 & \text{on } \partial M \\[0.5em]
        \nu \lrcorner d\w &=& 0 & \text{on } \partial M
    \end{array}
    \right.
\end{equation}
where $\mu\in \mathbb{R}$. The spectrum of the Laplacian of Neumann on forms is discrete, consisting of eigenvalues 
\[0<\mu_{1,p} \leq \mu_{2,p} \leq \ldots \]Note that $\mu_{0,p}=0$ is an eigenvalue if and only if $H_{\Abs}(M)\neq0.$ For $p\in\{0,\cdots,n\}$, using the integration by parts formula \eqref{eq:IPP3}, we deduce that the kernel of the Neumann problems is $H_{\Abs}(M)$ \eqref{eq:Hap}. We define the following spaces  \begin{align}
    H_N^1(M) &:= \left\{ \w \in H^1(M) \mid \nu \lrcorner \w = 0 \text{ on } \partial M \right\} \ \text{and}\ \nonumber \\
    H^2_{N}(M) &:= \left\{ \w \in H^2(M) \mid \nu \lrcorner \w = 0 \text{ and } \nu \lrcorner d\w = 0 \text{ on } \partial M \right\}. \label{eq:hn3}
\end{align}
 The $k\textsuperscript{th}$ non-zero eigenvalue $\mu_{k,p}$ \cite{NeumannProblemConner,BesselFidaGeorgeNicolas,GueriniThese} is given by 
   \begin{align}
      \mu_{k,p}&= \underset{\substack{ V\subset H_N^1(M)\\ \dim(V)=k+\dim H_{\Abs}(M)}}{\min}\ \underset{\substack{0\neq \w \in V }}{\max} \frac{\lVert d\w \rVert^2_{L^2(M)} + \lVert \delta \w \rVert^2_{L^2(M)}}{\lVert \w \rVert^2_{L^2( M)}}\label{eq:vp1AlternativeNeumannAbsolueForme1} \\ 
      &= \underset{\substack{V\subset H^2_{N}(M)/ H_{\Abs}(M)\\ \dim(V)=k}}{\min} \ \  \underset{\substack{0\neq \w \in V/ H_{\Abs}(M)}}{\max} \ \frac{\lVert \Delta\w \rVert^2_{L^2( M)}}{\lVert d\w \rVert^2_{L^2(M)} + \lVert \delta \w \rVert^2_{L^2(M)}}. \label{eq:vp1AlternativeNeumannAbsolueForme}
   \end{align} 
The Steklov (Dirichlet-to-Neumann) problem for differential forms is defined as follows: \cite{RaulotSavo}
\begin{equation}\label{eq:steklovFormes}
    \left\{
    \hspace{-0.3em}
    \begin{array}{r@{\quad}c@{\quad}l@{\quad}l}
        \Delta {\w} &=& 0 & \text{on } M \\[0.15em]
        \nu \lrcorner {\w} &=& 0 & \text{on } \partial M \\[0.15em]
        \nu \lrcorner d\w &=& -\sigma \iota^* {\w} & \text{on } \partial M
    \end{array}
    \right.
\end{equation}where $\sigma\in \mathbb{R}$. The spectrum of the Steklov operator on forms is discrete, consisting of eigenvalues 
\[0 < \sigma_{1,p} \leq \sigma_{2,p} \leq \ldots 
\]Note that $\sigma_{0,p}=0$ is an eigenvalue if and only if $H_{\Abs}(M)\neq0.$  For $p\in\{0,\cdots,n\}$, the kernel of the Steklov problem is $H_{\Abs}(M)$ \eqref{eq:Hap}. We consider the following spaces: \begin{align*}
    \check{H}_N^1(M) &:= \left\{\w\in H^1_N(M) \mid \w\perp_{L^2(\partial M)}H_{\Abs}(M)\right\}\\[0.2cm]
    \mathcal{H}_N(M)      &:= \left\{\w\in \Omega^p(M) \mid \Delta\w=0\ \text{on}\ M\ \text{and}\ \nu\lrcorner\w=0\ \text{on}\ \partial M \right\}.
\end{align*} The $k$\textsuperscript{th} non-zero eigenvalue $\sigma_{k,p}$ \cite{Karpukhin, RaulotSavo} is given by
   \begin{align}
\sigma_{k,p}&=\underset{\substack{ V\subset\check{H}_N^1(M)\\\dim(V)=k}}{\min}\ \underset{\substack{0\neq \w \in V }}{\max} \frac{\lVert d\w \rVert^2_{L^2(M)} + \lVert \delta \w \rVert^2_{L^2(M)}}{\lVert \w \rVert^2_{L^2(\partial M)}}\label{eq:sigmakformesSteklov}\\  
&=\underset{\substack{V\subset\mathcal{H}_N(M) \\ \dim(V)=k}}{\min}\ \  \underset{\substack{0\neq \w \in V }} {\max} \frac{\lVert \nu\lrcorner d\w \rVert^2_{L^2(\partial M)}}{\lVert d\w \rVert^2_{L^2(M)} + \lVert \delta \w \rVert^2_{L^2(M)}}\label{eq:sigmakformesSteklov2}.
\end{align}

\section{Spectral properties of NBS}\label{sec:NBS}
In the following, we show that the spectrum associated with problem \eqref{eq:NBSF3} is discrete and consists entirely of eigenvalues of finite multiplicity. The proof is largely based on the methodology presented in \cite{BSFidaGeorgeOlaNicolas, FerreroGazzolaWeth}.
The main result of this section is the following:

\begin{thm}\label{thm:TheoremeGeneralNBSF3PourLaPreuve} Let $(M^n,g)$ be a compact Riemannian manifold with a smooth boundary $\partial M$ and let $\nu$ be the unit inward normal vector to the boundary. Then, for $p\in\{1,\cdots,n-1\}$ the following boundary-value problem: \begin{equation}\label{eq:NBSF3preuve}
    (NBS)
    \left\{
    \hspace{-0.3em}
    \begin{array}{r@{\quad}c@{\quad}l@{\quad}l}
        \Delta^2 \w &=& 0 & \text{on } M \\[0.15em]
        \nu \lrcorner \w &=& 0 & \text{on } \partial M \\[0.15em]
        \nu \lrcorner d\w &=& 0 & \text{on } \partial M \\[0.15em]
        \nu \lrcorner \Delta \w &=& 0 & \text{on } \partial M \\[0.15em]
        \nu \lrcorner d\Delta \w + \ell \iota^* \w &=& 0 & \text{on } \partial M
    \end{array}
    \right.
\end{equation} on $p$-forms, has a discrete spectrum consisting of an unbounded non-decreasing sequence of positive real eigenvalues with finite multiplicities $(\ell_{i,p})_{i\geq0}$. Its kernel is given by $H_{\Abs}(M)$ \eqref{eq:Hap}.
\end{thm}

\subsection{Elliptic boundary problems and smooth solutions}\label{sec:EllipticiteConvFaible}

We recall the definition of the notion of ellipticity in the sense of Shapiro-Lopatinskij for boundary conditions imposed on linear and elliptic differential operators, as presented in \cite[Paragraph 1.6]{GunterSchwartz}. The study of proper boundary conditions for an elliptic differential operator (assumed to be of order one) acting on smooth sections of a Hermitian vector bundle $E \rightarrow M$ was undertaken in the fifties by Z. Y. Shapiro and Y. B. Lopatinskij. This method is crucial in showing that solutions to the boundary value problem exist and are smooth and it has been widely used in the literature (see, e.g., \cite{BesselFidaGeorgeNicolas, BSFidaGeorgeOlaNicolas, BucklingFida, FidaOla, Raulot2, raulot}).

Let us consider the following {boundary-value problem}:
\begin{equation}\label{eq:PbPourLaDefDeShapiro}
    \left\{
    \hspace{-0.3em}
    \begin{array}{r@{\quad}c@{\quad}l@{\quad}l}
        A \w &=& 0 & \text{on } M \\[0.15em]
        B_1\w &=& 0 & \text{on } \partial M \\[0.15em]
        B_2\w &=& 0 & \text{on } \partial M \\[0.15em]
        \qquad \vdots &&& \quad\vdots \\[0.15em]
        B_k\w &=& 0 & \text{on } \partial M
    \end{array}
    \right.
\end{equation}
where $A:\Gamma(E)\longrightarrow\Gamma(E)$ is a linear differential operator of order $m$, and for all $1\leq j\leq k$ $B_j:\Gamma(E)\longrightarrow\Gamma(E_j)$ is a linear differential operator. Here $E\longrightarrow M$ and $(E_j\longrightarrow \partial M)_{1\leq j\leq k}$ are vector bundles.

The principal symbol \begin{eqnarray*}
 \sigma_A:T^*_xM &\longrightarrow & \Endo (E_x)\\
 \xi& \longmapsto &  \sigma_A(\xi)  \notag 
    \end{eqnarray*} has locally a polynomial homogeneous expression of order $k$ in the coordinates of $\xi.$ Therefore, it extends to generalized covectors with values in the algebra of differential operators on $C^\infty(\mathbb{R},E_x\otimes\mathbb{C})$, i.e., $\tilde{\xi}\in T^*_xM\otimes(\mathbb{C}\oplus\mathbb{C}\cdot \partial_t),$ where the multiplication in $\mathbb{C}\oplus\mathbb{C}\cdot \partial_t$ is given by $(a_1+b_1\partial_t)(a_2+b_2\partial_t)=a_1a_2+(a_1b_2+a_2b_1)\partial_t+b_1b_2\partial_t^2.$

    It is known by \cite[p.182]{taylor} that the Hodge-de Rham Laplacian and bi-Laplacian are elliptic, and their principal symbols are given by: 
    \[\sigma_\Delta(\xi) = -|\xi|^2 \qquad \text{and} \qquad \sigma_{\Delta^2}(\xi) =|\xi|^4,\]
    where $\xi \in T_x^*M$. 

Here is the Definition of the ellipticity in the sense of Shapiro-Lopatinskij originally formulated in \cite{Lopatinskij,Shapiro} and reformulated in \cite{Hormander, GunterSchwartz}.  \begin{definition}\cite[Definition 20.1.1]{Hormander}\cite[Definition 1.6.1]{GunterSchwartz}\label{def:ellipticiteShapiro}
The problem \eqref{eq:PbPourLaDefDeShapiro} is said to be \emph{elliptic} in the sense of Shapiro-Lopatinskij if and only if:
\begin{itemize}
    \item[$\bullet$] $A$ is elliptic, i.e., $\sigma_{A}(\xi)$ is invertible for all $\xi \in T^*M \backslash \{0\}$. 
    \item[$\bullet$] For all $v \in T_x^*\partial M \backslash \{0\}$, the map $\Phi$
    \begin{align*}
    \Phi : \mathcal{M}^+_{A,v} &\longrightarrow \bigoplus_{j=1}^{k} E_j \\
    y &\longmapsto \left( \sigma_{B_1}(-iv+\partial_t\nu^\flat)y, \sigma_{B_2}(-iv+\partial_t\nu^\flat)y,\dots, \sigma_{B_k}(-iv+\partial_t\nu^\flat)y \right)(0)
    \end{align*}
    is an isomorphism,
    where $-iv + \partial_t\nu^\flat\in T^*_xM\otimes(\mathbb{C}\oplus\mathbb{C}\cdot \partial_t)$ for $x\in\partial M$ and
    \begin{equation}\label{eq:Mv+}
    \mathcal{M}^+_{A,v} := \{\text{bounded solutions } y=y(t) \text{ on } \mathbb{R}_+ \text{ of the ODE } \sigma_{A}(-iv+\partial_t\nu^\flat)y=0\}.
    \end{equation}
    
\end{itemize}
\end{definition}
A direct computation shows that in the case of the bi-Laplacian, the space $\mathcal{M}^+_{\Delta^2,v}$ is given by:
\[
\mathcal{M}^+_{\Delta^2,v} = \spann\left( e^{-|v| t}(at+b) \cdot \w_0 \mid a,b \in \mathbb{R},\ \w_0 \in \Lambda^p T^*_x M_{|\partial M} \right)
\]
and its dimension is $2 \binom{n}{p}$.

We now present those results from \cite[Chapter 2]{LionsMagenes} that are useful for our work. We adapt the foundational results from \cite{LionsMagenes, GunterSchwartz} to our specific framework, which provides the necessary basis for establishing the existence of solutions in our setting.\\ The following is an interpretation of \cite[Chapter 2 - Theorem 2.1]{LionsMagenes} in our context.
On $\Gamma(E),$ we say that the problems \[
\begin{array}{c@{\hspace{1cm}}c}
    \left\{
    \hspace{-0.3em}
    \begin{array}{r@{\quad}c@{\quad}l@{\quad}l}
        Au &=& 0 & \text{on } M \\[0.15em]
        B_1u &=& 0 & \text{on } \partial M \\[0.15em]
        \qquad \vdots &&& \quad\vdots \\[0.15em]
        B_ku &=& 0 & \text{on } \partial M
    \end{array}
    \right.
    & \text{and} \qquad \quad
    \left\{
    \hspace{-0.3em}
    \begin{array}{r@{\quad}c@{\quad}l@{\quad}l}
        A^*u &=& 0 & \text{on } M \\[0.15em]
        C_1u &=& 0 & \text{on } \partial M \\[0.15em]
        \qquad \vdots &&& \quad\vdots \\[0.15em]
        C_ku &=& 0 & \text{on } \partial M
    \end{array}
    \right.
\end{array}
\] are dual, if 
\begin{enumerate}
    \item[$\bullet$] $A$ is a linear elliptic differential operator of order $m$ and $A^*$ is the \emph{formal adjoint} of $A,$ i.e., for all $u,v\in \Gamma(E)$ with compact support in $M\setminus \partial M$, $\int_M\langle Au,v\rangle d\mu_g=\int_M\langle u,A^*v\rangle d\mu_g.$
    \item[$\bullet$]  There exist operators $(T_1,\cdots,T_k)$, $(S_1,\cdots,S_k)$ such that $\forall i \in \{1,\cdots,k\}$, $\deg(T_i)=m-1-\deg(B_i)$ and $\deg(S_i)=m-1-\deg(C_i)$ and for all $u,v\in\Gamma(E),$ we have  \begin{equation} \label{eq:greenLionsMagenes}
    \int_{M} \langle Au,{v}\rangle \, d\mu_g - \int_{M} \langle u,{A^*v}\rangle \, d\mu_g  = \sum_{i=1}^{k} \int_{\partial M} \langle S_i u, {C_i v}\rangle \, d\mu_g - \sum_{i=1}^{k} \int_{\partial M} \langle B_i u, {T_i v} \rangle \, d\mu_g.
    \end{equation}
\end{enumerate}
  In this paper, we take $A=\Delta$ or $A=\Delta^2$, both are formally self adjoint: $A=A^*$. According to $A$ and the boundary conditions which we consider, the operators $B_j,\ C_j,\ S_j,$ and $T_j$ vary in each situation due to the integration by parts formulas presented in \propref{prop:IPP}.\\ Assuming that the problems are dual, let $\mathcal{N}$ and $\mathcal{N}^*$ be the spaces defined by:  $\mathcal{N} := \{u \in \Gamma(E) \mid Au = 0, B_1u = 0, \dots, B_{k}u = 0\}$ and $\mathcal{N}^* := \{u \in \Gamma(E) \mid A^*u = 0, C_1u = 0, \dots, C_{k}u = 0\}$.

The following result will be used throughout the paper:
\begin{thm}\cite[Chapter 2 - Theorems 5.3 and 5.4]{LionsMagenes}\label{theo:thm5.4}
   Under the following assumptions:
    \begin{enumerate}[label=\roman*.]
        \item[$\bullet$] $M$ is a compact manifold with smooth boundary $\partial M$.
        \item[$\bullet$]  $A : \Gamma(E) \to \Gamma(E)$ is a differential operator of order $m$.
        \item[$\bullet$] The problem \eqref{eq:PbPourLaDefDeShapiro} is elliptic in the sense of Shapiro-Lopatinskij,

    \end{enumerate}
  Then for any $s \geq m$, the following map
    \begin{align*}
    \mathcal{P}:H^s(M)/\mathcal{N} &\longrightarrow  H^{s-m}(M) \times \prod_{j=1}^{k} H^{s-m_j-\frac{1}{2}}(\partial M) \\
    \left[\w\right] &\longmapsto \left( A\w, B_1\w, \dots, B_{k}\w \right)
\end{align*}
  is a topological and algebraic isomorphism onto the space \[
\begin{aligned}
    \im(\mathcal{P}) &= \left\{ (f; g_1, \dots, g_k) \in H^{s-m}(M) \times \prod_{j=1}^{k} H^{s-m_j-\frac{1}{2}}(\partial M) \mid \right. \\
    & \quad \left. \int_M \langle f, v\rangle \, d\mu_g + \sum_{j=1}^{k} \int_{\partial M} \langle g_j, T_j v \rangle \, d\mu_g = 0 \ \text{for all } v \in \mathcal{N}^* \right\}.
\end{aligned}
\]
In particular, solutions of $(A\w,B_1\w,B_2\w,\cdots,B_k\w)=(0,\w_1,\cdots,\w_k)$ such that $\w_i\in \Omega(\partial M)$, are in all $H^s(M)$ with $s\geq m$, i.e., are smooth.
\end{thm}

\subsection{Formulation of the boundary-value problem}\label{sec:weak}
We look for a generalization of the scalar NBS \eqref{eq:IntroNBS} to differential forms that allows us to establish the Kuttler--Sigillito inequalities in this setting. Based on the definition of $H^2_{N}(M)$ given in \eqref{eq:hn3}, we consider the subspace
\begin{equation}\label{eq:HN1check}
\begin{aligned}
\check{H}^2_{N}(M)
&:= \left\{ \omega \in H^2_{N}(M)\ \mid\ \omega \perp_{L^2(\partial M)} H_{\Abs}(M) \right\}.
\end{aligned}
\end{equation} To prove the inequalities, we need $\ell_{1,p}$ to be characterized by:
\begin{align*}
\ell_{1,p}
&= \inf\left\{ 
\frac{\lVert \Delta \omega \rVert^2_{L^2(M)}}{\lVert \iota^* \omega \rVert^2_{L^2(\partial M)}} 
\;\middle|\; 
\omega \in \check{H}^2_{N}(M),\ \iota^* \omega \neq 0 \ \text{on}\ \partial M 
\right\}.
\end{align*}
The aim is to define the adapted boundary-value problem. Let us assume that $\bar{\omega}$ is smooth. We will show a posteriori that this is the case. We compute the critical points of the following functional, expressing the Poisson ratio for the NBS problem: $\mathcal{Q}(\bar{\omega}) := \frac{\lVert \Delta \bar{\omega} \rVert^2_{L^2(M)}}{\lVert \iota^*\bar{\omega}\rVert^2_{L^2(\partial M)}}$ on $\check{H}^2_{N}(M)$ which will be the eigenforms of the NBS problem that we determine. The critical points of the functional $ \mathcal{Q}$ are those $\bar{\omega}$ such that $ d_{\bar{\omega}} \mathcal{Q}$ is zero, i.e., $\bar{\omega}\in \check{H}^2_{N}(M)$ is a critical point of $\mathcal{Q}$ if and only if for all $ \w' \in \check{H}^2_{N}(M)$, we have
\begin{equation}\label{eq:pointsCritiquesNBS}
    (\Delta\bar{\omega},\Delta\w')_{L^2(M)}\lVert\iota^*\bar{\omega}\rVert^2_{L^2(\partial M)} = (\iota^*\bar{\omega},\iota^*\w')_{L^2(\partial M)}\lVert\Delta\bar{\omega}\rVert^2_{L^2(M)}.
\end{equation}
\begin{prop}
    The form $\bar{\w}\in\check{H}^2_N(M)$ is a critical point of $\mathcal{Q}$ if and only if it is a smooth solution of NBS.
\end{prop}
\begin{proof}
Using an integration by parts \eqref{eq:IPP2} in \eqref{eq:pointsCritiquesNBS}, we get\[
\begin{aligned}
(\Delta^2\bar{\omega},\omega')_{L^2(M)} &- (\nu\lrcorner d\Delta\bar{\omega},\iota^*\omega')_{L^2(\partial M)} - (\nu\lrcorner \Delta\bar{\omega},\iota^*\delta\omega')_{L^2(\partial M)} \\[1ex]
&= \frac{\lVert \Delta\bar{\omega}\rVert^2_{L^2(M)}}{\lVert \iota^*\bar{\omega}\rVert^2_{L^2(\partial M)}}(\iota^*\bar{\omega},\iota^*\omega')_{L^2(\partial M)}.
\end{aligned}
\]
Denoting $\ell=\frac{\lVert \Delta\bar{\omega}\rVert^2_{L^2(M)}}{\lVert \iota^*\bar{\omega}\rVert^2_{L^2(\partial M)}}$, it follows that
\begin{equation}\label{eq:preuveFonctionnelleNBS}
(\Delta^2\bar{\omega},\w')_{L^2(M)} = \left(\nu\lrcorner d\Delta\bar{\omega}+\ell\iota^*\bar{\omega},\iota^*\w'\right)_{L^2(\partial M)}+ (\nu\lrcorner \Delta\bar{\omega},\iota^*\delta\w')_{L^2(\partial M)}.
\end{equation} To conclude the proof, we use the following Lemma: 
\begin{lemma}\label{lem:Bilap1}
     For all $s\geq4$, $f\in H^{s-4}(M)$, $\w_1 \in H^{s-\frac{1}{2}}(\partial M)$, $\w_2 \in H^{s-\frac{3}{2}}(\partial M)$, $\w_3 \in H^{s-\frac{1}{2}}(\partial M)$, and $\w_4 \in H^{s-\frac{3}{2}}(\partial M)$ the boundary problem
   \begin{equation}\label{eq:Bilap1}
    (biLap_1)
    \left\{
    \hspace{-0.3em}
    \begin{array}{r@{\quad}c@{\quad}l@{\quad}l}
        \Delta^2 \w &=& f & \text{on } M \\[0.15em]
        \nu \lrcorner \w &=& \w_1 & \text{on } \partial M \\[0.15em]
        \nu \lrcorner d\w &=& \w_2 & \text{on } \partial M \\[0.15em]
        \iota^* \w &=& \w_3 & \text{on } \partial M \\[0.15em]
        \iota^* \delta\w &=& \w_4 & \text{on } \partial M
    \end{array}
    \right.
\end{equation} has a unique solution in $H^s(M)$. Moreover, if $f,\w_1,\cdots,\w_4$ are smooth, then $\w$ is also smooth.
\end{lemma}
\begin{proof} Let us first show that the problem $(biLap_1)$ is elliptic in the sense of Shapiro-Lopatinskij. Following the notations of \secref{sec:EllipticiteConvFaible}, we denote by $B_1, B_2, B_3,$ and $B_4$ the following maps:
\begin{align*}
    B_1: \Omega^p(M) &\longrightarrow \Gamma(E_1=\Lambda^{p-1}T^*\partial M) \quad ; & 
    B_2: \Omega^p(M) &\longrightarrow \Gamma(E_2=\Lambda^{p}T^*\partial M) \\
    \w &\longmapsto \nu \lrcorner \w & 
    \w &\longmapsto \nu \lrcorner d\w \\ \\
    B_3: \Omega^p(M) &\longrightarrow \Gamma(E_3=\Lambda^{p}T^*\partial M) \qquad ; & 
    B_4: \Omega^p(M) &\longrightarrow \Gamma(E_4=\Lambda^{p-1}T^*\partial M) \\
    \w &\longmapsto \iota^*\w & 
    \w &\longmapsto \iota^*\delta\w
\end{align*}
    Note that $\dim \mathcal{M}^+_{\Delta^2,v} = \dim\bigoplus_{j=1}^4 E_j = 2\binom{n}{p}$. We compute the principal symbols of the $B_i$ applied to $y=e^{-\lvert v\rvert t}(at+b)\w_0$ for some $\w_0\in\Lambda^p T^*M_{|\partial M}$.  We find \begin{align*}
\sigma_{B_1}(-iv + \partial_t \nu^\flat)(y)(0) &= b \, \nu \lrcorner \omega_0, \\
\sigma_{B_2}(-iv + \partial_t \nu^\flat)(y)(0) &= \iota^* \omega_0 (a - b |v|) + i b \, \iota^* v \wedge (\nu \lrcorner \omega_0), \\
\sigma_{B_3}(-iv + \partial_t \nu^\flat)(y)(0) &= b \, \iota^* \omega_0, \\
\sigma_{B_4}(-iv + \partial_t \nu^\flat)(y)(0) &= -a \, \nu \lrcorner \omega_0 + |v| b \, \nu \lrcorner \omega_0 + i v b \lrcorner \iota^* \omega_0.
\end{align*}
 Therefore, for all $y=e^{-|v| t}(at + b)\in \mathcal{M}_{\Delta^2,v}^+$, the map $\Phi$ of  \defref{def:ellipticiteShapiro} is given by
    \[
    \Phi(y) = \big( b\nu\lrcorner\w_0, (a-b|v|)\iota^*\w_0 + ib\iota^*v\wedge(\nu\lrcorner\w_0), b\iota^*\w_0, (b|v| - a)\nu\lrcorner\w_0 + bv\lrcorner\iota^*\w_0 \big).
    \] Let us show that $\Phi$ is injective. Suppose $\Phi(y) = 0$, then
    \begin{align*}
    &\begin{cases}
    b\nu\lrcorner\w_0 &= 0 \\
    (a-b|v|)\iota^*\w_0 + ib\iota^*v\wedge(\nu\lrcorner\w_0) &= 0 \\
    b\iota^*\w_0 &= 0 \\
    (b|v| - a)\nu\lrcorner\w_0 + bv\lrcorner\iota^*\w_0 &= 0
    \end{cases} 
    \Longrightarrow
    \begin{cases}
    b\nu\lrcorner\w_0 &= 0 \\
    (a-b\lvert v\rvert)\iota^*\w_0 &= 0 \\
    b\iota^*\w_0 &= 0 \\
    -(a-b\lvert v\rvert)\nu\lrcorner\w_0 &= 0.
    \end{cases}
    \end{align*}
    Hence, $(a-b\lvert v\rvert)\w_0 = 0$ and $b\w_0 = 0$, so $y = 0$. Thus, $\Phi$ is injective and consequently it is an isomorphism. Therefore $(biLap_1)$ is elliptic in the sense of Shapiro-Lopatinskij. Therefore, if $f,\w_1,\ldots,\w_n$ are smooth, then so is $\w.$ \\
    Performing an integration by parts \eqref{eq:IPP4} we get \[
\begin{aligned}
(\Delta^2 \w, \w')_{L^2(M)} &- (\w, \Delta^2 \w')_{L^2(M)} \\
&= (\nu \lrcorner d \Delta \w, \iota^* \w')_{L^2(\partial M)} - (\iota^* \w, \nu \lrcorner d \Delta \w')_{L^2(\partial M)} \\
&\quad - (\iota^* \delta \Delta \w, \nu \lrcorner \w')_{L^2(\partial M)} + (\nu \lrcorner \w, \iota^* \delta \Delta \w')_{L^2(\partial M)} \\
&\quad + (\nu \lrcorner \Delta \w, \iota^* \delta \w')_{L^2(\partial M)} - (\iota^* \delta \w, \nu \lrcorner \Delta \w')_{L^2(\partial M)} \\
&\quad - (\iota^* \Delta \w, \nu \lrcorner d \w')_{L^2(\partial M)} + (\nu \lrcorner d \w, \iota^* \Delta \w')_{L^2(\partial M)},
\end{aligned}
\] and based on the notations given in \secref{sec:EllipticiteConvFaible} we get, $ T_1 \w = -\iota^* \Delta \w $, $ T_2 \w = -\iota^* \delta \Delta \w $, $ T_3 \w = \nu \lrcorner \Delta \w $, and $ T_4 \w = \nu \lrcorner d \Delta \w $. Similarly,  
$ C_1 \w = \iota^* \w $, $ C_2 \w = -\nu \lrcorner \w $, $ C_3 \w = \iota^* \delta \w $, and $ C_4 \w = -\nu \lrcorner d \w $. Finally, $B_1=\nu\lrcorner d\w'$, $B_2=\nu\lrcorner\w'$, $B_3=\iota^*\delta\w'$ and $B_4=\iota^*\w'.$ Therefore, the associated kernel and co-kernel are \begin{equation*}
\begin{aligned}
\mathcal{N} = \mathcal{N}^* = \Big\{ \omega \in H^4(M) \;\Big|\; & \Delta^2 \omega = 0 \text{ on } M, \\
& \nu \lrcorner \omega = 0, \;\; \nu \lrcorner d\omega = 0, \;\; \iota^* \omega = 0, \;\; \iota^* \delta \omega = 0 \text{ on } \partial M \Big\}.
\end{aligned}\end{equation*} For any $\w \in \mathcal{N}^*$, an integration by parts using \eqref{eq:IPP2} shows that $\Delta\w = 0$ in $M$. Since $\w_{|\partial M} = 0$, it follows from \cite{ColetteAnne} that $\w = 0$ on $M$. Therefore, $\mathcal{N}=\mathcal{N}*=0$. So by \thref{theo:thm5.4} for all $s\geq4,$ the map \begin{eqnarray*}
\mathcal{P}_{biLap_1}:{H^s}(M)&\longrightarrow & H^{s-4}(M) \oplus H^{s-\frac{1}{2}}(\partial M)\oplus H^{s-\frac{3}{2}}(\partial M)\oplus H^{s-\frac{1}{2}}(\partial M)\oplus H^{s-\frac{3}{2}}(\partial M)\\
  \left[\w\right]& \longmapsto & (\Delta^2\w,\nu\lrcorner\w,\nu\lrcorner d\w,\iota^*\w,\iota^*\delta\w)  \notag 
\end{eqnarray*} is an isomorphism on its image that is given for all $s\geq 4$ by
\[ \im(\mathcal{P}_{biLap_1}) =  H^{s-4}(M) \times H^{s-\frac{3}{2}}(\partial M)\times H^{s-\frac{1}{2}}(\partial M)\times H^{s-\frac{3}{2}}(\partial M)\times H^{s-\frac{1}{2}}(\partial M).\] This concludes the proof of the Lemma.
\end{proof}
Let us now show that $\Delta^2\bar{\omega}=0$ on $M$. We define the sets $D_\epsilon:=\{x\in M\ | \ d(x,\partial M)\geq \epsilon\}$ and  $D_{2\epsilon}:=\{x\in M\ | \ d(x,\partial M)\geq 2\epsilon\}$. There exists $\chi_\epsilon \in C^\infty(M, [0, 1])$ such that ${\chi_\epsilon}_{|D_{2\epsilon}} = 1$ and ${\chi_\epsilon}_{|(M \backslash D_\epsilon)} = 0$. Let $\w_\epsilon := \chi_\epsilon \Delta^2 \bar{\omega} \in \Omega^p(M)$ then ${\w_\epsilon}_{|(M \backslash D_\epsilon)} = 0$ thus, $\nu \lrcorner \w_\epsilon,\ \nu \lrcorner d\w_\epsilon$, $\nu\lrcorner\Delta \w_\epsilon$ and $\iota^* \w_\epsilon $ vanish on $\partial M$ since $\w_\epsilon$ vanishes on a neighbourhood of $\partial M$. Then $ \int_M \langle \Delta^2 \bar{\omega}, \chi_\epsilon \Delta^2 \bar{\omega} \rangle d\mu_g = 0 = \int_M \chi_\epsilon \lvert \Delta^2 \bar{\omega} \rvert^2 d\mu_g,$ hence $\chi_\epsilon \Delta^2 \bar{\omega} = 0.$ Thus, $\chi \Delta^2 \bar{\omega} = 0$, so $\Delta^2 \bar{\omega} = 0$ on $D_\epsilon$, and letting $\epsilon$ tend to $0$, we conclude that $\Delta^2 \bar{\omega} = 0$ on $M$.  
\begin{figure}[H]
    \centering
    \captionsetup{labelformat=empty}
    \begin{tikzpicture}[scale=0.6]

        \definecolor{mygreen}{RGB}{126, 211, 33}
        \definecolor{myblue}{RGB}{74, 144, 226}
        \definecolor{myred}{RGB}{208, 2, 27}
        \definecolor{myorange}{RGB}{245, 166, 35}
        
        \def\rInner{1.5}  
        \def\rMid{3.0}    
        \def\rOuter{4.5}  

        \draw[thick] (0,0) circle (\rInner); 
        \draw[thick, myblue] (0,0) circle (\rMid); 
        \draw[thick, mygreen] (0,0) circle (\rOuter); 

        \draw[<->, >=latex, thick, myred] (0, \rInner) -- (0, \rMid) node[midway, right] {$\epsilon$};
        
        \draw[<->, >=latex, thick, myred] (0, -\rMid) -- (0, -\rOuter) node[midway, right] {$\epsilon$};

        \node[right] at (\rInner+0.1, 0) {$D_{2\epsilon}$};
        \node[right, myblue] at (\rMid+0.1, 0) {$D_\epsilon$};
        
        \draw[<-, >=latex, mygreen, thick] ({4.5*cos(45)}, {4.5*sin(45)}) -- ++(1.2, 1) node[right] {$\partial M$};
        
        \node[mygreen] at ({3.8*cos(15)}, {3.8*sin(15)}) {$M$};

        \node[myorange] at (0, 0) {$\chi_\epsilon = 1$};
        \node[myorange, rotate=-30] at ({3.75*cos(60)}, {3.75*sin(60)}) {$\chi_\epsilon = 0$};

    \end{tikzpicture}
    \caption{Figure: Cutoff function in a neighborhood of the boundary}
    \label{fig:cutoff_function}
\end{figure}
Finally, by \lemref{lem:Bilap1} there exists $\w' \in H^2_{N}(M)$ such that \[
    \left\{
    \hspace{-0.3em}
    \begin{array}{r@{\quad}c@{\quad}l@{\quad}l}
        \Delta^2 \w' &=& 0 & \text{on } M \\[0.15em]
        \nu \lrcorner \w' &=& 0 & \text{on } \partial M \\[0.15em]
        \nu \lrcorner d\w' &=& 0 & \text{on } \partial M \\[0.15em]
        \iota^* \delta\w' &=& \nu \lrcorner \Delta\bar{\omega} & \text{on } \partial M \\[0.15em]
        \iota^* \w' &=& \nu \lrcorner d\Delta\bar{\omega} + \ell \iota^* \bar{\omega} & \text{on } \partial M.
    \end{array}
    \right.
\]
Moreover, by integration by parts \eqref{eq:IPP2}, $\w'\perp_{L^2(\partial M)}H_{\Abs}(M)$. Since $\w'\in \check{H}^2_{N}(M)$, then by \eqref{eq:preuveFonctionnelleNBS} we get \[(\nu\lrcorner d\Delta\bar{\omega}+\ell\iota^*\bar{\omega},\iota^*\w')_{L^2(\partial M)}+ (\nu\lrcorner \Delta\bar{\omega},\iota^*\delta\w')_{L^2(\partial M)}=0\] hence $\nu\lrcorner d\Delta\bar{\omega}+\ell\iota^*\bar{\omega}=0$ and $\nu\lrcorner\Delta\bar{\omega}=0$. Therefore, $\bar{\omega}$ is a smooth critical point of $\mathcal{Q}$ if and only if it satisfies:
\begin{equation*}
    \left\{
    \hspace{-0.3em}
    \begin{array}{r@{\quad}c@{\quad}l@{\quad}l}
        \Delta^2 \bar{\omega} &=& 0 & \text{on } M \\[0.15em]
        \nu \lrcorner \bar{\omega} &=& 0 & \text{on } \partial M \\[0.15em]
        \nu \lrcorner d \bar{\omega} &=& 0 & \text{on } \partial M \\[0.15em]
        \nu \lrcorner \Delta \bar{\omega} &=& 0 & \text{on } \partial M \\[0.15em]
        \nu \lrcorner d \Delta \bar{\omega} + \ell \iota^* \bar{\omega} &=& 0 & \text{on } \partial M
    \end{array}
    \right.
\end{equation*}for $\ell=\frac{\lVert \Delta\bar{\omega}\rVert^2_{L^2(M)}}{\lVert \iota^*\bar{\omega}\rVert^2_{L^2(\partial M)}},$ i.e., is an eigenform of NBS associated to $\ell.$\end{proof}
\begin{remark}
By definition, the spectrum of NBS is the spectrum of the operator 
\[
\begin{aligned}
NBS^{\text{form}}: \Omega^{p}(\partial M) &\longrightarrow \Omega^{p}(\partial M) \\
\omega_1 &\longmapsto -\nu \lrcorner d\Delta\hat{\omega}
\end{aligned}
\]
where $\hat{\omega}$ is the smooth solution of\begin{equation}
    \left\{
    \hspace{-0.3em}
    \begin{array}{r@{\quad}c@{\quad}l@{\quad}l}
        \Delta^2 \hat{\omega} &=& 0 & \text{on } M \\[0.15em]
        \nu \lrcorner \hat{\omega} &=& 0 & \text{on } \partial M \\[0.15em]
        \nu \lrcorner d\hat{\omega} &=& 0 & \text{on } \partial M \\[0.15em]
        \nu \lrcorner \Delta\hat{\omega} &=& 0 & \text{on } \partial M \\[0.15em]
        \iota^* \hat{\omega} &=& \omega_1 & \text{on } \partial M
    \end{array}
    \right.
\end{equation}
\end{remark}
\begin{lemma}\label{lem:p=nNBS}
    For $p=0$, the NBS problem reduces to the classical NBS for scalar functions \eqref{eq:IntroNBS}. For $p=n$, any solution $\w$ for NBS must be zero.
\end{lemma}

\begin{proof}
It is immediate that the NBS problem for $p=0$ reduces to the NBS for functions. For $p=n$ we have $\w=f \vol g$ with $f$ a real function. Then the second and last boundary conditions are trivial. The remaining conditions in the NBS give that $\Delta f=0$ on $M$ and $f_{|\partial M}=0$, using an integration by parts, we get that $f=0$ and so $\w=0$. Hence this problem is trivial.
\end{proof}That is why for all the results related to NBS we consider $p\in\{1,\cdots,n-1\}.$
\subsection{Ellipticity of NBS}\label{sec:NBSelliptic}
In this section, we establish the ellipticity of the NBS problem. Given that the bi-Laplacian ($\Delta^2$) is a fourth-order elliptic operator subject to boundary conditions defined by differential operators, we use the Shapiro-Lopatinskij approach. This allows us to verify the ellipticity of the boundary value problem and subsequently that the eigenfunctions (and the corresponding eigenspaces) possess the required regularity, ensuring they are smooth. Recall that we fix the degree $p\in \{1,\cdots,n-1\}$ of the differential forms which we consider, see \lemref{lem:p=nNBS}.
\begin{thm}
    Let $(M^n,g)$ be a compact Riemannian manifold with a smooth boundary $\partial M$. The NBS problem \eqref{eq:NBSF3} is elliptic in the sense of Shapiro-Lopatinskij.
\end{thm}

\begin{proof}
    Let $B_1,B_2,B_3$ and $B_4$ be the following maps
\begin{align*}
B_1: \Omega^p(M) &\longrightarrow \Gamma(E_1=\Lambda^{p-1}T^*\partial M) \qquad ; & B_2: \Omega^p(M) &\longrightarrow \Gamma(E_2=\Lambda^{p}T^*\partial M) \\
\w &\longmapsto \nu\lrcorner\w & \w &\longmapsto \nu\lrcorner d\w \\ \\
B_3: \Omega^p(M) &\longrightarrow \Gamma(E_3=\Lambda^{p-1}T^*\partial M) \qquad ; & B_4: \Omega^p(M) &\longrightarrow \Gamma(E_4=\Lambda^{p}T^*\partial M) \\
\w &\longmapsto\nu\lrcorner\Delta\w& \w & \longmapsto \nu\lrcorner\Delta d\w +\ell\iota^*\w 
\end{align*} where $\ell$ is a fixed real number. We note that $E_1$ and $E_3$ have rank $\binom{n-1}{p-1}$ and that $E_2$ and $E_4$ have rank $\binom{n-1}{p}$. Therefore, the rank of $\bigoplus_{j=1}^4 E_j$ is $2\big(\binom{n-1}{p} + \binom{n-1}{p-1}\big) = 2\binom{n}{p} = \dim \mathcal{M}_{\Delta^2,v}^+$. We have to compute the principal symbols of the $B_j$ for $1 \leq j \leq 4$.\\ Let us fix $\xi\in T^*M$ and let $f:M\longrightarrow\mathbb{R}$ be such that $d_xf=\xi$. The operator $B_1$ is of order 0, so its principal symbol is $B_1$ itself: $\sigma_{B_1}(\xi)\w = \nu \lrcorner \w$. The operator $B_2$ is of order 1, recall that $\xi$ has both a normal and a tangential component, i.e., $\xi=\xi_\nu\nu^{\flat}+\iota^*\xi$ for all point in $\partial M$, we have then $\sigma_{B_2}(\xi)\w = \xi_\nu\iota^*\w - \iota^*\xi\wedge(\nu\lrcorner\w).$ The operator $B_3$ is of order 2 and
 $\sigma_{B_3}(\xi)\w=\sigma_{\nu\lrcorner}\circ\sigma_\Delta(\xi)\w=-\nu\lrcorner\lvert\xi\rvert^2\w=-\lvert\xi\rvert^2\nu\lrcorner\w.$ \\ Finally, $B_4$ is of order $3$ and $\sigma_{B_4}(\xi)\w=\big(\sigma_{\nu\lrcorner}\circ\sigma_\Delta\circ\sigma_d\big)(\xi)\w$. We have $ \sigma_d(\xi)\w= \xi\wedge\w$, see \cite[p.181]{taylor}. Therefore $\sigma_{\Delta \circ d} = -|\xi|^2 \xi \wedge \w$. Since $\nu \lrcorner$ is of order 0, we get $\sigma_{B_4}(\xi)\w=-\lvert\xi\rvert^2\nu\lrcorner(\xi\wedge\w).$\\ We now need to evaluate all of the principal symbols for $\xi=(-iv+\partial_t\nu^\flat)$ with $v\in T^*M\setminus\{0\}$ and apply it to $y(t)=e^{-\lvert v\rvert t}(at+b)\w_0$. For $t=0$, we get: $\sigma_{B_1}(-iv+\partial_t\nu^\flat)(y)(0)=b\nu\lrcorner\w_0,$ $\sigma_{B_2}(-iv+\partial_t\nu^\flat)(y)(0)=\iota^*\w_0(a - b\lvert v\rvert) + ib\iota^*v\wedge(\nu\lrcorner\w_0)$, $\sigma_{B_3}(-iv+\partial_t\nu^\flat)(y)(0)=2a\lvert v \rvert\nu\lrcorner\w_0$ and $\sigma_{B_4}(-iv+\partial_t\nu^\flat)(y)(0) = 2 |v| a \left( iv \wedge (\nu \lrcorner \w_0) - |v| \iota^* \w_0 \right).$ Based on the calculations above, we obtain the following map:
\begin{align*}
\Phi : \mathcal{M}^+_{\Delta^2,v} &\longrightarrow \bigoplus_{j=1}^4 E_j \\
y &\longmapsto \Phi(y)
\end{align*}
such that  \[\Phi(y)=\big( b\nu\lrcorner\w_0,(a-b\lvert v \rvert)\iota^*\w_0+ib\iota^*v\wedge(\nu\lrcorner\w_0),2a\lvert v \rvert \nu\lrcorner\w_0,2\lvert v\rvert \big(iav\wedge(\nu\lrcorner\w_0)-\lvert v\rvert a\iota^*\w_0\big)\big).\] To end the proof, we must show that $\Phi$ is an isomorphism. Since both source and target spaces have the same finite dimension, it is enough to show that $\Phi$ is injective. Suppose that $\Phi(y)=0,$ then
\begin{align*}
&\begin{cases}
  b\nu\lrcorner\w_0 &= 0 \qquad (1) \\
  (a-b|v|)\iota^*\w_0 + ib\iota^*v\wedge(\nu\lrcorner\w_0) &= 0  \qquad (2)\\
2a\lvert v \rvert \nu\lrcorner\w_0 &= 0  \qquad (3)\\
 2\lvert v\rvert \big(iav\wedge(\nu\lrcorner\w_0)-\lvert v\rvert a\iota^*\w_0\big)  &= 0  \qquad (4)
\end{cases}
\end{align*}
If $a$ and $b$ are zero, then $y$ is zero. If $a$ or $b$ are non-zero, then from (1) and (3), we deduce that $\nu\lrcorner\w_0 = 0$. Substituting this result into (2) and (4), we find that $\iota^*\w_0 = 0$. Consequently, $\w_0 = 0$, which implies that $y = 0$. Hence, $\Phi$ is injective and the problem \eqref{eq:NBSF3preuve} is elliptic.
\end{proof}

\subsection{Discrete spectrum of NBS}\label{sec:discretespectrumofNBS}

The aim of this section is to prove \thref{thm:TheoremeGeneralNBSF3PourLaPreuve}. We still consider $p\in\{1,\cdots,n-1\}$ (see \lemref{lem:p=nNBS}). We recall that $H_N^2(M)$ has been defined in \eqref{eq:hn3}: 
 \begin{equation*}\label{espace:H2N1} 
H_{N}^2(M) := \{\w \in H^2(M) \ | \ \nu\lrcorner\w = 0, \ \text{and} \ \nu\lrcorner d\w = 0\ \text{on} \ \partial M \}   
 \end{equation*}where $\nu\lrcorner\w$ is defined in $H^{\frac{3}{2}}(\partial M)$ and $\nu\lrcorner d\w$ in $H^{\frac{1}{2}}(\partial M)$ by \thref{th:trace}.\label{espace:Z} We also\label{espace:checkH2N1} introduce\label{espace:Zcheck} the following spaces:\begin{align*}
   Z                  &:= \bigl\{ \w \in \Omega^p(M) \bigm| \Delta^2\w = 0 \text{ on } M, \nu\lrcorner\w=0,\ \nu\lrcorner d\w=0\ \text{and}\ \nu\lrcorner \Delta\w=0 \bigr\}\\
   &\,\subset H_N^2(M) \nonumber \\[1ex]
    \check{Z}          &:= \bigl\{ \w \in Z \bigm| \w \perp_{L^2(\partial M)} H_{\Abs}(M) \bigr\}\subset \check{H}_N^2(M)\nonumber.
\end{align*}Let us define for all $ \w, \w' \in \Omega^p(M) $ the following bilinear forms:\begin{equation*}
\begin{aligned}
( \w, \w' )_{V} 
&:= \int_M \langle \Delta \w, \Delta \w' \rangle \, d\mu_g
\qquad\text{and} \qquad
( \w, \w' )_{W} 
:= \int_{\partial M} \langle \iota^* \w, \iota^* \w' \rangle \, d\mu_g.
\end{aligned}
\end{equation*}
\begin{lemma}\label{lem:psNBS}
    The forms $( \cdot, \cdot)_{V} $ and $ ( \cdot, \cdot )_{W} $ are scalar products on $ \check{H}^2_N(M) $ and $ Z $, respectively.
\end{lemma} 

\begin{proof}
  It can be shown that the forms defined above are bilinear, symmetric, and non-negative. Let us show that they are positive definite. For $ \w \in \check{H}^2_N(M) $, if $ ( \w, \w )_{V} = 0 $ then $ \Delta \w = 0 $ on $L^2(M)$. Thus, \eqref{eq:IPP3} gives that $ d\w $ and $ \delta \w $ are zero, and since $ \nu \lrcorner \w = 0 $ it follows that $ \w \in H_{\Abs}(M) $. Moreover, since $ \w \in \check{H}^2_N(M) $ this implies that $\int_{\partial M} \langle \iota^* \w, \iota^* \w \rangle d\mu_g = 0$ by the definition of elements of $\check{H}^2_N(M) $. Therefore, $\int_{\partial M} | \iota^* \w |^2 d\mu_g = 0,$
which means that $ \iota^* \w = 0 $. As a result, $ \w_{|\partial M} = 0 $, and because $\w$ is harmonic, we can conclude that $ \w = 0 $ on $ M $ according to \cite{ColetteAnne}.\\ 
Similarly, for $ \w \in Z $, if $ ( \w, \w )_{W} = 0 $, then 
$\int_{\partial M} \lvert \iota^* \w \rvert^2 d\mu_g = 0,$
which implies that $ \iota^* \w $ is zero on $ \partial M $. Substituting into \eqref{eq:IPP2} with $ \w \in Z $, we obtain $ \Delta \w = 0 $ on $ M $. Since $ \w_{|\partial M} = 0 $ we conclude again that $ \w = 0 $ on $ M $ according to \cite{ColetteAnne}.\\
\end{proof} 

\begin{definition}\label{def:espaceV}
We define the space $W$ as the completion of $Z$ with respect to $( \w, \w')_{W}$, and the space $V$ as the completion of $\check{Z}$ with respect to $( \w, \w' )_{V}$.\end{definition} \index{W02@\col{$W_*$}{completion of $Z_*$ w.r.t $(\cdot,\cdot)_{W_*}$}|hyperpage}
\index{V01@\col{$V_*$}{completion of $\check{Z}_*$ w.r.t $(\cdot,\cdot)_V$}|hyperpage}
We note that for any $ \w $ an $\ell$-eigenform of NBS and $ \w' \in Z $, we have from equation \eqref{eq:IPP2} the following identity:
    \begin{equation}\label{eq:utiliserdanspreuveth}
        \int_M \langle \Delta \w, \Delta \w' \rangle \, d\mu_g = \ell \int_{\partial M}  \langle \iota^* \w, \iota^* \w' \rangle \, d\mu_g.
   \end{equation}
From \eqref{eq:IPP1} and \eqref{eq:IPP2} by taking $ \w \in Z $ and $ \w' = \w $, we get 
  \begin{equation}\label{eq:integrationPourWdansZ3}
      \int_M \lvert \Delta\w  \rvert^2d\mu_g+\int_{\partial M} \langle \nu \lrcorner d \Delta\w, \iota^* \w\rangle d\mu_g =0.\end{equation}
\begin{proof}[\textbf{Proof of  \thref{thm:TheoremeGeneralNBSF3PourLaPreuve}}]
We use different intermediate results in this proof. \begin{lemma}  \label{lem:lpositiveNBS}
  The kernel of NBS coincides with $H_{\Abs}(M)$. Moreover, if $\w$ is a solution of $NBS$ with $\ell\neq 0,$ then $\w\perp_{L^2(\partial M)}H_{\Abs}(M)$ and $\ell >0.$
  \end{lemma}
  
  \begin{proof}
Let $\w$ be a solution of NBS with $\ell=0$. Integrating by parts using equation \eqref{eq:IPP2} with $\w = \w'$, we obtain:
    \begin{equation*}
\begin{split}
    & \int_M\langle \underbrace{\Delta^2 \w}_{0} , \w\rangle  d\mu_g  =  \int_M\langle \Delta\w , \Delta \w\rangle  d\mu_g \\
    & + \int_{\partial M} \Big( \langle \underbrace{\nu \lrcorner d \Delta\w}_{0}, \iota^* \w\rangle-\langle \iota^* \Delta\w, \underbrace{\nu \lrcorner d \w}_{0}\rangle  + \langle \underbrace{\nu \lrcorner \Delta\w}_{0}, \iota^*\delta \w\rangle-\langle \iota^* \delta \Delta\w, \underbrace{\nu \lrcorner \w}_{0}\rangle    \Big)d\mu_g,
\end{split}
\end{equation*} 
    thus $\Delta\w=0$. Substituting this into \eqref{eq:IPP3}, we get: 
    \begin{equation*}
\begin{aligned}
   \int_M \langle \underbrace{\Delta \w}_{0}, \w \rangle \, d\mu_g &= \int_M \left( \lvert d\w \rvert^2 + \lvert \delta\w \rvert^2 \right) d\mu_g+ \int_{\partial M} \langle \underbrace{\nu \lrcorner d\w}_{0}, \iota^* \w \rangle \, d\mu_g \\
   &\quad - \int_{\partial M} \langle \iota^* \delta \w ,\underbrace{\nu \lrcorner \w}_{0} \rangle \, d\mu_g.
\end{aligned}
\end{equation*} 
    This implies that $d\w=0$ and $\delta\w=0$, so $\w$ lies in $H_{\Abs}(M)$.\\
    Conversely, if $\w\in H_{\Abs}(M)$, then it is a solution of NBS for $\ell=0$ by definition of $H_{\Abs}(M)$ \eqref{eq:Hap}. Now if $\w$ is a solution of the $NBS$, $\ell\neq 0$ and $\w\in Z$, we have by \eqref{eq:utiliserdanspreuveth} for all $\alpha\in H_{\Abs}(M)$ that \[0=\int_M\langle\Delta\w,\Delta\alpha\rangle d\mu_g=\ell\int_{\partial M} \langle\iota^*\w,\iota^*\alpha\rangle d\mu_g.\] Since $\ell\neq 0$ we get that $\w$ is ${L^2(\partial M)}$-orthogonal to $H_{\Abs}(M)$, therefore $\w\in\check{Z}\subset V$. Moreover, still by \eqref{eq:utiliserdanspreuveth}, we get $\int_M\lvert\Delta\w\rvert^2 d\mu_g=\ell\int_{\partial M}\lvert\iota^*\w\rvert^2d\mu_g.$ Therefore $\ell >0,$ which concludes the proof.
\end{proof}

\begin{prop}\label{prop:equivalence}
    The norms $ \lVert \cdot \rVert_{V} $ and $ \lVert \cdot \rVert_{H^2(M)} $ are equivalent on $\check{Z}$ and $V\subset \check{H}^2_N(M)$.
\end{prop}

\begin{proof}
    As $ \Delta $ is a second-order differential operator, there exists $ c > 0 $ such that for all $ \w \in \check{Z} $, $ \lVert \w \rVert_{V} \leq c \lVert \w \rVert_{H^2(M)} $. It remains to show that for all $ \w \in \check{Z} $, $ \lVert \w \rVert_{H^2(M)} \leq c \lVert \w \rVert_{V} $. We show the following Lemma: 
\begin{lemma}\label{lm:SteklovFormeGénéraleCorrigée}
    For all $s \ge 2$, the map
\begin{align}\label{eq:SteklovFormeGénéraleCorrigée}
    \mathcal{P}_{\Delta_{\text{Neu}}} : H^s(M)/H_{\Abs}(M)&\longrightarrow H^{s-2}(M) \times H^{s-\frac{1}{2}}(\partial M) \times H^{s-\frac{3}{2}}(\partial M) \\ \nonumber
    [\omega] &\longmapsto (\Delta \omega, \nu \lrcorner \omega, \nu \lrcorner d\omega)
\end{align}
is a topological isomorphism on its image which is
\begin{align*}
    \bigg\{ (\omega_0, \omega_1, \omega_2) &\in H^{s-2}(M) \times H^{s-\frac{1}{2}}(\partial M) \times H^{s-\frac{3}{2}}(\partial M) \ \bigg| \\
    &\quad \forall \omega'\in H_{\Abs}(M), (\omega_0, \omega')_{L^2(M)} +  (\omega_2, \iota^* \omega' )_{L^2(\partial M)} = 0 \bigg\}.
\end{align*}
\end{lemma}
\begin{proof}
\underline{\textbf{Step 1:}} We begin by establishing the ellipticity of problem \eqref{eq:SteklovFormeGénéraleCorrigée} in the sense of Shapiro-Lopatinskij. Let $B_1$ and $B_2$ be the following maps: 
    \begin{align*}
    B_1: \Omega^p(M) &\longrightarrow \Gamma(E_1 = \Lambda^{p-1} T^*\partial M) \ \quad ;  & B_2: \Omega^p(M) &\longrightarrow \Gamma(E_2 = \Lambda^{p} T^*\partial M) \\
    \w &\longmapsto \nu \lrcorner \w & \w &\longmapsto \nu \lrcorner d\w
    \end{align*}
Similarly to \secref{sec:NBSelliptic}, we shall prove that the map
    \begin{align*}
    \Phi : \mathcal{M}^+_{\Delta,v} &\longrightarrow E_1 \oplus E_2 \\
    y &\longmapsto \left( \sigma_{B_1}(-iv+\partial_t\nu^\flat)(y),\, \sigma_{B_2}(-iv+\partial_t\nu^\flat)(y) \right)(0)
\end{align*}
    is an isomorphism. In this setting, the space $\mathcal{M}^+_{\Delta,v}$ is given by \label{espace:MDELTA3}\begin{align*}
\mathcal{M}^+_{\Delta,v} 
= &\spann\big(ae^{-|v| t}.\w_0 \,|\, a\in \mathbb{R}, \w_0\in \Lambda^pT^*_xM_{|\partial M} \big)
\end{align*} and has dimension $\binom{n}{p}$. We calculate the principal symbols of $B_1$ and $B_2$ for $\xi=(-iv+\partial_t\nu^\flat)$ with $v\in T^*M\setminus\{0\}$, apply them to $y(t)=e^{-\lvert v\rvert t}(at+b)\w_0\in \mathcal{M}^+_{\Delta,v}$ and evaluate them for $t=0$. We obtain that the map $\Phi$ is given by: \begin{align*}
    \Phi : \mathcal{M}^+_{\Delta,v} &\longrightarrow E_1 \oplus E_2 \\
    y &\longmapsto \left( a \nu \lrcorner \omega_0, -a |v| \iota^*\omega_0 + ia \iota^*v \wedge (\nu \lrcorner \omega_0) \right).
\end{align*}
    Suppose that $\Phi(y) = 0$. Then either $a=0$, or we have
   \begin{equation*}
    \left\{
    \hspace{-0.3em}
    \begin{array}{r@{\quad}c@{\quad}l@{\quad}l}
        a \nu \lrcorner \omega_0 &=& 0 & \\[0.5em]
        -a \lvert v\rvert \iota^*\omega_0 + ia \iota^* v \wedge (\nu \lrcorner \omega_0) &=& 0 &
    \end{array}
    \right.
    \quad\Longrightarrow\quad
    \left\{
    \hspace{-0.3em}
    \begin{array}{r@{\quad}c@{\quad}l@{\quad}l}
        \nu \lrcorner \omega_0 &=& 0 & \\[0.5em]
        \iota^*\omega_0 &=& 0 &
    \end{array}
    \right.
\end{equation*}
In either case, it follows that $\w_0 = 0$, which implies $y = 0$. Therefore, $\Phi$ is injective. Since both the source and target spaces have the same finite dimension, $\Phi$ is an isomorphism, which implies that \eqref{eq:SteklovFormeGénéraleCorrigée} is elliptic in the sense of Shapiro-Lopatinskij.\vspace*{\baselineskip}   

   \underline{\textbf{Step 2:}} We follow the approach and notations of \secref{sec:EllipticiteConvFaible}. For all $\omega \in \Omega^p(M)$, we have the following boundary-value problem
\[
    \left\{
    \hspace{-0.3em}
    \begin{array}{r@{\quad}c@{\quad}l@{\quad}l}
        A\omega &=& \Delta\omega & \\[0.15em]
        B_1\omega &=& \nu \lrcorner \omega & \\[0.15em]
        B_2\omega &=& \nu \lrcorner d\omega. &
    \end{array}
    \right.
\]

Integration by part \eqref{eq:IPP1} is then a special case of \eqref{eq:greenLionsMagenes} with $A^* = A = \Delta$ and\begin{equation}\label{eq:neu}
\begin{split}
    & \int_M\Big( \langle \Delta \w , \w'\rangle -  \langle \w , \Delta \w'\rangle \Big) d\mu_g \\
   =  & \int_{\partial M} \Big( \langle \underbrace{\nu \lrcorner \w}_{B_1\w}, \underbrace{\iota^*\delta \w'}_{-T_1\w'}\rangle +\langle \underbrace{\nu \lrcorner d \w}_{B_2\w}, \underbrace{\iota^* \w'}_{-T_2\w'}\rangle-\langle \underbrace{\iota^* \w}_{S_1\w}, \underbrace{\nu \lrcorner d \w'}_{-C_1\w'}\rangle  -\langle \underbrace{\iota^* \delta \w}_{S_2\w}, \underbrace{\nu \lrcorner \w'}_{-C_2\w'}\rangle  \Big)d\mu_g.
\end{split}
\end{equation}
In particular the kernel $\mathcal{N}$ and the co-kernel $\mathcal{N}^*$ of $A$ are given by
\[
    \mathcal{N}^* = \mathcal{N} = \{ \omega \in \Omega^p(M)\mid \Delta\omega = 0\ \text{on}\ M, \nu \lrcorner \omega = 0, \nu \lrcorner d\omega = 0\ \text{on}\ \partial M\} = H_{\Abs}(M).
\]

Moreover, by \thref{theo:thm5.4}, $\mathcal{P}_{\Delta_{\text{Neu}}}$ is a topological isomorphism from $H^s(M)/\mathcal{N}$ to its image which is given by
\begin{align*}
    \bigg\{ (\omega_0, \omega_1, \omega_2) &\in L^2(M) \times H^{\frac{3}{2}}(\partial M) \times H^{\frac{1}{2}}(\partial M) \ \bigg| \\
    &\quad \forall \alpha \in \mathcal{N}^*, ( \omega_0, \alpha )_{L^2(M)} + (\omega_1, T_1\alpha )_{L^2(\partial M)} + ( \omega_2, T_2\alpha )_{L^2(\partial M)} = 0 \bigg\}.
\end{align*}
Since by \eqref{eq:neu}, $T_1\alpha = \iota^* \delta \alpha = 0$ (as $\alpha \in H_{\Abs}(M)$), and $T_2\alpha = \iota^*\alpha$, we get
\begin{align*}
    \im(\mathcal{P}_{\Delta_{\text{Neu}}}) = \bigg\{ (\omega_0, \omega_1, \omega_2) &\in L^2(M) \times H^{\frac{3}{2}}(\partial M) \times H^{\frac{1}{2}}(\partial M) \ \bigg| \\
    &\quad \forall \alpha \in H_{\Abs}(M), (\omega_0, \alpha)_{L^2(M)} + (\omega_2, \iota^*\alpha)_{L^2(\partial M)} = 0 \bigg\}.
\end{align*}
Which ends the proof of \lemref{lm:SteklovFormeGénéraleCorrigée}.\end{proof}

We now consider the space \label{espace:checkL2}\[
(H_{\Abs}(M))^{\perp_{L^2(M)}} = \left\{ \omega \in L^2(M) \;\middle|\; \w\perp_{L^2(M)} H_{\Abs}(M)\right\}\subset L^2(M)
\] and prove the following Lemma:
 \begin{lemma}\label{lm:isom}
     The map 
$
\Delta_{|\check{H}^2_{N}(M)} : \check{H}^2_{N}(M) \longrightarrow (H_{\Abs}(M))^{\perp_{L^2(M)}}
$ is a topological isomorphism.
 \end{lemma}
\begin{proof}
   Using integration by parts \eqref{eq:IPP3'}, for all $\omega \in H^2_N(M)$ and $\alpha \in H_{\Abs}(M)$, we have
\[
    \langle \Delta \omega, \alpha \rangle_{L^2(M)} = 0.
\]
By \lemref{lm:SteklovFormeGénéraleCorrigée}, since $\left( H_{\Abs}(M) \right)^{\perp_{L^2(M)}} \times \{0\} \times \{0\}$ is contained in\\ $\mathcal{P}_{\Delta_{\text{Neu}}}(H_N^2(M) / H_{\Abs}(M))$, we conclude that $\Delta_{|\check{H}^2_N(M)}$ is continuous and surjective.\\
Let $\w \in \check{H}^2_{N}(M)$ such that $\Delta\w = 0$. Using integration by parts \eqref{eq:IPP3'}, we obtain that $d\w$ and $\delta\w$ are zero. Additionally, since $\nu \lrcorner \w = 0$, we conclude that $\w \in H_{\Abs}(M)$. Thus, $\iota^*\w = 0$ as $\w\in H_{\Abs}(M)$ and $\w\perp_{L^2(\partial M)} H_{\Abs}(M)$. Therefore, $\w=0$ by \cite{ColetteAnne}, so $\ker \Delta_{|{\check{H}^2_{N}(M)}} = \{0\}$. Hence, the operator $\Delta_{|\check{H}^2_{N}(M)} : \check{H}^2_{N}(M) \longrightarrow (H_{\Abs}(M))^{\perp_{L^2(M)}}$ is injective.
\end{proof}

Let us go back to the proof of \propref{prop:equivalence}. 
We define by $\mathcal{L}:=({\mathcal{P}_{\Delta_{Neu}}})_{|(H_{\Abs}(M))^{\perp_{L^2(M)}}\oplus \{0\}\oplus \{0\}}^{-1}$ the following:
\begin{eqnarray*}
\mathcal{L}:(H_{\Abs}(M))^{\perp_{L^2(M)}}\oplus \{0\}\oplus \{0\}&\longrightarrow &H^2_{N}(M)/H_{\Abs}(M)\\
(\hat{\w},0,0) & \longmapsto &  \left[\w\right]  \notag 
\end{eqnarray*}where $\w$ is such that $\Delta\w=\hat{\w}$ on $M$, $\nu\lrcorner\w=0$ and $\nu\lrcorner d\w=0$ on $\partial M$. Then by \lemref{lm:isom} the map $\mathcal{L}$ is continuous and for all $\hat{\w}\in (H_{\Abs}(M))^{\perp_{L^2(M)}}$, there exists $c>0$ such that
\[
\lVert \mathcal{L}(\hat{\w}) \rVert _{{H}^2(M)}\leq c \lVert \hat{\w} \rVert _{L^2(M)}
\]
therefore, for all $\w\in \check{H}^2_{N}(M)$ we have:
\[
\lVert \w \rVert _{H^2(M)}\leq c \lVert \Delta\w \rVert _{L^2(M)}.
\]
Since $\check{Z} \subset \check{H}^2_{N}(M)$, we have for all $\w\in \check{Z}$:
\[
\lVert \w \rVert _{H^2(M)}\leq c\lVert \Delta\w \rVert _{L^2(M)}=c \lVert \w \rVert_{V}.
\]
Thus, the norms $\lVert\cdot\rVert_{V}$ and $\lVert\cdot\rVert_{H^4(M)}$ are equivalent on $\check{Z}$, which completes the proof of \propref{prop:equivalence}.

\end{proof}
\begin{lemma}\label{lem:W3<V}
    There exists $c>0$ such that $\lVert\cdot\rVert_{W} \leq c\lVert\cdot\rVert_{V}$ on $\check{Z}.$
\end{lemma}

\begin{proof}
   Recall that $\lVert\w\rVert^2_{H^1(\partial M)}=\lVert\w\rVert^2_{L^2(\partial M)}+\lVert\nabla\w\rVert^2_{L^2(\partial M)}$. For all $\w\in\check{Z}$, we have by definition $\lVert\w\rVert^2_{W}=\lVert\iota^*\w\rVert^2_{L^2(\partial M)}.$ Moreover, \[\lVert\iota^*\w\rVert^2_{L^2(\partial M)}=\lVert\w\rVert^2_{L^2(\partial M)}\leq \lVert\w\rVert^2_{H^1(\partial M)}.\] Furthermore, by continuity of the trace operator $T:H^2(M)\longrightarrow H^1(\partial M)$, there exists $c>0$ such that $\lVert \w \rVert^2_{W}\leq  \lVert \w \rVert^2_{H^1(\partial M)}\leq c \lVert \w \rVert^2_{H^2(M)}.$ Finally, using the equivalence of the norms $\lVert \cdot \rVert_{V}$ and $\lVert \cdot \rVert_{H^2(M)}$, we obtain $\lVert \w \rVert^2_{W}\leq c \lVert \w \rVert_{V}^2$ for some constant $c.$
\end{proof}
From \lemref{lem:W3<V}, and having by definition that $\check{Z}$ is dense in $V$ and $Z$ is dense in $W$, the inclusion map $i_{W}:\check{Z}\longrightarrow W$, which is linear and continuous, extends to another continuous inclusion map $I$ denoted as
\begin{eqnarray*}
  I:V&\longrightarrow &W\\
  \w& \longmapsto &\w  \notag 
\end{eqnarray*} 
such that $i_{W}=I\circ i_{V}$, where $i_{V}: \check{Z}\longrightarrow V.$
\begin{lemma}\label{lem:I3compact}
    The map $I$ is compact.
\end{lemma}
\begin{proof}
Let $J$ be the following map: 
\begin{eqnarray*}
J: Z & \longrightarrow & L^2(\partial M) \nonumber \\
\w & \longmapsto &  \iota^* \w 
\end{eqnarray*}Its extension to $W$ exists and is an isometry by definition of $\lVert\cdot\rVert_{W}$. We prove the compactness of $I$ via the compactness of the map $J\circ I$:
\[
\begin{array}{cccccccc}
J\circ I:&V & \overset{I}{\longrightarrow} & W & \overset{J}{\longrightarrow} & L^2(\partial M) \\
&\w & \longmapsto &  \w & \longmapsto & \iota^*\w
\end{array}
\]
First, note that $V$ is included in $H^2(M)$ by the equivalence of norms. According to  \propref{prop:equivalence}, and since $H^2(M)\subset H^1(M)$, the map $J:V{\longrightarrow} H^1(M);\ \w\mapsto \w$ is continuous. Moreover, the mapping $\w \mapsto \iota^* \w$ is continuous from $H^1(M)$ to $H^{\frac{1}{2}}(\partial M)$ by the trace Theorem, and the inclusion map $H^\frac{1}{2}(\partial M)\subset L^2(\partial M)$ is compact by the Rellich-Kondrachov Theorem. Hence, $J\circ I$ can be represented as follows:

{\small
\[
\begin{array}{ccccccc}
V & {\longrightarrow} & H^1(M) & {\longrightarrow} &  H^\frac{1}{2}(\partial M)& {\longrightarrow} &  L^2(\partial M)\\
\w & \longmapsto & \w & \longmapsto & \iota^*\w& \longmapsto & \iota^*\w
\end{array}
\]
}Thus, $J\circ I$ is compact as it is composed of continuous maps, one of which is compact. To deduce that $I$ is compact, we require the following general Lemma.
\begin{lemma}\label{lem:I3=I2oI1compacte}
    Let $E, E'$ and $E''$ be three Banach spaces. Let $F:E\longrightarrow E''$ be a compact map, $G:E'\longrightarrow E''$ an isometry and $H:E\longrightarrow E'$ a map such that $F=G\circ H$. Then $H$ is compact.
\end{lemma}

\begin{proof}
     Let $(v_n)_n$ be a bounded sequence in $E$. By the compactness of $F$, there exists a subsequence $(v_{k_n})_n$ such that $F(v_{k_n}) = G(H(v_{k_n}))$ converges in $E''$. Since $G: E' \rightarrow E''$ is an isometry then $G(E')$ is a complete set in $E''$, consequently, $G(E')$ is closed in $E''$. Thus, the sequence $G(H(v_{k_n}))$ converges to $G(x)$ for some $x \in E'$. Due to the isometry property of $G$, we have $H(v_{k_n}) \rightarrow x$ in $E'$. Therefore, we can deduce that the map $H$ is compact.
\end{proof}
Consequently, $I$ is compact according to  \lemref{lem:I3=I2oI1compacte}.
\end{proof}

\begin{lemma}\label{lem:i3injective}
    The map $J\circ I$ is injective.
\end{lemma}

\begin{proof}
  Let $\omega \in V \subset H^2(M)$ such that $J\circ I(\omega) = 0$, i.e., $\iota^*\omega = 0$ in $H^{\frac{3}{2}}(\partial M)$. Let $(\omega_m)_{m \in \mathbb{N}} \in \check{Z}$ such that $\omega_m \xrightarrow[m \to \infty]{} \omega$ in $V$, which means in $H^2(M)$ according to \propref{prop:equivalence}. By \eqref{eq:IPP2}, for all $\omega' \in \Omega^p(M)$, and for all $m\in\mathbb{N}$, we have
    \begin{align*}
  0 &= \int_M \langle \Delta \omega_m, \Delta \omega' \rangle \, d\mu_g \\
    &\quad + \int_{\partial M} ( \langle \nu \lrcorner d\Delta \omega_m, \iota^*\omega' \rangle - \langle \iota^* \Delta \omega_m, \nu \lrcorner d\omega' \rangle \\
    &\quad\phantom{{} + \int_{\partial M} \left( \right.} + \langle \underbrace{\nu\lrcorner\Delta \omega_m}_{0}, \iota^*\delta\omega' \rangle - \langle\iota^*\delta\Delta\omega_m, \nu\lrcorner\omega'\rangle) d\mu_g.
\end{align*}
    
    By density, this equality remains true for all $\omega' \in H^2(M)$. Replacing $\omega'$ by $\omega$ and taking $m \to \infty$, we obtain
    \[
        0 = \int_M \langle \Delta \omega, \Delta \omega \rangle \, d\mu_g.
    \]
    Therefore, $\omega = 0$ in $V$, which proves that $J\circ I$ is injective and thus $I$ is also.
\end{proof}
To complete the proof of \thref{thm:TheoremeGeneralNBSF3PourLaPreuve}, we introduce the operator $K:V\longrightarrow V$ defined as follows:
\begin{equation}\label{eq:operateurK}
K:=D_{V}^{-1}\circ ~^tI\circ D_{W}\circ I
\end{equation}
where $D_{V}:V\longrightarrow V'$ and $D_{W}:W\longrightarrow W'$ are natural dual isomorphisms defined as follows: $D_{V}(\w):=(\cdot,\w)_{V}$ and $D_{W}(\w):=(\cdot,\w)_{W}$ for all $\w$ in $V$ and $W$ respectively, and $~^tI(\theta)=\theta\circ I$ for all $\theta$ in $W'.$ Since $V\subset W,$ we note that $K$ is characterized for all $\w$ and $\w'$ in $V$ by:
\begin{equation}\label{eq:CharacterizationK3}
    (K\w,\w')_{V}=(~^tI\circ D_{W}\circ I(\w))(\w')=(D_{W}\circ I(\w))(I(\w'))=(I(\w),I(\w'))_{W}=(\w,\w')_{W}.
\end{equation}
We observe that the operator $K$ is self-adjoint and non-negative. Moreover, since $I$ is injective we get $\ker K=\{0\}$, so $K$ is positive definite. The compactness of $K$ follows from the fact that it is the composition of continuous mappings with $I$, which is compact. Finally, since $K$ is compact, self-adjoint, and positive definite on $V$, which is a Hilbert space, there exists a sequence of real, positive, and non-increasing eigenvalues $(\chi_{i,p})_{i\geq 1}$ that tend to zero, with finite-dimensional eigenspaces, as per \cite[p.431]{LaurentSchwartz}. The eigenspaces form a dense space in $V$, and the $(\w_i)_i$ form a Hilbert basis. To conclude the proof of \thref{thm:TheoremeGeneralNBSF3PourLaPreuve} let us show the following Proposition: \index{C09@\col{$(\chi_{i,p})_{i\geq 1}$}{eigenvalues of the operator $K$}|hyperpage}

\begin{prop}\label{prop:k1w=xwNBS}
    A form $\w$ is a solution of NBS for $\ell_{i,p}\in \mathbb{R}^*$ if and only if \begin{equation}\label{eq:K1w=XwNBS}
        K\w=\frac{1}{\ell_{i,p}}\w.
    \end{equation}
\end{prop}
\begin{proof}
    Let $\w\in \Omega^p(M)$ be a solution of NBS with eigenvalue $ \ell_{i,p} \in \mathbb{R}^*$. Integrating by parts \eqref{eq:IPP2} for $ \w' \in \check{Z} $, we have:  \[
(\Delta \w, \Delta \w')_{L^2(M)} = \ell_{i,p} (\iota^* \w, \iota^* \w')_{L^2(\partial M)} .
\]  Since $\check{Z}$ is dense in $V$, we have for all $\w'\in V$, $(\w, \w')_{V} = \ell_{i,p} ( \w, \w')_{W}$. Then by definition of $K$ we get $(\w, \w')_{V} = \ell_{i,p} (K \w, \w')_{V}$ for all $\w'\in V.$ Hence $K\w=\frac{1}{\ell_{i,p}}\w$, consequently $\w$ is an eigenform for $K$ associated to $\frac{1}{\ell_{i,p}}=\chi_{i,p}$.\\
Conversely, let $\w_i$ be a $\chi_{i,p}$-eigenform in $V$ for $K$. For all $\w\in\check{Z},$ we have by definition of $K$ \begin{align}\label{eq:preuve1}
\chi_{i,p}(\Delta\w_i,\Delta \w)_{L^2(M)}\notag &= \chi_{i,p}(\w_i,\w)_{V}\\ \notag
&= (K\w_i,\w)_{V} \\ \notag
&=(I(\w_i),I(\w))_{W} \\ \notag
&= (\w_i,\w)_{W} \\
&= (\iota^*\w_i,\iota^* \w)_{L^2(\partial M)}.
\end{align}  

\begin{lemma}\label{lem:dansLaPreuveDeL'operateur}
   The map 
\begin{align*}
   \check{Z} &\longrightarrow \Omega^p(M) \\
    \omega &\longmapsto \nu \lrcorner d\Delta \omega
\end{align*}
extends continuously to $V \longrightarrow H^{-\frac{3}{2}}(\partial M)$ and satisfies for all $\omega \in V$, $\omega' \in H_{N}^2(M)$: 
    \[
    \int_M\langle \Delta\w , \Delta \w'\rangle  d\mu_g  + \int_{\partial M}  \langle \nu \lrcorner d \Delta\w, \iota^* \w'\rangle d\mu_g =0.
    \]
\end{lemma}
\begin{proof}
Using integration by parts \eqref{eq:IPP2} for all $\omega \in \Omega^p(M), \omega' \in H^2(M),$ we have:
\begin{align*}
\int_M \langle \Delta^2 \omega, \omega' \rangle \, d\mu_g &= \int_M \langle \Delta \omega, \Delta \omega' \rangle \, d\mu_g \\
&\quad + \int_{\partial M} \Big[ \langle \nu \lrcorner d\Delta \omega, \iota^* \omega' \rangle - \langle \iota^* \Delta \omega, \nu \lrcorner d\omega' \rangle \\
&\qquad \qquad + \langle \nu \lrcorner \Delta \omega, \iota^* \delta \omega' \rangle - \langle \iota^* \delta \omega, \nu \lrcorner \omega' \rangle \Big] \, d\mu_g.
\end{align*}
In particular, for all $\omega \in \check{Z},$ $\omega' \in H_N^2(M),$ we get:
\[
0 = \int_M \langle \Delta \omega, \Delta \omega' \rangle \, d\mu_g + \int_{\partial M} \langle \nu \lrcorner d\Delta \omega, \iota^* \omega' \rangle \, d\mu_g.
\] Hence the result by density of $\check{Z}$ in $V.$
\end{proof}
Using \eqref{eq:preuve1}, we obtain:
\begin{equation} \label{eq:justepourlapreuve}
    \left(\frac{1}{\chi_{i,p}}\iota^*\w_i + \nu \lrcorner d \Delta\w_i, \iota^* \w\right)_{L^2(\partial M)} = 0.
\end{equation}
Therefore, for all $\w\in Z$, \[(\nu\lrcorner d\Delta\w_i + \frac{1}{\chi_{i,p}}\iota^*\w_i,\iota^*\w)_{L^2(\partial M)}=0.\] 
\begin{lemma}\label{lm:O3surjective}
    The map 
\begin{align*}
    O : Z &\longrightarrow \Omega^p(\partial M) \\
    \omega &\longmapsto \iota^*\omega
\end{align*}
is surjective. It induces a map $O_V : V \longrightarrow \check{H}^{\frac{3}{2}}(\partial M)$ which is continuous and surjective.
\end{lemma}
\begin{proof}
    Since $V \subset H^2(M)$ by \propref{prop:equivalence}, the Trace \thref{th:trace} implies that $O$ induces a continuous map $O_V : V \longrightarrow \check{H}^{\frac{3}{2}}(\partial M)$. If $O$ is surjective, then $O_V$ is also by density. Let us show that $O$ is surjective. Let us consider the boundary-value problem:
    \begin{equation}\label{eq:bilap2}
    (biLap_2)
    \left\{
    \hspace{-0.3em}
    \begin{array}{r@{\quad}c@{\quad}l@{\quad}l}
        \Delta^2 \w &=& f & \text{on } M \\[0.15em]
        \nu \lrcorner \w &=& \w_1 & \text{on } \partial M \\[0.15em]
        \nu \lrcorner d \w &=& \w_2 & \text{on } \partial M \\[0.15em]
        \iota^* \w &=& \w_3 & \text{on } \partial M \\[0.15em]
        \nu \lrcorner \Delta \w &=& \w_4 & \text{on } \partial M.
    \end{array}
    \right.
\end{equation}
    Using the same computations of principal symbols as in \secref{sec:NBSelliptic}, this problem is elliptic in the sense of Shapiro-Lopatinskij. Moreover, repeating the proof of \lemref{lem:Bilap1}, the associated kernel $\mathcal{N}$ and co-kernel $\mathcal{N}^*$ are trivial. Therefore, by \thref{theo:thm5.4}, for all $s \ge 4$, the map
    \begin{eqnarray*}
\mathcal{P}_{biLap_2}:{H^s}(M)&\longrightarrow & H^{s-4}(M) \oplus H^{s-\frac{1}{2}}(\partial M)\oplus H^{s-\frac{1}{2}}(\partial M)\oplus H^{s-\frac{3}{2}}(\partial M)\oplus H^{s-\frac{3}{2}}(\partial M)\\
  \left[\w\right]& \longmapsto & (\Delta^2\w,\nu\lrcorner\w,\nu\lrcorner d\w,\iota^*\w,\nu\lrcorner\Delta\w)  \notag 
\end{eqnarray*}is an isomorphism. Taking $f, \omega_1, \omega_2$, and $\omega_4$ to be zero, this shows that $O$ is surjective.
\end{proof}
Since $O$ is surjective by \lemref{lm:O3surjective}, setting $\ell_{i,p}=\frac{1}{\chi_{i,p}}$ we consequently have that $$(\nu\lrcorner d\Delta\w_i + \ell_{i,p}\iota^*\w_i,\alpha)_{L^2(\partial M)}=0$$ for all $\alpha\in H^{\frac{3}{2}}(M).$ 

Now, by \thref{theo:thm5.4}, for $s\geq 4$ the map \begin{align} \mathcal{P} \colon H^{s}(M) \big/ H_{\Abs}(M) &\longrightarrow H^{s-4}(M) \times H^{s-\frac{1}{2}}(\partial M) \times H^{s-\frac{3}{2}}(\partial M) \times H^{s-\frac{5}{2}}(\partial M) \times H^{s-\frac{7}{2}}(\partial M) \nonumber\\
  [\tilde{\omega}] &\longmapsto \big( \Delta^2 \tilde{\omega}, \, \nu \lrcorner \tilde{\omega}, \, \nu \lrcorner d\tilde{\omega}, \, \nu \lrcorner \Delta\tilde{\omega}, \, \nu \lrcorner d\Delta\tilde{\omega} \big)\label{eq:PNNBSBS3}
\end{align}is a topological isomorphism onto its image, given by:
\begin{align*}
 \im(\mathcal{P}) = \bigg\{ (\omega_0, \omega_1, \omega_2, \omega_3, \omega_4) \mathrel{\bigg|} 
  &\ \forall \alpha \in H_{\Abs}(M), \\
  &\ (\omega_0, \alpha)_{L^2(M)} - (\omega_1, \underbrace{\iota^* \delta\Delta\alpha}_{0})_{L^2(\partial M)} + ( \omega_2, \underbrace{\iota^* \Delta\alpha}_{0} )_{L^2(\partial M)} \\
  &\quad + ( \omega_3, \underbrace{\iota^* \delta\alpha}_{0})_{L^2(\partial M)} + (\omega_4, \iota^* \alpha)_{L^2(\partial M)} = 0\bigg\}.
\end{align*}
Therefore, there exists a unique $\tilde{\w}\in \check{H}^5(M)$ such that
\begin{align*}
    (NBS)
    &\left\{
    \hspace{-0.3em}
    \begin{array}{r@{\quad}c@{\quad}l@{\quad}l}
        \Delta^2 \tilde{\omega} &=& 0 & \text{on } M \\[0.15em]
        \nu \lrcorner \tilde{\omega} &=& 0 & \text{on } \partial M \\[0.15em]
        \nu \lrcorner d\tilde{\omega} &=& 0 & \text{on } \partial M \\[0.15em]
        \nu \lrcorner \Delta \tilde{\omega} &=& 0 & \text{on } \partial M \\[0.15em]
        \nu \lrcorner d\Delta \tilde{\omega} &=& -\ell \, \iota^* \omega_i & \text{on } \partial M
    \end{array}
    \right.
\end{align*} we note that $\tilde{\w}$ is in $V.$ It suffices to show that $\tilde{\omega} = \omega_i$ so that $\omega_i \in H^5(M)$ is a solution to (NBS), and therefore belongs to $\Omega^p(M)$.
Using \eqref{eq:justepourlapreuve} for $\omega' \in H^2_N(M)$, we obtain:
\begin{align*}
\int_M \langle \Delta(\tilde{\omega} - \omega_i), \Delta \omega' \rangle \, d\mu_g &= - \int_{\partial M} \langle \nu \lrcorner d\Delta \tilde{\omega} - \nu \lrcorner d\Delta \omega_i, \iota^* \omega' \rangle \, d\mu_g \\
&= - \left( \int_{\partial M} \langle \ell \, \iota^* \omega_i, \iota^* \omega' \rangle \, d\mu_g - \int_{\partial M} \langle \ell \, \iota^* \omega_i, \iota^* \omega' \rangle \, d\mu_g \right) \\
&= 0.
\end{align*}
Therefore $\tilde{\omega} - \omega_i = 0$ in $V$, consequently $\w_i\in H^5(M)$ and is an $\ell_{i,p}$-solution of NBS, in particular $\w_i$ is in $\Omega^p(M)$. This ends the proof of \propref{prop:k1w=xwNBS}.
\end{proof}
 This completes the proof of \thref{thm:TheoremeGeneralNBSF3PourLaPreuve}.
\end{proof}

\section{Variational characterization of eigenvalues}\label{sec:VariationalcharacterizationofeigenvaluesNBS}

\subsection{Variational characterizations for NBS} In this section, we provide the variational characterizations associated with the eigenvalues of NBS, which we use to establish the Kuttler-Sigillito inequalities. We still consider $p\in\{1,\cdots,n-1\}$ (see \lemref{lem:p=nNBS}).

\begin{thm}\label{th:l1P}
A variational characterization of the first positive eigenvalue of problem \eqref{eq:NBSF3} is given by
\begin{align}
\label{eq:l1NBS}
\ell_{1,p} &= \inf\left\{ \frac{\lVert \Delta \w \rVert^2_{L^2(M)}}{\lVert \iota^*\w\rVert^2_{L^2(\partial M)}} \ |\ \w \in \check{H}^2_{N}(M) \ \text{and}\ \iota^*\w \neq 0 \ \text{on}\ \partial M \right\}.
\end{align}
\end{thm}
\begin{proof}
   Recall that in the proof of \thref{eq:NBSF3preuve}, we introduced the compact, self-adjoint operator $K : V \longrightarrow V$, whose discrete spectrum is given by $\left( \frac{1}{\ell_{i,p}} \right)_{i \in \mathbb{N}}$, where the $(\ell_{i,p})_{i\in \mathbb{N}}$ are the positive eigenvalues of NBS \eqref{eq:NBSF3preuve}. Therefore,
    \[
        \frac{1}{\ell_{1,p}} = \sup_{\omega \in V \setminus \{0\}} \frac{( K\omega, \omega )_V}{\|\omega\|_V^2} = \sup_{\omega \in V \setminus \{0\}} \frac{\|\iota^*\omega\|_{L^2(\partial M)}^2}{\|\Delta \omega\|_{L^2(M)}^2}
    \]  so that
    \[
        \ell_{1,p} = \inf \left\{ \frac{\|\Delta \omega\|_{L^2(M)}^2}{\|\iota^* \omega\|_{L^2(\partial M)}^2} \ \middle| \ \omega \in V \setminus \{0\} \right\}.
    \]
    By \propref{prop:equivalence}, since $V$ is the closure of $Z$ for the $H^2(M)$-norm, we have
    \begin{align*}
        \ell_{1,p} = \inf \bigg\{ \frac{\|\Delta \omega\|_{L^2(M)}^2}{\|\iota^* \omega\|_{L^2(\partial M)}^2} \ \bigg| \ \omega \in \Omega^p(M), \ &\Delta^2 \omega = 0 \text{ on } M, \\
        &\nu \lrcorner \omega = 0, \ \nu \lrcorner d\omega = 0, \ \nu\lrcorner\Delta\w = 0 \text{ on } \partial M \bigg\}.
    \end{align*}
    To obtain \thref{th:l1P}, we need to drop the conditions $\Delta^2 \omega = 0$ on $M$ and $\nu \lrcorner \Delta \omega = 0$ on $\partial M$. This relies on the following Lemma:
    
    \begin{lemma}\label{lem:Hn=V+H0}
         We have the following orthogonal decomposition for $(\cdot, \cdot)_V$:
\[
    \check{H}^2_N(M) = V \oplus ({H}^1_0  \cap \check{H}^2_N )(M).
\] \end{lemma}

\begin{proof}
    Note that by \lemref{lem:psNBS}, $(\cdot, \cdot)_V$ is indeed a scalar product on $\check{H}^2_N(M) $. Let $\omega \in \check{Z}$ and $\omega' \in \check{H}^2_N(M) $, then, by \eqref{eq:IPP2}, we have
    \begin{align*}
        (\omega, \omega')_V &= \int_M \langle \Delta \omega, \Delta \omega' \rangle \, d\mu_g \\
        &= - \int_{\partial M} \langle \nu \lrcorner d\Delta \omega, \iota^*\omega' \rangle\, d\mu_g.
    \end{align*}
    Therefore, if $\omega' \in {H}^1_0(M)  \cap \check{H}^2_N(M) $, we have $(\omega, \omega')_V = 0$. Hence, 
    \[
        {H}^1_0(M)  \cap \check{H}^2_N(M) \subset V^{\perp_{(\cdot,\cdot)_V}} \quad \text{in } \check{H}^2_N(M).
    \]

    Conversely, let us take $\omega' \in \check{H}^2_N(M) $ such that for all $\omega \in V$, $(\omega, \omega')_V = 0$. Let $(\omega_i)_{i \in \mathbb{N}}$ be an orthonormal Hilbert basis of $V$ where $\omega_i$ is the eigenform of $K$ associated to $\chi_{i,p} = \frac{1}{\ell_{i,p}}$. For all $i \in \mathbb{N}$, we have
    \[
        (\omega_i, \omega')_V = 0 \iff (\Delta \omega_i, \Delta \omega')_{L^2(M)} = 0.
    \]
    Therefore, by \lemref{lem:dansLaPreuveDeL'operateur},
    \[
        (\nu \lrcorner d\Delta \omega_i, \iota^*\omega')_{L^2(\partial M)} = 0 \iff (\ell_{i,p} \iota^* \omega_i, \iota^* \omega')_{L^2(\partial M)} = 0.
    \]
    Since this holds for all $i \in \mathbb{N}$, by the density of $(\omega_i)_{i\in\mathbb{N}}$ in $V$, we have for all $\omega \in V$,
    \[
        (\iota^*\omega, \iota^*\omega')_{L^2(\partial M)} = 0.
    \]
    Moreover, by \lemref{lm:O3surjective}, the map $O : \omega \mapsto \iota^*\omega$ is surjective from $V$ to $\check{H}^{\frac{3}{2}}(\partial M)$. Since $\omega' \in \check{H}^{2}_N(M)$, we have $\iota^*\omega' \in \check{H}^{\frac{3}{2}}(\partial M)$, therefore, $\iota^*\omega'=0.$ Hence, $\omega' \in {H}^1_0(M) \cap \check{H}^2_N(M)$. This concludes the proof of \lemref{lem:Hn=V+H0}.
\end{proof}
 Let us go back to the proof of \thref{th:l1P}. Recall that in \secref{sec:weak}, for all $\omega \in \check{H}^2_N(M)$ with $\iota^*\omega \neq 0$, we write
\[
    \mathcal{Q}(\omega) := \frac{\|\Delta \omega\|_{L^2(M)}^2}{\|\iota^*\omega\|_{L^2(\partial M)}^2}.
\]

Let $\omega \in \check{H}^2_N(M)$ with $\iota^*\omega \neq 0$. By \lemref{lem:Hn=V+H0}, there exist $u \in V$ and $v \in {H}^1_0(M) \cap \check{H}^2_N(M)$ such that $(\Delta u, \Delta v)_{L^2(M)} = 0$ and $\omega = u + v$. Therefore, since $\iota^*v = 0$ (as $v \in {H}^1_0(M)$), we have
    \begin{align*}
        \mathcal{Q}(\omega) &= \frac{\|\Delta u + \Delta v\|_{L^2(M)}^2}{\|\iota^*u + \iota^*v\|_{L^2(\partial M)}^2} \\
        &= \frac{\|\Delta u\|_{L^2(M)}^2 + \|\Delta v\|_{L^2(M)}^2}{\|\iota^*u\|_{L^2(\partial M)}^2} \\
        &\ge \mathcal{Q}(u).
    \end{align*}
    
    Consequently,
    \[
        \inf \left\{ \mathcal{Q}(\omega) \ \middle| \ \omega \in \check{H}^2_N(M), \ \iota^*\omega \neq 0 \right\} \ge \inf \left\{ \mathcal{Q}(u) \ \middle| \ u \in V, \ \iota^*u \neq 0 \right\}.
    \]
    
    Since the reverse inequality is immediate, we obtain
    \[
        \ell_{1,p} = \inf \left\{ \mathcal{Q}(\omega) \ \middle| \ \omega \in \check{H}^2_N(M), \ \iota^*\omega \neq 0 \right\}
    \]
    which concludes the proof of \thref{th:l1P}.
   \end{proof}\begin{remark}
       It follows from the proof that the infimum in \thref{th:l1P} is attained by the eigenforms of NBS associated to $\ell_{1,p}$ and all their combinations with forms of $(H_0^1\cap \check{H}^2_N)(M)$.
   \end{remark}We now give a characterization of higher eigenvalues. Although the following Proposition is classical, we provide its proof for completeness.
\begin{prop}\label{prop:l_k_2}
For every $ k \in \mathbb{N} $, we have the following characterization for the $ k^{\text{th}} $ positive eigenvalue of NBS:
\[
\ell_{k,p} = \min_{\substack{\mathcal{V}_k \subset V \\ \dim \mathcal{V}_k = k}} \max_{0 \neq \w \in \mathcal{V}_k} \frac{\lVert \Delta \w \rVert^2_{L^2(M)}}{\lVert \iota^* \w \rVert^2_{L^2(\partial M)}}.
\]
\end{prop}
We rely on the following immediate Lemma:
\begin{lemma}\label{lem: v}
Let $H $ be a Hilbert space, and let $V$ and $ E_{k-1} $ be subspaces of $ H $ of dimension $ k $ and $ k-1 $ respectively, where $k\geq 1$. Then $ V \cap E_{k-1}^{\perp} \neq \{0\} $.
\end{lemma}

\begin{proof}[\textbf{Proof of \propref{prop:l_k_2}}]
Let $ (\w_i)_{i \geq 1} $ be a Hilbert basis of $ (V, (\cdot, \cdot)_{V}) $, with $ \w_i $ being $\ell_{i,p}$-eigenforms. According to \lemref{lem: v}, for $ \mathcal{V}_k \subset V $ such that $ \dim \mathcal{V}_k = k $, there exists a non-zero $ \w \in \mathcal{V}_k $ such that $ (\w, \w_i)_{V} = 0 $ for all $ i = 1, \cdots, k-1 $. Thus, $ \w = \underset{i \geq k}{\sum}(\w, \w_i)_{V} \w_i $ and we have:
\[
\lVert \w \rvert^2_{V} = \sum_{i \geq k} (\w, \w_i)_{V}^2,
\]
and
\[
\lVert \w \rVert_{W}^2 = (K \w, \w)_{V} = \sum_{i \geq k} \frac{1}{\ell_{i,p}} (\w, \w_i)_{V}^2 \leq \frac{1}{\ell_{k,p}} \sum_{i \geq 1} (\w, \w_i)_{V}^2.
\]
Thus, we conclude that $\ell_{k,p} \leq \frac{\lVert \w \rvert^2_{V}}{\lVert \w \rVert_{W}^2}.$ Therefore, $\ell_{k,p} \leq \underset{0 \neq \w \in \mathcal{V}_k}{\sup} 
  \frac{\lVert \w \rvert^2_{V}}{\lVert \w \rVert_{W}^2}.$ This implies that
\[
\ell_{k,p} \leq \min_{\substack{\mathcal{V}_k \subset V \\ \dim \mathcal{V}_k = k}} \max_{0 \neq \w \in \mathcal{V}_k} \frac{\lVert \Delta \w \rVert^2_{L^2(M)}}{\lVert \iota^* \w \rVert^2_{L^2(\partial M)}}.
\]

To show the second inequality, we define $ \overline{V}_k = \spann(\w_1, \cdots, \w_k) \subset V $. For $ \w \in \overline{V}_k $, we have: 
\[
\lVert \w \rvert^2_{V} = \sum_{i=1}^k (\w, \w_i)_{V}^2
\]
and
\[
\lVert \w \rVert_{W}^2 = (K \w, \w)_{V} = \sum_{i=1}^k \frac{1}{\ell_{i,p}} (\w, \w_i)_{V}^2 \geq \frac{1}{\ell_{k,p}} \sum_{i=1}^k (\w, \w_i)_{V}^2.
\]
Thus, $\frac{\lVert \w \rvert^2_{V}}{\lVert \w \rVert_{W}^2} \leq \ell_{k,p},$
therefore, $\underset{0 \neq \w \in \mathcal{V}_k}{\sup} \frac{\lVert \w \rvert^2_{V}}{\lVert \w \rVert_{W}^2} \leq \ell_{k,p}.$
This concludes the proof.
\end{proof}

\begin{thm}\label{theo:lk}
For all $ k \in \mathbb{N} $, we have the following characterization of the $ k^{\text{th}} $ positive eigenvalue:
    \begin{equation}\label{eq:lkNBS}
  \ell_{k,p} = \min_{\substack{\mathcal{V}_k(M) \subset \check{H}^2_{N}(M)  \\ \dim (\mathcal{V}_k / (H_0^1(M) \cap \mathcal{V}_k)) = k}} \ \underset{\w \in \mathcal{V}_k \backslash H_0^2(M)}{\max} \frac{\lVert \Delta \w \rVert^2_{L^2(M)}}{\lVert \iota^*\w \rVert^2_{L^2(\partial M)}}.  
    \end{equation}
 
\end{thm} 

\begin{proof}
   We have from \lemref{lem:Hn=V+H0} the following orthogonal decomposition $\check{H}^2_{N}(M)$ $= V \oplus H_0^1(M)\cap\check{H}^2_{N}(M)$. Let $\mathcal{V}_k\subset \check{H}^2_{N}(M)$ such that $\dim (\mathcal{V}_k/H_0^1(M)\cap \mathcal{V}_k)=k$, and let $p_{V}:\check{H}^2_{N}(M)\longrightarrow V$ be the orthogonal projection with respect to the inner product $(\cdot,\cdot)_{V}$. We define $\mathcal{U}:=p_{V}(\mathcal{V}_k)$ hence $\mathcal{U}\subset V$ and $\dim\mathcal{U}=k$. Then, by \propref{prop:l_k_2}, we have $\ell_{k,p}\leq \underset{0\neq u\in \mathcal{U}}{\sup}\frac{\lVert u\rvert^2_{V}}{\lVert u \rVert_{W}^2}$. Now, let $\w\in \mathcal{V}_k\backslash H_0^1(M)$, we decompose $\w$ as follows: $\w=p_{V}(\w)+\w-p_{V}(\w)$ where $p_{V}(\w)\in \mathcal{U}\subset V$ and $\w-p_{V}(\w)\in V^{\perp_{(\cdot,\cdot)_{V}}}\subset \mathcal{U}^{\perp_{(\cdot,\cdot)_{V}}}$. Thus, we have $ \lVert \w \rVert_{V}^2=\lVert p_{V}(\w)\rVert_{V}^2+\lVert\w-p_{V}(\w)\rVert_{V}^2$ and since $\w-p_{V}(\w)\in H_0^1(M)$, we get $\lVert\w-p_{V}(\w)\rVert_{W}^2=0$ so $\lVert\w\rVert_{W}^2=\lVert p_{V}(\w)\rVert_{W}^2+\lVert\w-p_{V}(\w)\rVert_{W}^2$.

    We now have:
    \[
    \frac{ \lVert \w \rVert_{V}^2}{\lVert\w\rVert_{W}^2} = \frac{\lVert p_{V}(\w)\rVert_{V}^2 + \lVert\w-p_{V}(\w)\rVert_{V}^2}{\lVert p_{V}(\w)\rVert_{W}^2} \geq \frac{\lVert p_{V}(\w)\rVert_{V}^2}{\lVert p_{V}(\w)\rVert_{W}^2}.
    \]
    Therefore,
    \[
    \underset{0\neq \w\in \mathcal{V}_k\backslash H_0^1(M)}{\sup}\frac{\lVert \w\rVert_{V}^2}{\lVert \w \rVert_{W}^2} \geq \underset{0\neq \w\in \mathcal{V}_k\backslash H_0^1(M)}{\sup}\frac{\lVert p_{V}(\w)\rVert_{V}^2}{\lVert p_{V}(\w)\rVert_{W}^2} = \underset{ u\in p_{V}(\mathcal{V}_k)}{\sup}\frac{\lVert u\rVert_{V}^2}{\lVert u \rVert_{W}^2} \geq \ell_{k,p}.
    \]

    This implies that for all $k\in\mathbb{N}$, we have
    \[
    \ell_{k,p}\leq\min_{\substack{\mathcal{V}_k\subset \check{H}^2_{N}(M) \\ \dim (\mathcal{V}_k/H_0^1(M)\cap \mathcal{V}_k)=k}}\ \underset{u\in \mathcal{V}_k\backslash H_0^1(M)}{\max}\frac{\lVert u\rVert_{V}^2}{\lVert u \rVert_{W}^2}.
    \]
    Finally, let $\overline{V}_k=\spann(\w_1,\cdots,\w_k)\subset V$. Since $\overline{V}_k$ is a subspace of $V$ orthogonal to $H_0^1(M)$, we have $\overline{V}_k\cap H_0^1(M)=\{0\}$. Therefore,
    \[
    \min_{\substack{\mathcal{V}_k\subset \check{H}^2_{N}(M) \\ \dim (\mathcal{V}_k/H_0^2(M)\cap \mathcal{V}_k)=k}}\ \underset{\w\in \mathcal{V}_k\backslash H_0^1(M)}{\max}\frac{\lVert \w\rVert_{V}^2}{\lVert \w\rVert_{W}^2} \leq \underset{u\in \mathcal{V}_k\backslash H_0^1(M)}{\sup}\frac{\lVert u\rVert_{V}^2}{\lVert u\rVert_{W}^2} = \ell_{k,p}.
    \]
    This concludes the proof.
\end{proof}
To prove Kuttler-Sigillito inequalities, we need the following alternative characterizations.
\begin{thm}\label{thm:l'minimumtheorem}
Let us define \begin{align}\label{eq:tilde=l}
\tilde{\ell}_{1,p} &= \inf\bigg\{ \frac{\lVert \nu\lrcorner d\omega \rVert^2_{L^2(\partial M)}}{\lVert \omega \rVert^2_{L^2( M)}} \mathrel{\bigg|} \omega\in\Omega^p(M),\ \Delta \omega = 0\ \text{on}\ M, \nonumber \\
&\qquad\qquad \nu \lrcorner \omega = 0\ \text{on}\ \partial M \ \text{and} \ \omega \perp_{L^2(M)} H_{\Abs}(M) \bigg\}.
\end{align}\index{L02@\col{$\tilde{\ell}_{k,p}$}{NBS eigenvalue characterization}|hyperpage}We have then $\tilde{\ell}_{1,p}=\ell_{1,p}$.
\end{thm}The characterization above follows the Fichera’s duality principle \cite{BuosoFalcoGonzalezMiranda, Fichera, Kuttler&Sigillito}.
\begin{proof}
 Let $\omega_{1,p}$ be an $\ell_{1,p}$-eigenform of the NBS operator, with $\lVert \iota^* \omega_{1,p} \rVert_{L^2(\partial M)} = 1$. We have $\omega_{1,p} \in \check{Z} \subset \Omega^p(M) \cap \check{H}_N^2(M)$. Let us write $\hat{\omega} = \Delta \omega_{1,p}$, it is a solution of the following boundary-value problem:
 
  \begin{equation*}
  \left\{
  \begin{alignedat}{3}
    \Delta \hat{\omega}          &= \Delta^2 \omega_{1,p}             &&= 0                                        &&\quad\text{on } M \\
    \nu \lrcorner \hat{\omega}    &= \nu \lrcorner \Delta \omega_{1,p} &&= 0                                        &&\quad\text{on } \partial M \\
    \nu \lrcorner d\hat{\omega}   &= \nu \lrcorner d\Delta \omega_{1,p}&&= -\ell_{1,p} \iota^* \omega_{1,p}        &&\quad\text{on } \partial M.
  \end{alignedat}
  \right.
\end{equation*}

 Let us write $\hat{\omega} = \hat{\omega}_{\Abs} + \hat{\omega}'$, where $\hat{\omega}_{\Abs} \in H_{\Abs}(M)$ and $\hat{\omega}' \perp_{L^2(M)} H_{\Abs}(M)$, with:
  \begin{equation*}
    \Delta \hat{\omega}' = \Delta \hat{\omega} = \Delta^2 \w_{1,p} = 0, \quad d\hat{\omega}' = d\hat{\omega}, \quad \text{and} \quad \nu \lrcorner \hat{\omega}' = \nu \lrcorner \hat{\omega} = 0.
  \end{equation*}
 By \eqref{eq:IPP1}, we have:
  \begin{align*}
    \int_M \langle \hat{\omega}', \Delta\omega_{1,p} \rangle d\mu_g
    &= -\int_{\partial M} \langle \nu \lrcorner d\hat{\omega}', \iota^* \omega_{1,p} \rangle d\mu_g\\
    &= \ell_{1,p} \underbrace{\lVert \iota^* \omega_{1,p} \rVert_{L^2(\partial M)}^2}_{1}.
  \end{align*}
  Thus, by the Cauchy-Schwarz inequality:
  \begin{equation*}
    \lVert \hat{\omega}' \rVert_{L^2(M)} \times \underbrace{\lVert \Delta\omega_{1,p} \rVert_{L^2(M)}}_{\ell_{1,p}^{\frac{1}{2}}} \ge \int_M \langle \hat{\omega}', \Delta\omega_{1,p} \rangle d\mu_g= \ell_{1,p}
  \end{equation*}
  which yields:
  \begin{equation*}
    \lVert \hat{\omega} \rVert_{L^2(M)} \ge \lVert \hat{\omega}' \rVert_{L^2(M)} \ge \ell_{1,p}^{\frac{1}{2}}.
  \end{equation*}
 Therefore:
  \begin{equation*}
    \frac{\lVert \nu \lrcorner d\hat{\omega} \rVert_{L^2(\partial M)}^2}{\lVert \hat{\omega} \rVert_{L^2(M)}^2} 
    = \frac{\ell_{1,p}^2 \overbrace{\lVert \iota^* \omega_{1,p} \rVert_{L^2(\partial M)}^2}^{1}}{\lVert \hat{\omega} \rVert_{L^2(M)}^2} 
    \le \frac{\ell_{1,p}^2}{\ell_{1,p}} 
    = \ell_{1,p}.
  \end{equation*}
  Hence:
  \begin{equation*}
  \tilde{\ell}_{1,p} \le \ell_{1,p}.
  \end{equation*}

Conversely, let $(\omega_m)_{m \ge 0} \in \Omega^p(M)^{\mathbb{N}}$ be a sequence of smooth non-zero $p$-forms such that for all $m\geq 0$:
  \begin{equation*}
    \Delta\omega_m = 0\ \text{on}\ M, \quad \nu \lrcorner \omega_m = 0\ \text{on}\ \partial M,\quad \lVert \omega_m \rVert_{L^2(M)} = 1, \quad \omega_m \perp_{L^2(M)} H_{\Abs}(M)
  \end{equation*}
  and
  \begin{equation*}
    \frac{\lVert \nu \lrcorner d\omega_m \rVert_{L^2(\partial M)}^2}{\lVert \omega_m \rVert_{L^2(M)}^2} \xrightarrow[m \to \infty]{} \tilde{\ell}_{1,p}.
  \end{equation*}

  For all $m \ge 0$, by \lemref{lm:SteklovFormeGénéraleCorrigée}, there exists a unique $\hat{\omega}_m \perp_{L^2(M)} H_{\Abs}(M)$ such that:
  \begin{equation*}
  \left\{
  \begin{alignedat}{2}
    \Delta \hat{\omega}_m        &= \omega_m &&\quad\text{on } M \\
    \nu \lrcorner \hat{\omega}_m  &= 0        &&\quad\text{on } \partial M \\
    \nu \lrcorner d\hat{\omega}_m &= 0        &&\quad\text{on } \partial M.
  \end{alignedat}
  \right.
\end{equation*}
Again, by integration by parts \eqref{eq:IPP1},
\begin{equation*}
  \lVert \omega_m \rVert_{L^2(M)}^2 =  \lVert \Delta\hat{\omega}_m \rVert_{L^2(M)}^2 = \int_M \langle \Delta \hat{\omega}_m, \omega_m \rangle d\mu_g = -\int_{\partial M} \langle \iota^* \hat{\omega}_m, \nu \lrcorner d\omega_m \rangle d\mu_g.
\end{equation*}
Since $\lVert \Delta \hat{\omega}_m \rVert_{L^2(M)}^2 = \lVert \omega_m \rVert_{L^2(M)}^2 = 1$, we deduce by the Cauchy-Schwarz inequality:

\begin{equation*}
  1 = \lVert \omega_m \rVert_{L^2(M)}^2 \le \lVert \iota^* \hat{\omega}_m \rVert_{L^2(\partial M)}^2 \times \lVert \nu \lrcorner d\omega_m \rVert_{L^2(\partial M)}^2.
\end{equation*}

Therefore,
\begin{equation*}
  \ell_{1,p} \le \frac{\lVert \Delta \hat{\omega}_m \rVert_{L^2(M)}^2}{\lVert \iota^* \hat{\omega}_m \rVert_{L^2(\partial M)}^2} = \frac{1}{\lVert \iota^* \hat{\omega}_m \rVert^2} \le \lVert \nu \lrcorner d\omega_m \rVert_{L^2(\partial M)}^2 \xrightarrow[m \to \infty]{} \tilde{\ell}_{1,p}.
\end{equation*}
Thus:
\begin{equation*}
  \ell_{1,p} \le \tilde{\ell}_{1,p}.
\end{equation*}
\end{proof}
\begin{prop}\label{prop:minatteint}
Let $\omega \in \Omega^p(M)$ such that $\Delta \omega = 0$ on $M$, $\nu \lrcorner \omega = 0$ on $\partial M$, and $\omega \perp_{L^2(M)} H_{\Abs}(M)$. The infimum is attained at $\tilde{\ell}_{1,p}$ by $\omega = \Delta \hat{\omega}$, where $\hat{\omega}$ is an $\ell_{1,p}$-eigenform.
\end{prop}
\begin{proof}
In \thref{thm:l'minimumtheorem}, we showed that $\ell_{1,p} = \tilde{\ell}_{1,p}$. We now show that if
\begin{equation*}
\frac{\lVert \nu \lrcorner d\omega \rVert_{L^2(\partial M)}^2}{\lVert \omega \rVert_{L^2(M)}^2} = \ell_{1,p}
\end{equation*}
then $\omega = \Delta \hat{\omega}$, where $\hat{\omega}$ is an $\ell_{1,p}$-eigenform. By \lemref{lm:SteklovFormeGénéraleCorrigée}, there exists $\hat{\omega} \in \Omega^p(M)$ satisfying:\begin{equation}\label{eq:SteklovFormeGeneral}
    \left\{
    \hspace{-0.3em}
    \begin{array}{r@{\quad}c@{\quad}l@{\quad}l}
        \Delta \hat{\omega} &=& \omega & \text{on } M \\[0.15em]
        \nu \lrcorner \hat{\omega} &=& 0 & \text{on } \partial M \\[0.15em]
        \nu \lrcorner d\hat{\omega} &=& 0 & \text{on } \partial M \\[0.15em]
        \multicolumn{4}{l}{\hat{\omega} \perp_{L^2(M)} H_{\Abs}(M).}
    \end{array}
    \right.
\end{equation}

Then, using integration by parts \eqref{eq:IPP1}:
\begin{align*}
  \lVert \omega \rVert_{L^2(M)}^2 
  &= ( \omega, \Delta \hat{\omega})_{L^2(M)} \\
  &= -(\iota^* \hat{\omega}, \nu \lrcorner d\omega)_{L^2(\partial M)} \\
  &\le \lVert \iota^* \hat{\omega} \rVert_{L^2(\partial M)} \lVert \nu \lrcorner d\omega \rVert_{L^2(\partial M)} \\
  &=\ell_{1,p}^{\frac{1}{2}} \lVert \omega \rVert_{L^2(M)} \lVert \iota^* \hat{\omega} \rVert_{L^2(\partial M)}.
\end{align*}

This implies:
\begin{align*}
  & \lVert \omega \rVert_{L^2(M)}^2 \le \ell_{1,p}^{\frac{1}{2}} \lVert \iota^* \hat{\omega} \rVert_{L^2(\partial M)}\implies \frac{\lVert \Delta \hat{\omega} \rVert_{L^2(M)}^2}{\lVert \iota^* \hat{\omega} \rVert_{L^2(\partial M)}^2} \le \ell_{1,p}.
\end{align*} 
Therefore, $\hat{\omega} \text{ is an } \ell_{1,p}\text{-eigenform}$ satisfying $\w=\Delta\hat{\omega}$.
\end{proof}
\begin{prop}\label{prop:check<l}
    Let us define: 
\begin{align}
  \check{\ell}_{1,p} &= \inf \Bigg\{ \frac{\lVert \nu \lrcorner d\omega \rVert_{L^2(\partial M)}^2}{\lVert \omega \rVert_{L^2(M)}^2} \mathrel{\Bigg|} \omega \in \Omega^p(M), \ \Delta\omega = 0 \text{ on } M, \nonumber \\
  &\qquad\qquad \nu \lrcorner \omega = 0 \text{ on } \partial M \quad\text{and}\quad \omega \perp_{L^2(\partial M)} H_{\Abs}(M) \Bigg\}. \label{eq:lTilde<l}
\end{align}
\index{L03@\col{$\check{\ell}_{k,p}$}{NBS eigenvalue characterization}|hyperpage}We have $\check{\ell}_{1,p}\leq\tilde{\ell}_{1,p}=\ell_{1,p}$. \end{prop} 
\begin{proof}
    Let $\hat{\omega} \in \Omega^p(M)$ such that $\nu \lrcorner \hat{\omega} = 0$ on $\partial M$, $\Delta\hat{\omega} = 0$ on $M$, and $\hat{\omega} \perp_{L^2(M)} H_{\Abs}(M)$. We have the following decomposition: 
\begin{equation*}
  \iota^* \hat{\omega} = \iota^* \hat{\omega}_{\Abs} + \beta \quad \text{such that} \quad \langle \iota^* \hat{\omega}_{\Abs}, \beta \rangle_{L^2(\partial M)} = 0.
\end{equation*}
We now use this intermediate Lemma, whose proof follows the same lines as that of \lemref{lm:SteklovFormeGénéraleCorrigée} by using \thref{theo:thm5.4}.
\begin{lemma}\label{lm:deltadir}
    The map
\begin{align*}
  (\mathcal{P}_{\Delta_{Dir}}) \colon H^2(M) &\longrightarrow L^2(M) \times H^{\frac{3}{2}}(\partial M) \times H^{\frac{3}{2}}(M) \\
  \omega &\longmapsto (\Delta\omega, \, \nu \lrcorner \omega, \, \iota^* \omega)
\end{align*}
is a topological isomorphism.
\end{lemma}
Let $\hat{\omega}_2$ be the form satisfying
\begin{equation*}
  \left\{
  \begin{alignedat}{2}
    \Delta \hat{\omega}_2       &= 0     &&\quad\text{on } M\\
    \nu \lrcorner \hat{\omega}_2 &= 0     &&\quad\text{on } \partial M \\
    \iota^* \hat{\omega}_2           &= \beta &&\quad\text{on } \partial M.
  \end{alignedat}
  \right.
\end{equation*}
given by the \lemref{lm:deltadir}. By the injectivity of $(\mathcal{P}_{\Delta_{Dir}})$, we have $\hat{\omega} = \hat{\omega}_{\Abs} + \hat{\omega}_2$. Moreover, since $\iota^* \hat{\omega}_2 \perp_{L^2(\partial M)} \iota^* H_{\Abs}(M)$ and $d\hat{\omega}_2 = d\hat{\omega}$, we obtain:
\begin{align*}
  \check{\ell}_{1,p} &\le \frac{\lVert \nu \lrcorner d\hat{\omega}_2 \rVert_{L^2(\partial M)}^2}{\lVert \hat{\omega}_2 \rVert_{L^2(M)}^2} \\[0.5em]
  &= \frac{\lVert \nu \lrcorner d\hat{\omega} \rVert_{L^2(\partial M)}^2}{\lVert \hat{\omega} - \hat{\omega}_{\Abs} \rVert_{L^2(M)}^2} \\[0.5em]
  &= \frac{\lVert \nu \lrcorner d\hat{\omega} \rVert_{L^2(\partial M)}^2}{\lVert \hat{\omega} \rVert_{L^2(M)}^2 + \lVert \hat{\omega}_{\Abs} \rVert_{L^2(M)}^2} \\[0.5em]
  &\le \frac{\lVert \nu \lrcorner d\hat{\omega} \rVert_{L^2(\partial M)}^2}{\lVert \hat{\omega} \rVert_{L^2(M)}^2}
\end{align*}
for all $\hat{\omega} \in \Omega^p(M)$ such that $\nu \lrcorner \hat{\omega} = 0$, $\Delta\hat{\omega} = 0$, and $\hat{\omega} \perp_{L^2(M)} H_{\Abs}(M)$.
Hence,
\begin{equation*}
 \check{\ell}_{1,p} \le \tilde{\ell}_{1,p} = \ell_{1,p}.
\end{equation*}
\end{proof}
\begin{remark}
    It is not clear whether $\check{\ell}_{1,p}=\ell_{1,p}$, but we will not need this equality in the sequel.
\end{remark}
\subsection{Other NBS problems}\label{sec:otherNBS}
In a first attempt to generalize the Kuttler-Sigillito inequalities to differential forms, we first studied the boundary-value problem NBS3 defined below, as the natural analogue of the DBS problem introduced in \cite{BSFidaGeorgeOlaNicolas}. However, the variational characterizations of this problem did not allow for a generalization of the Kuttler-Sigillito inequalities. Consequently, we modified the boundary conditions to introduce NBS2. While its variational formulation enabled us to establish the first two inequalities of \eqref{eq:ineqKuttlerSigillito} and inequalities \eqref{eq:firstineq}, it remained insufficient for the other inequalities. That is why we further refined the boundary conditions, leading to the current version, NBS. This formulation proved to be the most effective, as it both improved upon the results obtained via NBS2 and provided the necessary framework to establish all our desired results. Throughout the previous sections, we have studied the NBS problem \eqref{eq:NBSF3preuve} and its properties. We now present the two other NBS problems with different boundary conditions. We also give the associated variational characterizations of their eigenvalues, without providing proofs, as they follow the same methodology. In this section, $\ell^{NBS2}$ and $\ell^{NBS3}$ denote the eigenvalues of the NBS2 and NBS3 problems, respectively, defined below.

\begin{thm}\label{thm:TheoremeGeneralNBSF2}
 For $p\in\{1,\cdots,n-1\}$, the following boundary-value problem
\begin{equation}\label{eq:NBSF2}
    (NBS2)
    \left\{
    \hspace{-0.3em}
    \begin{array}{r@{\quad}c@{\quad}l@{\quad}l}
        \Delta^2 \w &=& 0 & \text{on } M \\[0.15em]
        \nu \lrcorner \w &=& 0 & \text{on } \partial M \\[0.15em]
        \nu \lrcorner d\w &=& 0 & \text{on } \partial M \\[0.15em]
        \iota^* \delta \w &=& 0 & \text{on } \partial M \\[0.15em]
        \nu \lrcorner d\Delta \w + \ell^{NBS2} \iota^* \w &=& 0 & \text{on } \partial M
    \end{array}
    \right.
\end{equation}
on $p$-forms, has a discrete spectrum consisting of an unbounded non-decreasing sequence of positive real eigenvalues with finite multiplicities $(\ell^{NBS2}_{i,p})_{i \geq 1}$ and possibly $\ell^{NBS2}_{0,p}=0$. Moreover, its kernel is $H_{\Abs}(M)$.
\end{thm}
The proof of this Theorem is similar to the one of \thref{thm:TheoremeGeneralNBSF3PourLaPreuve}. \begin{remark}
By definition, the spectrum of NBS2 is the spectrum of the operator 
\[
\begin{aligned}
NBS2^{\text{form}}: \Omega^{p}(\partial M) &\longrightarrow \Omega^{p}(\partial M) \\
\omega_1 &\longmapsto -\nu \lrcorner d\Delta\hat{\omega}
\end{aligned}
\]
where $\hat{\omega}$ is the smooth solution of\begin{equation}
    \left\{
    \hspace{-0.3em}
    \begin{array}{r@{\quad}c@{\quad}l@{\quad}l}
        \Delta^2 \hat{\omega} &=& 0 & \text{on } M \\[0.15em]
        \nu \lrcorner \hat{\omega} &=& 0 & \text{on } \partial M \\[0.15em]
        \nu \lrcorner d\hat{\omega} &=& 0 & \text{on } \partial M \\[0.15em]
        \iota^* \delta\hat{\omega} &=& 0 & \text{on } \partial M \\[0.15em]
        \iota^* \hat{\omega} &=& \omega_1 & \text{on } \partial M
    \end{array}
    \right.
\end{equation}
The existence of such a unique $\hat{\omega}$ is a consequence of \lemref{lem:Bilap1}.
\end{remark}
\begin{prop}\label{prop:NBS2VPk}
For any $k \in \mathbb{N}$, the following characterization of the $k^{\text{th}}$ strictly positive eigenvalue holds:
\begin{equation*}
\ell^{NBS2}_{k,p} = \min_{\substack{\mathcal{V}_k \subset \check{H}_{N,2}^2(M)  \\ \dim (\mathcal{V}_k/(H_0^2(M) \cap \mathcal{V}_k) = k}}\ \max_{\w \in \mathcal{V}_k \backslash H_0^2(M)}\frac{\lVert \Delta \w \rVert^2_{L^2(M)}}{\lVert \iota^*\w \rVert^2_{L^2(\partial M)}}.
\end{equation*} Where \begin{equation}\label{eq:H2Ncheck}
    \check{H}_{N,2}^2(M) := \{\w \in H^2(M)\ |\ \nu\lrcorner d\w = 0,\ \nu\lrcorner \w =0,\ \iota^*\delta\w = 0\ \text{on} \ \partial M \ \text{and}\ \w \perp_{L^2(\partial M)} H_{\Abs}(M)\}.
\end{equation}
\end{prop}

\begin{proof}
The proof is identical to that of \thref{theo:lk}.
\end{proof}
The NBS2 version allows us to establish analogues of \eqref{eq:3èmePBmuksigma1<lk}, \eqref{eq:3èmePBmu1sigmak<lk}, and \eqref{eq:q1sigma1<l1}. For these same inequalities, using NBS improves the results.

\begin{thm}\label{thm:TheoremeGeneralNBSF1}
 For $p\in\{1,\cdots,n-1\}$, the following boundary-value problem
\begin{equation}\label{eq:NBSF1}
    (NBS3)
    \left\{
    \hspace{-0.3em}
    \begin{array}{r@{\quad}c@{\quad}l@{\quad}l}
        \Delta^2 \w &=& 0 & \text{on } M \\[0.15em]
        \nu \lrcorner \w &=& 0 & \text{on } \partial M \\[0.15em]
        \nu \lrcorner d\w &=& 0 & \text{on } \partial M \\[0.15em]
        \nu \lrcorner \Delta \w + \ell^{NBS3} \iota^* \delta \w &=& 0 & \text{on } \partial M \\[0.15em]
        \nu \lrcorner d\Delta \w + \ell^{NBS3} \iota^* \w &=& 0 & \text{on } \partial M
    \end{array}
    \right.
\end{equation}
for differential p-forms, has a discrete spectrum consisting of an unbounded non-decreasing sequence of positive real eigenvalues with finite multiplicities $(\ell^{NBS3}_{i,p})_{i \geq 1}$ and possibly $\ell^{NBS3}_0=0$.
Moreover, its kernel is $H_{\Abs}(M)$.

\end{thm}
\begin{proof}
    The proof of this Theorem is similar to the one of \thref{thm:TheoremeGeneralNBSF3PourLaPreuve}. 
\end{proof}\begin{remark}
    
By definition, the spectrum of NBS3 is the spectrum of the operator
    \begin{eqnarray*}
   NBS3^{form}: \Omega^p(\partial M) \times \Omega^{p-1}(\partial M) & \longrightarrow & \Omega^p(\partial M) \times \Omega^{p-1}(\partial M)\\
    (\w_3, \w_4) & \longmapsto & (-\nu \lrcorner d\Delta \hat{\w},-\nu \lrcorner \Delta \hat{\w})  \notag 
    \end{eqnarray*} 
    where $\hat{\w}\in H^4(M)$ is the solution of
    \begin{equation}\label{eq:bilapO1surj}
    \left\{
    \hspace{-0.3em}
    \begin{array}{r@{\quad}c@{\quad}l@{\quad}l}
        \Delta^2 \hat{\w} &=& 0 & \text{on } M \\[0.15em]
        \nu \lrcorner \hat{\w} &=& 0 & \text{on } \partial M \\[0.15em]
        \nu \lrcorner d\hat{\w} &=& 0 & \text{on } \partial M \\[0.15em]
        \iota^* \hat{\w} &=& \w_3 & \text{on } \partial M \\[0.15em]
        \iota^* \delta \hat{\w} &=& \w_4 & \text{on } \partial M.
    \end{array}
    \right.
\end{equation}The existence of such a unique $\hat{\omega}$ is a consequence of \lemref{lem:Bilap1}.
\end{remark}\begin{prop}\label{prop:ell_K}
 For all $k\in\mathbb{N}$, we have the following characterization of the $k^{\text{th}}$ positive eigenvalue of NBS3:
    \[
    \ell^{NBS3}_{k,p}=\min_{\substack{\mathcal{V}_k\subset \check{H}^2_{N}(M) \\ \dim (\mathcal{V}_k/H_0^2(M)\cap \mathcal{V}_k)=k}} \ \underset{\w\in \mathcal{V}_k\backslash H_0^2(M)}{\max}\frac{\lVert \Delta \w \rVert^2_{L^2(M)}}{\lVert \iota^*\w\rVert^2_{L^2(\partial M)}+\lVert \iota^*\delta \w\rVert^2_{L^2(\partial M)}}.
    \]
\end{prop}
See \eqref{eq:HN1check} for the definition of $\check{H}^2_{N}(M).$
\begin{proof}
  The proof is identical to that of \thref{theo:lk}.

\end{proof}

\begin{lemma}\label{lem:p=n}
    For $p=0$, the NBS2 and NBS3 problems reduce to the classical NBS problem for scalar functions \eqref{eq:IntroNBS}. For $p=n$, any solution $\w$ for NBS2 must be zero, whereas NBS3 reduces to the scalar DBS \eqref{eq:IntroDBS}. 
\end{lemma}

\begin{proof}
It is immediate that this two NBS problems for $p=0$ reduce to the NBS problem for functions. For $p=n$ we have $\w=f \vol g$ with $f$ a real function. Then in the NBS2 and NBS3 the second and last boundary conditions are trivial. The remaining conditions in the NBS2 give that $\Delta f=0$ on $M$ and $f_{|\partial M}=0$, using an integration by parts, we get that $f=0$ and so $\w=0$. Hence this problem is trivial. In NBS3 we get the scalar DBS conditions. 
\end{proof}
That is why for all the results related to NBS2 we consider $p\in\{1,\cdots,n-1\}.$ We end this section by presenting a Corollary that establishes an ordering between the eigenvalues of the three NBS problems.:
\begin{cor}\label{cor:l<l<l}
   Using the notations defined previously, we have the following inequalities:
    \[
    \ell^{NBS3}_{k,p} \leq \ell_{k,p} \leq \ell^{NBS2}_{k,p}.
    \]
\end{cor}

\begin{proof}
    For all $k\geq 1$, it follows immediately from \thref{theo:lk}, \propref{prop:NBS2VPk} and \propref{prop:ell_K}.
\end{proof}

\section{Spectral inequalities for NBS}\label{sec:kuttler-sigillito}
After having stated the three NBS problems and given the characterization of their eigenvalues, we now present a generalization of the Kuttler-Sigillito inequalities \cite{hassannezhadETsiffert,Kuttler&Sigillito} on differential forms.

\begin{thm}\label{thm:3èmePBkuttlersigilitto1PourLaPreuve}
    Let $(M^n,g)$ be a compact Riemannian manifold with smooth boundary $\partial M$. The following inequality holds between the first non-zero eigenvalues of the Steklov, Neumann, NBS2 and NBS problems on differential forms:
    \begin{equation*}
    \mu_{1,p}\sigma_{1,p}<\ell_{1,p}.
    \end{equation*}
\end{thm}
\begin{remark}
    By \corref{cor:l<l<l}, this implies the following:  $\mu_{1,p}\sigma_{1,p}<\ell^{NBS2}_{1,p}$.
\end{remark}
\begin{proof}[Proof of \thref{thm:3èmePBkuttlersigilitto1PourLaPreuve}]
    Using the following variational characterizations \eqref{eq:vp1AlternativeNeumannAbsolueForme}, \eqref{eq:sigmakformesSteklov}, and \eqref{eq:l1NBS}, we obtain:
    \begin{align*}
    \mu_{1,p} \sigma_{1,p} &=\inf_{\substack{\w \notin H_{\Abs}(M) \\ \nu \lrcorner \w = 0 \\ \nu \lrcorner d\w = 0}}
    \frac{\lVert \Delta \w \rVert^2_{L^2(M)}}{\lVert d\w \rVert^2_{L^2(M)} + \lVert \delta \w \rVert^2_{L^2(M)} }
    \times \inf_{\substack{ \nu \lrcorner \w = 0 \\ \w \perp_{L^2(\partial M)} H_{\Abs}(M) }}
    \frac{\lVert d\w \rVert^2_{L^2(M)} + \lVert \delta \w \rVert^2_{L^2(M)}}{\lVert \w \rVert^2_{L^2(\partial M)}} \\
    &\leq \inf_{\substack{ \nu \lrcorner \w = 0 \\ \nu \lrcorner d\w = 0\\ \w \perp_{L^2(\partial M)}H_{\Abs}(M)}}
    \frac{\lVert \Delta \w \rVert^2_{L^2(M)}}{\lVert d\w \rVert^2_{L^2(M)} + \lVert \delta \w \rVert^2_{L^2(M)} }
    \times\inf_{\substack{ \nu \lrcorner \w = 0 \\ \nu \lrcorner d\w = 0\\ \w \perp_{L^2(\partial M)}H_{\Abs}(M) }}
    \frac{\lVert d\w \rVert^2_{L^2(M)} + \lVert \delta \w \rVert^2_{L^2(M)}}{\lVert \w \rVert^2_{L^2(\partial M)}}
    \\
    &\leq \inf_{\substack{\nu \lrcorner \w = 0 \\ \nu \lrcorner d\w = 0\\ \w \perp_{L^2(\partial M)}H_{\Abs}(M)}}
    \frac{\lVert \Delta \w \rVert^2_{L^2(M)}}{\lVert \w \rVert^2_{L^2(\partial M)}} =\ell_{1,p}.
    \end{align*} We then have the inequality in the broad sense. To show that it is strict, we prove the following Lemma:
    \begin{lemma}\label{lem:mu1sigma1<l1stricte}
  If $\mu_{1,p}\sigma_{1,p}=\ell_{1,p}$, then there exists $\w$ an $\ell_{1,p}$-eigenform for problem \eqref{eq:NBSF3} such that $\w$ is also a $\sigma_{1,p}$-eigenform for Steklov problem \eqref{eq:steklovFormes} and $\w - \projection_{L^2(M)}^{H_{\Abs}(M)}(\w)$ is a $\mu_{1,p}$-eigenform for Neumann problem \eqref{eq:neumannFormesAbsolue}. 
\end{lemma}
We recall that $ \projection_{L^2(M)}^{H_{\Abs}(M)}(\w)$ is the orthogonal projection of $\w$ onto $H_{\Abs}(M)$ in $L^2(M)$.
\begin{proof}[Proof of \lemref{lem:mu1sigma1<l1stricte}]
   Let $\w\in \check{H}^2_{N}(M)$ be an $\ell_{1,p}$-eigenform for NBS $\eqref{eq:NBSF3}$. We know that  

\[
\frac{\lVert d\w \rVert^2_{L^2(M)} + \lVert \delta \w \rVert^2_{L^2(M)}}{\lVert \w \rVert^2_{L^2(\partial M)}}\geq \sigma_{1,p}
\]

and the quotient  

\[
\frac{\lVert \Delta \w \rVert^2_{L^2(M)}}{\lVert d\w \rVert^2_{L^2(M)} + \lVert \delta \w \rVert^2_{L^2(M)} }
\]

is invariant by translation by an element of $H_{\Abs}(M)$. Therefore, for $\Tilde{\w}=\w-\projection_{L^2(M)}^{H_{\Abs}(M)}(\w)$, we have  

\[
\frac{\lVert \Delta \Tilde{\w}\rVert^2_{L^2(M)}}{\lVert d\Tilde{\w} \rVert^2_{L^2(M)} + \lVert \delta \Tilde{\w} \rVert^2_{L^2(M)} }=\frac{\lVert \Delta \w \rVert^2_{L^2(M)}}{\lVert d\w \rVert^2_{L^2(M)} + \lVert \delta \w \rVert^2_{L^2(M)} }\geq\mu_{1,p}
\]

which implies that  

\[
\mu_{1,p}\sigma_{1,p}\leq \frac{\lVert \Delta \w \rVert^2_{L^2(M)}}{\lVert d\w \rVert^2_{L^2(M)} + \lVert \delta \w \rVert^2_{L^2(M)} }\times \frac{\lVert d\w \rVert^2_{L^2(M)} + \lVert \delta \w \rVert^2_{L^2(M)}}{\lVert \w \rVert^2_{L^2(\partial M)}}=\ell_{1,p}=\mu_{1,p}\sigma_{1,p}.
\]

Thus,  

\[
\mu_{1,p}=\frac{\lVert \Delta \Tilde{\w}\rVert^2_{L^2(M)}}{\lVert d\Tilde{\w} \rVert^2_{L^2(M)} + \lVert \delta \Tilde{\w} \rVert^2_{L^2(M)} }=\frac{\lVert \Delta \w \rVert^2_{L^2(M)}}{\lVert d\w \rVert^2_{L^2(M)} + \lVert \delta \w \rVert^2_{L^2(M)} }
\]

and  

\[
\sigma_{1,p}=\frac{\lVert d\w \rVert^2_{L^2(M)} + \lVert \delta \w \rVert^2_{L^2(M)}}{\lVert \w \rVert^2_{L^2(\partial M)}}.
\]

Hence, $\Tilde{\w}$ and $\w$ are solutions of the following problems:  
\begin{gather*}
    \left\{
    \hspace{-0.3em}
    \begin{array}{r@{\quad}c@{\quad}l@{\quad}l}
        \Delta \tilde{\omega} &=& \mu_{1,p} \tilde{\omega} & \text{on } M \\[0.15em]
        \nu \lrcorner \tilde{\omega} &=& 0 & \text{on } \partial M \\[0.15em]
        \nu \lrcorner d\tilde{\omega} &=& 0 & \text{on } \partial M
    \end{array}
    \right.
    \qquad ; \qquad
    \left\{
    \hspace{-0.3em}
    \begin{array}{r@{\quad}c@{\quad}l@{\quad}l}
        \Delta^2 \omega &=& 0 & \text{on } M \\[0.15em]
        \nu \lrcorner \omega &=& 0 & \text{on } \partial M \\[0.15em]
        \nu \lrcorner d\omega &=& 0 & \text{on } \partial M \\[0.15em]
        \nu \lrcorner \Delta\omega &=& 0 & \text{on } \partial M \\[0.15em]
        \nu \lrcorner d\Delta\omega + \ell_{1,p} \iota^* \omega &=& 0 & \text{on } \partial M
    \end{array}
    \right.
    \\[2ex]
    \text{and}\qquad
    \left\{
    \hspace{-0.3em}
    \begin{array}{r@{\quad}c@{\quad}l@{\quad}l}
        \Delta \omega &=& 0 & \text{on } M \\[0.15em]
        \nu \lrcorner d\omega &=& -\sigma_{1,p} \iota^* \omega & \text{on } \partial M \\[0.15em]
        \nu \lrcorner \omega &=& 0 & \text{on } \partial M.
    \end{array}
    \right.
\end{gather*}

\end{proof}

Back to the proof of  \thref{thm:3èmePBkuttlersigilitto1PourLaPreuve}. Since $\nu\lrcorner d\w=0$, it follows that $\iota^*\w$ is zero due to the conditions of the Steklov problem. Moreover, as $\nu\lrcorner\w=0$, we get that $\w_{|\partial M}=0$, which implies together with $\Delta\w=0$ that $\w=0$ by \cite{ColetteAnne}. This is a contradiction, because $\w$ is an eigenform and cannot be zero. Consequently, the inequality is strict. 

\end{proof}

The following Theorem is a generalization of  \thref{thm:3èmePBkuttlersigilitto1PourLaPreuve} that is given in \cite[Theorem 1.1]{hassannezhadETsiffert} and \cite[Inequality I]{Kuttler&Sigillito}.
\begin{thm}\label{thm:ks2}
    Let $(M^n,g)$ be a compact Riemannian manifold with a smooth boundary $\partial M.$ The following inequalities hold for $k\geq 2$, between the eigenvalues of the Steklov, Neumann, NBS2 and NBS problems for differential forms:\begin{align}
&\bullet \quad \mu_{k,p} \sigma_{1,p} \leq \ell_{k,p} \label{eq:3èmePBmuksigma1<lkPourLaPreuve} \\[0.5em]
&\bullet \quad \mu_{1,p} \sigma_{k,p} \leq \ell_{k,p}. \label{eq:3èmePBmu1sigmak<lkPourLaPreuve}
\end{align}
\end{thm}
\begin{remark}
    By \corref{cor:l<l<l}, this implies the following: $\mu_{k,p} \sigma_{1,p} <\ell^{NBS2}_{k,p}$ and $\mu_{1,p} \sigma_{k,p} < \ell^{NBS2}_{k,p}.$\end{remark}
\begin{proof}[Proof of \thref{thm:ks2}]
    We begin by proving \eqref{eq:3èmePBmuksigma1<lkPourLaPreuve}. Let $\mathcal{V}=\spann(\w_1,\cdots,\w_k)$, where the $\w_i$ are $\ell_{i,p}$-eigenforms, note that $\mathcal{V}$ has dimension $k$. Let $\Tilde{\mathcal{V}}=H_{\Abs}(M)\oplus \mathcal{V}$ such that $\dim\Tilde{\mathcal{V}}=k+\dim H_{\Abs}(M)$. For all $\w\in \mathcal{V}$, we have $\w\perp_{L^2(M)} H_{\Abs}(M)$. Since $\w$ satisfies the boundary conditions of the Neumann problem, we can use the characterization \eqref{eq:vp1AlternativeNeumannAbsolueForme1}, and deduce that
\begin{align*}
\mu_{k,p} & \leq \sup_{\substack{\w \in \Tilde{\mathcal{V}}}} \frac{\lVert \Delta \w \rVert^2_{L^2(M)}}{\lVert d \w \rVert^2_{L^2(M)} + \lVert \delta \w \rVert^2_{L^2(M)}} \\
& = \sup_{\substack{\w \in \mathcal{V}}} \frac{\lVert \Delta \w \rVert^2_{L^2(M)}}{\lVert d \w \rVert^2_{L^2(M)} + \lVert \delta \w \rVert^2_{L^2(M)}} \times \frac{\lVert \w \rVert^2_{L^2(\partial M)}}{\lVert \w \rVert^2_{L^2(\partial M)}} \\
& \leq \ell_{k,p} \sup_{\substack{\w \in \mathcal{V}}} \frac{\lVert \w \rVert^2_{L^2(\partial M)}}{\lVert d \w \rVert^2_{L^2(M)} + \lVert \delta \w \rVert^2_{L^2(M)}} \\
& = \ell_{k,p}\left(\inf_{\substack{\w \in \mathcal{V}}} \frac{\lVert d \w \rVert^2_{L^2(M)} + \lVert \delta \w \rVert^2_{L^2(M)}}{\lVert \w \rVert^2_{L^2(\partial M)}}\right)^{-1} \\
& \leq \frac{\ell_{k,p}}{\sigma_{1,p}}.
\end{align*}
Let us next show \eqref{eq:3èmePBmu1sigmak<lkPourLaPreuve}. Using the same definitions above of $\mathcal{V}$ and $\Tilde{\mathcal{V}}$, we have \begin{align*}
\sigma_{k,p} & \leq \sup_{\substack{\w \in \mathcal{V}}} \frac{\lVert d \w \rVert^2_{L^2(M)} + \lVert \delta \w \rVert^2_{L^2(M)}}{\lVert \w \rVert^2_{L^2(\partial M)}} \\
& = \sup_{\substack{{\Tilde{\w}} \in \Tilde{\mathcal{V}} \\ \Tilde{\w}_{|\partial M} \neq 0}} \frac{\lVert d \Tilde{\w} \rVert^2_{L^2(M)} + \lVert \delta \Tilde{\w} \rVert^2_{L^2(M)}}{\lVert \iota^* \tilde{\w} \rVert^2_{L^2(\partial M)}} \\
& \leq \sup_{\substack{\w \in \mathcal{V} \\ \w_{|\partial M} \neq 0}} \frac{\lVert d \w \rVert^2_{L^2(M)} + \lVert \delta \w \rVert^2_{L^2(M)}}{\lVert \iota^* \w \rVert^2_{L^2(\partial M)}} \times \frac{\lVert \Delta \w \rVert^2_{L^2(M)}}{\lVert \Delta \w \rVert^2_{L^2(M)}} \\
& \leq \ell_{k,p} \sup_{\substack{\w \in \mathcal{V}}} \frac{\lVert d \w \rVert^2_{L^2(M)} + \lVert \delta \w \rVert^2_{L^2(M)}}{\lVert \Delta \w \rVert^2_{L^2(M)}} \\
& = \ell_{k,p} \left(\inf_{\substack{\w \in \mathcal{V}}} \frac{\lVert \Delta \w \rVert^2_{L^2(M)}}{\lVert d \w \rVert^2_{L^2(M)} + \lVert \delta \w \rVert^2_{L^2(M)}}\right)^{-1} \\
& \leq \frac{\ell_{k,p}}{\mu_{1,p}}.
\end{align*}
\end{proof}

\begin{remark}
    Unlike the case where $ k=1 $, the generalized inequalities may not be strict because having
    \[
    \ell_{k,p}=\underset{\w\in \mathcal{V}_k/ \mathcal{V}_k\cap H_0^2(M)}{\sup}\ \frac{\lVert \Delta \w \rVert^2_{L^2(M)}}{\lVert \iota^*\w\rVert^2_{L^2(\partial M)}}
    \]
    for $\mathcal{V}_k \subset \check{H}^2_{N,1}(M)$ and $\dim(\mathcal{V}_k/\mathcal{V}_k \cap H_0^2)=k$, does not imply that $\mathcal{V}=\spann\{\w_1,\cdots,\w_k\}$.
\end{remark}

The following Theorem is a generalization of \cite[Inequality II]{Kuttler&Sigillito}.
\begin{thm}\label{thm:proofQSigma<l}
    Let $(M^n, g)$ be a compact Riemannian manifold with smooth boundary $\partial M$. The following inequality holds between the eigenvalues of the Steklov, DBS, NBS2 and NBS problems for differential forms:
    \begin{equation}\label{eq:q1sigma1<l1PourLaPreuve}
    q_{1,p} \sigma_{1,p}^2 <\ell_{1,p}.
    \end{equation}
\end{thm}
\begin{remark}By \corref{cor:l<l<l}, this implies the following: $q_{1,p} \sigma_{1,p}^2<\ell^{NBS2}_{1,p}.$\end{remark}

\begin{proof}[Proof of \thref{thm:proofQSigma<l}]
 Using the characterizations \eqref{eq:sigmakformesSteklov} and \eqref{eq:sigmakformesSteklov2} of the classical Steklov problem, we have:

\[
\sigma^2_{1,p} \leq \inf_{\substack{\nu \lrcorner \w = 0 \\  \w\perp_{L^2(\partial M)}H_{\Abs}(M)}}
\frac{\lVert d\w \rVert^2_{L^2(M)} + \lVert \delta\w \rVert^2_{L^2(M)}}{\lVert \iota^*\w \rVert^2_{L^2(\partial M)}} \times \inf_{\substack{\nu \lrcorner \w = 0 \\ \Delta\w = 0}}
\frac{\lVert \nu\lrcorner d\w \rVert^2_{L^2(\partial M)}}{\lVert d\w \rVert^2_{L^2(M)} + \lVert \delta\w \rVert^2_{L^2(M)}} 
\]

\[
\leq \inf_{\substack{\nu \lrcorner \w = 0 \\ \Delta\w = 0 \\ \w\perp_{L^2(\partial M)} H_{\Abs}(M)}}
\frac{\lVert \nu\lrcorner d\w \rVert^2_{L^2(\partial M)}}{\lVert \iota^*\w \rVert^2_{L^2(\partial M)}}.
\]

Multiplying by $q_{1,p}$ and using \propref{prop:check<l} together with $\tilde{\ell}_{1,p}$ and $\check{\ell}_{1,p}$ given in \eqref{eq:tilde=l} and \eqref{eq:lTilde<l} respectively, we obtain:

\[
q_{1,p}\sigma_{1,p}^2 \leq \inf_{\substack{\Delta\w = 0 }}
\frac{\lVert \iota^*\w \rVert^2_{L^2(\partial M)}}{\lVert \w \rVert^2_{L^2(M)}} 
\times \inf_{\substack{\nu \lrcorner \w = 0 \\ \Delta\w = 0 \\ \w\perp_{L^2(\partial M)} H_{\Abs}(M)}}
\frac{\lVert \nu\lrcorner d\w \rVert^2_{L^2(\partial M)}}{\lVert \iota^*\w \rVert^2_{L^2(\partial M)}} 
\]

\[
\leq \inf_{\substack{\nu \lrcorner \w = 0 \\ \Delta\w = 0 \\ \w\perp_{L^2(\partial M)} H_{\Abs}(M)}}
\frac{\lVert \nu\lrcorner d\w \rVert^2_{L^2(\partial M)}}{\lVert \w \rVert^2_{L^2(M)}} \leq \check{\ell}_{1,p}\leq \tilde{\ell}_{1,p}=\ell_{1,p}.
\]
Let us next show that the inequality is strict. We first show the following Lemma: 
   \begin{lemma}
    If $ q_{1,p}\sigma_{1,p}^2 = \ell_{1,p} $, then there exists $ \w \in \Omega^p(M) $ eigenform associated to  both $ \ell_{1,p} $ and $ \sigma_{1,p} $.
\end{lemma}

\begin{proof}
 Let $\w$ be an $\ell_{1,p}$-eigenform for NBS. As in the proof of \lemref{lem:mu1sigma1<l1stricte}, we show that $ \w $ is also a $\sigma_{1,p}$-eigenform. Hence, it satisfies the following problems:
\[
    \left\{
    \hspace{-0.3em}
    \begin{array}{r@{\quad}c@{\quad}l@{\quad}l}
        \Delta^2 \w &=& 0 & \text{on } M \\[0.15em]
        \nu \lrcorner \w &=& 0 & \text{on } \partial M \\[0.15em]
        \nu \lrcorner d \w &=& 0 & \text{on } \partial M \\[0.15em]
        \nu \lrcorner \Delta \w &=& 0 & \text{on } \partial M \\[0.15em]
        \nu \lrcorner d\Delta \w + \ell_{1,p} \, \iota^* \w &=& 0 & \text{on } \partial M
    \end{array}
    \right.
    \quad; \quad
    \left\{
    \hspace{-0.3em}
    \begin{array}{r@{\quad}c@{\quad}l@{\quad}l}
        \Delta \w &=& 0 & \text{on } M \\[0.15em]
        \nu \lrcorner d \w &=& -\sigma_{1,p} \iota^* \w & \text{on } \partial M \\[0.15em]
        \nu \lrcorner \w &=& 0 & \text{on } \partial M.
    \end{array}
    \right.
\]
    \end{proof}
Let us finish the proof of  \thref{thm:proofQSigma<l}. In this case, $ \nu\lrcorner \w = 0 $ and $ \iota^*\w = 0 $, so $ \w_{|\partial M} = 0 $. Moreover, since $ \Delta\w = 0 $, we have $ \w = 0 $ on $ M $ according to \cite{ColetteAnne}. This leads to a contradiction, and the inequality is strict. 
\end{proof}
The following is a generalization of \cite[Inequality III]{Kuttler&Sigillito}.
\begin{thm}
    Let $(M^n, g)$ be a compact Riemannian manifold with smooth boundary $\partial M$. The following inequality holds between the eigenvalues of the Dirichlet, Neumann, DBS and NBS problems for differential forms:
    \begin{equation}\label{eq:mu1<lamba1+(q1l1)PourLaPreuve}
    \mu_{1,p}^{-1} < \lambda_{1,p}^{-1} + (q_{1,p} \ell_{1,p})^{-\frac{1}{2}}.
    \end{equation}
\end{thm}
\begin{proof}
   Let $ \w_1 \in \Omega^p(M) $ such that $ \nu\lrcorner d\w_1 = 0 $, $ \nu\lrcorner \w_1 = 0 $, and $ \w_1 \perp_{L^2(\partial M)} H_{\Abs}(M) $. Let $v\in \Omega^p(M)$ satisfy the following conditions:
\begin{equation*}
    \left\{
    \hspace{-0.3em}
    \begin{array}{r@{\quad}c@{\quad}l@{\quad}l}
        \Delta v &=& \Delta\w_1 & \text{on } M \\[0.5em]
        v_{|\partial M} &=& 0 & 
    \end{array}
    \right.
\end{equation*}and which has a solution due to  \thref{theo:thm5.4}. Let $u\in \Omega^p(M)$ such that $u=\w_1-v$, so we have $\Delta u=\Delta\w_1-\Delta v=0$. So we decompose $ \w_1 = u + v $ such that $ \Delta u = 0 $ and $ v_{|\partial M} = 0 $.
On the one hand, we have
\begin{equation*}
    \lVert u \rVert_{L^2( M)} \leq (q_{1,p}^{-1})^{\frac{1}{2}} \lVert u \rVert_{L^2(\partial M)}.
\end{equation*} Since $v_{|\partial M}=0,$ we get \begin{equation}\label{eq:q1u<q1w}
    \lVert u \rVert_{L^2( M)} \leq (q_{1,p}^{-1})^{\frac{1}{2}} \lVert \w_1 \rVert_{L^2(\partial M)}.
\end{equation}
On the other hand, by multiplying \eqref{eq:vpKDirichletForme} and \eqref{eq:vpKAlternativeDirichletForme} we get $\lVert v \rVert_{L^2(M)}\leq {\lambda_{1,p}^{-2}}\lVert \Delta v \rVert_{L^2(M)}.$ Now $\Delta v=\Delta\w_1$ as $\Delta u=0,$ so that \begin{equation}\label{eq:v<lambdaDeltaw}
    \lVert v \rVert_{L^2(M)}\leq {\lambda_{1,p}^{-1}}\lVert \Delta \w_1 \rVert_{L^2(M)}.
\end{equation}
Thus, using \eqref{eq:q1u<q1w} and \eqref{eq:v<lambdaDeltaw} we get
\begin{align}
    \lVert \w_1 \rVert_{L^2(M)} \nonumber &\leq \lVert u \rVert_{L^2(M)} + \lVert v \rVert_{L^2(M)} \\ \nonumber
    &\leq (q_{1,p}^{-1})^{\frac{1}{2}} \lVert \w_1 \rVert_{L^2(\partial M)} + (\lambda_{1,p}^{-1}) \lVert \Delta \w_1 \rVert_{L^2(M)} \\ \nonumber
    &\leq (q_{1,p}^{-1})^{\frac{1}{2}} (\ell_{1,p}^{-1})^{\frac{1}{2}} \lVert \Delta \w_1 \rVert_{L^2(M)} + \lambda_{1,p}^{-1} \lVert \Delta \w_1 \rVert_{L^2(M)} \\
    &\label{eq:preuveIneg} \leq (\lambda_{1,p}^{-1} + q_{1,p}^{-\frac{1}{2}} \ell_{1,p}^{-\frac{1}{2}})\lVert \Delta \w_1\rVert_{L^2(M)}. 
\end{align}
Now let $ \w $ be a $\mu_{1,p}$-eigenform for the Neumann problem \eqref{eq:neumannFormesAbsolue}, so we have $\lVert \Delta\w\rVert_{L^2(M)}=\mu_1\lVert \w\rVert_{L^2(M)}$. We decompose $\w$ as follows, $\w=\w_0+\w_1$ such that $\w_0\in H_{\Abs}(M)$ and $\w_1\perp_{L^2(\partial M)}\w_0$, then $\Delta\w=\Delta\w_1$. We perform an integration by parts \eqref{eq:IPP3'}, we obtain $(\Delta\w_1,\w_0)_{L^2(M)}=0.$ As we said before, $\Delta\w=\Delta\w_1,$ then $\Delta\w_1=\mu_{1,p}(\w_1+\w_0),$ so \[0=(\Delta\w_1,\w_0)_{L^2(M)}=\mu_{1,p}(\w_1,\w_0)_{L^2(M)}+\mu_{1,p}\lVert\w_0\rVert^2_{L^2(M)}.\] Hence we get $(\w_1,\w_0)_{L^2(M)}=-\lVert\w_0\rVert^2_{L^2(M)}.$ Using the calculations above, we have \begin{align}
    \lVert \w \rVert_{L^2(M)} \nonumber &= \lVert \w_0 \rVert_{L^2(M)} + \lVert \w_1 \rVert_{L^2(M)} + 2(\w_0, \w_1)_{L^2(M)} \\ \nonumber
    &= \lVert \w_0 \rVert_{L^2(M)} + \lVert \w_1\rVert_{L^2(M)} - 2 \lVert \w_0 \rVert_{L^2(M)} \\ \nonumber
    &= \lVert \w_1 \rVert_{L^2(M)} - \lVert \w_0 \rVert_{L^2(M)}\\ 
    &\label{eq:preuveIneg2} \leq \lVert\w_1\rVert_{L^2(M)}.
\end{align}
Going back to \eqref{eq:preuveIneg} and using \eqref{eq:preuveIneg2} and that $\Delta\w=\Delta\w_1$, we get \[\lVert \w \rVert_{L^2(M)}\leq \lVert \w_1 \rVert_{L^2(M)}\leq \lVert \Delta \w_1 \rVert_{L^2(M)}[(q_{1,p} \ell_{1,p})^{-\frac{1}{2}} + \lambda_{1,p}^{-1}] =\lVert \Delta \w \rVert_{L^2(M)}[(q_{1,p} \ell_{1,p})^{-\frac{1}{2}} + \lambda_{1,p}^{-1}] .\]
Which implies finally that \[
\mu_{1,p}^{-1} \leq \lambda_{1,p}^{-1} + (q_{1,p} \ell_{1,p})^{-\frac{1}{2}}
\] since $\lVert\Delta\w\rVert_{L^2(M)}=\mu_{1,p}\lVert\w\rVert_{L^2(M)}$. Let us now show that the inequality is strict. 
Then
\[
\lVert \w \rVert^2_{L^2(M)} = \left(\lVert u \rVert_{L^2(M)} + \lVert v \rVert_{L^2(M)} \right)^2,
\]
so
\[
\lVert u \rVert^2_{L^2(M)} + \lVert v \rVert^2_{L^2(M)} + 2(u, v)_{L^2(M)} = \lVert u \rVert^2_{L^2(M)} + \lVert v \rVert^2_{L^2(M)} + 2 \lVert u \rVert_{L^2(M)} \lVert v \rVert_{L^2(M)},
\]
thus $ u $ and $ v $ are collinear, either $ v = ku $ or $u=kv$, where $ k \in \mathbb{R} $. However, $ \w $ is an $ \mu_{1,p} $-eigenform, so it satisfies the following conditions:
\begin{equation*}
    \left\{
    \hspace{-0.3em}
    \begin{array}{r@{\quad}c@{\quad}l@{\quad}l}
        \Delta \w &=& \mu_1 \w & \text{on } M \\[0.15em]
        \nu \lrcorner \w &=& 0 & \text{on } \partial M \\[0.15em]
        \nu \lrcorner d\w &=& 0 & \text{on } \partial M
    \end{array}
    \right.
\end{equation*}
thus $ k \Delta v = \Delta u = 0 $, contradiction. When $u=kv$ there are two cases: the first case is if $ k = 0 $, then $ u = 0 $, so $ \w_{|\partial M} = 0 $ and $ \w $ is then associated to $ \lambda_1 $, which implies $ \mu_{1,p} = \lambda_{1,p} $, hence {$0<({\ell_{1,p} q_{1,p}})^{-1} = 0 $}, a contradiction. The second case is if $ k \neq 0 $, then $ \Delta v = 0 $, so $ \Delta \w = 0 $, and hence $ \w = 0 $, a contradiction. Therefore, the inequality \eqref{eq:mu1<lamba1+(q1l1)PourLaPreuve} is strict. 
\end{proof}

The following is a generalization of \cite[Inequality IV]{Kuttler&Sigillito}.
\begin{thm}\label{thm:dernierInegaliteIntroPourLaPreuve}
    Let $(M^n, g)$ be a compact Riemannian manifold with smooth boundary $\partial M$. The following inequality holds between the eigenvalues of the Dirichlet, Neumann, Steklov, and DBS problems for differential forms:
    \begin{equation}\label{eq:mu1<lamba1+(q1sigma1)PourLaPreuve}
    \mu_{1,p}^{-1} < \lambda_{1,p}^{-1} + (q_{1,p} \sigma_{1,p})^{-1}.
    \end{equation}
\end{thm}

\begin{proof}
   Using the inequality \eqref{eq:q1sigma1<l1PourLaPreuve}, we have 
\[
\sigma_{1,p} \leq \ell_{1,p}^{\frac{1}{2}} q_{1,p}^{-\frac{1}{2}}
\]
and by multiplying then by $ q_{1,p} $, we obtain 
\[
q_{1,p} \sigma_{1,p} \leq \ell_{1,p}^{\frac{1}{2}} q_{1,p}^{\frac{1}{2}}.
\]
Finally, by inserting the previous inequality into \eqref{eq:mu1<lamba1+(q1l1)PourLaPreuve}, we get
\[
\mu_{1,p}^{-1} \leq \lambda_{1,p}^{-1} + (q_{1,p} \sigma_{1,p})^{-1}.
\] Since the two inequalities \eqref{eq:q1sigma1<l1PourLaPreuve} and \eqref{eq:mu1<lamba1+(q1l1)PourLaPreuve} are strict, the inequality \eqref{eq:mu1<lamba1+(q1sigma1)PourLaPreuve} is also strict.

\end{proof}

\bibliographystyle{abbrv}
\bibliography{references}

\end{document}